\documentclass[11pt,reqno]{amsart}
\usepackage{xifthen,setspace}
\usepackage{bbm}
\usepackage{graphicx}
\usepackage{subfig}
\usepackage{pkgfile}
\usepackage{enumitem} 
\newlist{steps}{enumerate}{1}
\setlist[steps, 1]{label = Step \arabic*:}

\newtheorem{definition}{Definition}

\renewcommand{\leq}{\leqslant}
\renewcommand{\geq}{\geqslant}

\newcommand{\sa}{{\mathcal{C}}}
\newcommand{\p}{{\mathbb{P}}}
\newcommand{\e}{{\mathbb{E}}}
\newcommand{\bs}{{\boldsymbol{\sigma}}}
\newcommand{\bt}{{\boldsymbol{\tau}}}

\newcommand{\os}{{\overline{\sigma}_N}}
\newcommand{\cp}{{\mathscr{C}_p}}

\newcommand{\ma}{{\mathcal{A}}}

\usepackage{tikz}
\usepackage{xcolor}

\newcommand{\mysquare}[1][black]{\small\textcolor{#1}{\ensuremath\blacksquare}}

\allowdisplaybreaks

\title[ML Estimation in the $p$-Spin Curie-Weiss Model]{Phase Transitions of the Maximum Likelihood Estimates in the $p$-Spin Curie-Weiss Model}

\author[Mukherjee]{Somabha Mukherjee$^\ast$}\thanks{$^\ast$The first two authors contributed equally to the paper.}
\address{Department of Statistics and Data Science, National Institute of Singapore, Singapore} 
\email{somabha@nus.edu.sg}

\author[Son]{Jaesung Son$^\ast$}
\address{Department of Statistics, Columbia University, New York, USA} 
\email{js4638@columbia.edu}

\author[Bhattacharya]{Bhaswar B. Bhattacharya}
\address{Department of Statistics and Data Science, University of Pennsylvania, Philadelphia, USA} 
\email{bhaswar@wharton.upenn.edu}

\begin{document}
	
	
	
	\begin{abstract} 
		In this paper we consider the problem of parameter estimation in the  $p$-spin Curie-Weiss model, for $p \geq 3$. We provide a complete description  of the limiting properties of the maximum likelihood (ML) estimates of the inverse temperature and the magnetic field given a single realization from the $p$-spin Curie-Weiss model, complementing the well-known results in the 2-spin case \cite{comets}. Our results unearth various new phase transitions and surprising limit theorems, such as the existence of a  `critical' curve in the parameter space, where the limiting distribution of the ML estimates is a mixture with both continuous and discrete components. The number of mixture components is either two or three,  depending on, among other things, the sign of one of the parameters and the parity of $p$. Another interesting revelation is the existence of certain `special' points in the parameter space where the ML estimates exhibit a superefficiency phenomenon, converging to a non-Gaussian limiting distribution at rate $N^{\frac{3}{4}}$. Using these results we can obtain  asymptotically valid confidence intervals for the inverse temperature and the magnetic field at all points in the parameter space where consistent estimation is possible. 
	\end{abstract}

	\maketitle

	\section{Introduction}\label{intro}

	The Ising model is a discrete random field, where the Hamiltonian has a quadratic term designed to capture pairwise interactions between neighboring vertices of a network. This model was initially studied more than a century ago as a model of ferromagnetism \cite{ising}, and has since then emerged as one of the fundamental mathematical tools for understanding interacting  spin systems on graphs.  Recently, the Ising model has also turned out to be an useful primitive for capturing pairwise dependence among binary attributes with an underlying network structure, which arise naturally in spatial statistics, social networks, computer vision, neural networks, and computational biology, among others (cf.~\cite{spatial,cd_trees,geman_graffinge,disease,neural,innovations} and the references therein).  However, in many situations, both in modeling interacting spin systems and in real-world network data, dependencies arise not just from pairs, but from interactions between groups of particles or individuals. This leads to the study of $p$-spin Ising models, where the Hamiltonian is a multilinear polynomial of degree $p\geq 2$, designed for capturing higher-order interactions between the different particles. As in the case of 2-spin models, the $p$-spin Ising model can be represented as a spin system on a $p$-uniform hypergraph, where the individual entities represent the vertices of the hypergraph and the $p$-tuples of interactions are indexed by the hyperedges.  Higher-order Ising models arise naturally in the study of multi-atom interactions in lattice gas models, which includes, among others, the square-lattice eight-vertex model, the Ashkin-Teller model, and Suzuki's pseudo-3D anisotropic model  (cf.~  \cite{ab_ferromagnetic_pspin,multispin_simulations,pspinref1,ising_suzuki,ising_general,ferromagnetic_mean_field,turban,pspinref2} and the references therein).  More recently, higher-order spin systems  have been proposed as effective and mathematically tractable models for simultaneously capturing both peer-group effects and individual effects in social networks \cite{cd_ising_II}. 
	

	Estimating the parameters of a spin system and understanding the asymptotic properties of the resulting estimates, have become increasingly important due the ubiquity of dependent network data in computer science, physics, probability, and  statistics. In this paper, we consider the problem of estimating the parameters of a $p$-spin Ising model given a single spin realization from the model. This problem has been extensively studied for the $p=2$ (2-spin) case, which includes, among others, the classical results on consistency and optimality of the maximum likelihood (ML) estimates for lattice models \cite{comets_exp,gidas,guyon,discrete_mrf_pickard}, and the seminal paper of Chatterjee \cite{chatterjee}, where general conditions for  $N^{\frac{1}{2}}$-consistency of the maximum pseudolikelihood estimate  (MPLE) \cite{besag_lattice,besag_nl} were derived. The results in \cite{chatterjee} were later extended in \cite{BM16} and \cite{pg_sm} to obtain rates of estimation for 2-spin Ising models on general weighted graphs and joint estimation of parameters, respectively. These techniques were recently  used in Daskalakis et al. \cite{cd_ising_I,cd_ising_II} to obtain rates of convergence of the MPLE in general logistic regression models with dependent observations. Very recently, Dagan et al. \cite{cd_ising_estimation} considered the problem of parameter estimation in a more general model where the binary outcomes can be influenced by various underlying networks, and, as a consequence, improved some of the results in \cite{BM16}.  Related problems in hypothesis testing given a single sample from the 2-spin Ising model are studied in \cite{gb_testing,ising_testing}. 
	However, none of these results  say anything about the limiting distribution of the estimates, and hence,  cannot be used for inferential tasks, such as constructing confidence intervals and hypothesis testing. In fact, proving general limit theorems in these models is often extremely difficult, if not impossible, because of the presence of an unknown normalizing constant (partition function) in the estimation objective function, which is both computationally and mathematically intractable for Ising models on general graphs. As a consequence, it is natural to assume certain special structures on the underlying network interactions if one desires to obtain precise results such as central limit theorems. A particularly useful structural condition which preserves several interesting properties of general systems, is to assume that all pairwise interactions between the nodes of the network are present. This is the Ising model on the complete graph, more commonly known as the 2-spin Curie-Weiss model \cite{dm_gibbs_notes,cw_stein_nonnormal,ellis_book,ellis,glauber_dynamics}, which has been extensively studied in probability and statistical physics, and provides the foundations of our understanding of mean-field models with pairwise interactions. In particular, Comets and Gidas \cite{comets} provided a complete description of the limiting distribution of the ML estimates of the parameters in the 2-spin Curie-Weiss model.

	The 2-spin Curie-Weiss model naturally extends to the $p$-spin Curie-Weiss model, for any $p \geq 2$, in which the Hamiltonian has all the possible $p$-tuples of interactions (the Ising model on a complete hypergraph). More precisely, given an inverse temperature $\beta \geq 0$ and a magnetic field $h \in \R$, the $p$-spin Curie-Weiss model is a spin system on $\sa_N := \{-1,1\}^N$, defined as:
	\begin{align}\label{cwwithmg}
		\mathbb{P}_{\beta,h,p}(\bs)  = \frac{ \exp \left\{ \frac{\beta}{N^{p-1}} \sum_{1 \leq i_1, i_2, \ldots, i_p \leq N} \sigma_{i_1} \sigma_{i_2} \cdots \sigma_{i_p} + h \sum_{i=1}^N \sigma_i \right \} }{2^{N} Z_N(\beta,h,p)} ~,
	\end{align}
	for $\bs :=(\sigma_1,\ldots,\sigma_N) \in \sa_N$.
	The normalizing constant, also referred to as the partition function, $Z_N(\beta,h,p)$ is determined by the condition $\sum_{\bs \in \sa_N}\mathbb{P}_{\beta,h,p}(\bs)
	\newblock=1$, that is, 
	\begin{align}\label{ptnepn}
		Z_N (\beta,h,p)  = \frac{1}{2^{N}} \sum_{\bs \in \sa_N}  \exp \left\{ \frac{\beta}{N^{p-1}} \sum_{1 \leq i_1, i_2, \ldots, i_p \leq N} \sigma_{i_1} \sigma_{i_2} \cdots \sigma_{i_p} + h \sum_{i=1}^N \sigma_i \right \} .
	\end{align}
	Denote by $F_N(\beta,h,p) := \log Z_N(\beta,h,p)$ the log-partition function of the model. Hereafter, we will often abbreviate $\mathbb{P}_{\beta,h,p}, Z_N(\beta,h,p)$, and $F_N(\beta,h,p)$, by $\mathbb{P}, Z_N$, and $F_N$, respectively, when there is no scope of confusion. Various thermodynamic properties of this model, which is alternatively referred to as the fully connected $p$-spin model or the ferromagnetic $p$-spin model, are studied in \cite{ab_ferromagnetic_pspin,pspinref1,ferromagnetic_mean_field,pspinref2}.

	In this paper we consider the problem of estimating the parameters $\beta$ and $h$ given a single sample $\bm \sigma \sim \P_{\beta, h, p}$ from the $p$-spin Curie-Weiss model. It is well-known, since the model \eqref{cwwithmg} has only one sufficient statistic (the average magnetization $\os = \frac{1}{N} \sum_{i=1}^n \sigma_i$), that joint estimation of the parameters $(\beta, h)$ in this model is, in general, impossible. This motivates the study of individual (marginal) estimation, that is, estimating $h$ when $\beta$ is assumed to be known and  estimating $\beta$ when $h$ is assumed to be known. As mentioned before, for the 2-spin Curie-Weiss model, this problem was studied in \cite{comets},  where the limiting properties of the individual ML estimates were derived. Here, we consider the analogous problem for the $p$-spin Curie-Weiss model, for $p \geq 3$. In particular, we derive precise limit theorems for the individual ML estimates of $\beta$ and $h$, hereafter, denoted by $\hat{\beta}_N$ and $\hat{h}_N$,  at all the parameter points. In addition to providing a complete description of the asymptotic properties of the ML estimates, our results unearth several remarkable new phenomenon, which we briefly summarize below. 
	
	\begin{itemize}
		
		\item We identify a region of `regular' points in the parameter space, where the ML estimates $\hat{\beta}_N$ and $\hat{h}_N$ are $N^{\frac{1}{2}}$-consistent\footnote{A sequence of estimators $\{\theta_n\}_{n \geq 1}$ is said to be $a_n$-consistent for the parameter $\theta$, if $a_n(\hat \theta_n-\theta)$ is bounded in probability.} and asymptotically normal (Theorem \ref{cltintr3_1} and Theorem \ref{thmmle1}). Here, the limiting variance equals the limiting inverse Fisher information, which implies that the ML estimates are asymptotically efficient at these points (Remark \ref{remark1}). The variance of the limiting normal distribution can be easily estimated as well, hence, this result also provides a way to construct asymptotically valid confidence intervals for the model parameters (Section \ref{sec:applications}).

		\item More interestingly, there are certain `critical' points, which form a 1-dimensional curve in the parameter space, where $\hat{\beta}_N$ and $\hat{h}_N$ are still $N^{\frac{1}{2}}$-consistent, but the limiting distribution is a mixture with both continuous and discrete components. The number of mixture components is either two or three, depending on, among other things, the sign of one of the parameters and the parity of $p$. In particular, at the points where the critical curve intersects the region $h\ne 0$, the scaled ML estimates $N^{\frac{1}{2}}(\hat{\beta}_N - \beta )$ and $N^{\frac{1}{2}}(\hat{h}_N-h)$ have a surprising  three component mixture distribution, where two of the components are folded (half) normal distributions and the other is a point mass at zero (Theorem \ref{cltintr3_III} and Theorem \ref{thmmle_III}). This new phenomenon, which is absent in the 2-spin case, is an example of the many intricacies of models with higher-order interactions. 
		
		\item Finally, there are one or two `special' points in the parameter space, depending on whether $p \geq 3$ is odd or even, respectively, where both the individual ML estimates are superefficient, with fluctuations of order $N^{\frac{3}{4}}$ and non-Gaussian limiting distributions (Theorem \ref{cltintr3_2} and Theorem \ref{thmmle2}). 
	\end{itemize} 
	Our results also reveal various other interesting phenomena, such as, inconsistency of $\hat{\beta}_N$ in a region of the parameter space, and an additional (strongly) critical point, where $\hat{h}_N$ is $N^{\frac{1}{2}}$-consistent, but $\hat{\beta}_N$ is not.   These results, which are formally stated in Section \ref{sec:samplemean_mle}, together provide a complete characterization of the limiting properties of the ML estimates in the $p$-spin Curie-Weiss model.

	\subsection{Related Work on Structure Learning}
	
	Another related area of active research is the problem of structure learning in Ising models and Markov Random Fields. Here, one is given access to {\it multiple} i.i.d. samples from an Ising model, or a more general graphical model, and the goal is to estimate the underlying graph structure. Efficient algorithms and statistical lower bounds for this problem has been developed over the years under various structural assumptions on the underlying graph (cf.~\cite{structure_learning,discrete_tree,highdim_ising,graphical_models_binary,ising_nonconcave} and the references therein). Bresler \cite{bresler} made the first breakthrough for general bounded degree graphs, giving an efficient algorithm for structure learning, which required only logarithmic samples in the number of nodes of the graph.  This result has been subsequently generalized to Markov-random fields with higher-order interactions and alphabets with more than two elements (cf. \cite{graphical_models_algorithmic,multiplicative} and the references therein). The related problems of goodness-of-fit and independence testing given multiple samples from an Ising model has been studied in Daskalakis et al. \cite{cd_testing}. Recently, Neykov and Liu \cite{neykovliu_property} and Cao et al. \cite{high_tempferro} considered the problem of testing graph properties, such as connectivity and presence of cycles or cliques, using multiple samples from the Ising model on the underlying graph.

	All these results, however, are in contrast with the present work, where the underlying graph structure is assumed to be known and the goal is to  estimate the parameters given a {\it single} sample from the model. This is motivated by the applications described above, where it is more common to have access to only a single sample of node activities across the whole network, such as in disease modeling or social network interactions, where it is unrealistic, if not 
	impossible, to generate many independent samples from the underlying model within a reasonable amount of time.

	\subsection{Organization}
	The rest of the paper is organized as follows. We state our main results about the limiting distribution of the ML estimates in Section \ref{sec:samplemean_mle}. The proofs of the main results are given in Section \ref{sec:mlebetahpf}. In Section \ref{sec:applications} we describe how the asymptotic results can be used to construct confidence intervals for the model parameters. Various technical details are given in the Appendix. 
	
	\section{Main Results}
	\label{sec:samplemean_mle}
	
	In this section we state our main results on the limiting properties of the the ML estimates in the $p$-spin Curie-Weiss model. Recall that our aim is to estimate the parameters $\beta$ and $h$ given a single sample $\bs \sim \p_{\beta,h,p} $ using the method of maximum likelihood, which involves maximizing the probability mass function \eqref{cwwithmg} with respect to $\beta$ and $h$. Note that the model \eqref{cwwithmg} can be written more compactly as  
	\begin{align}\label{cwss}
		\mathbb{P}_{\beta,h,p}(\bs) & = \frac{1}{2^{N} Z_N(\beta,h,p)} \exp\Big\{N\left(\beta \overline{\sigma}^p_N + h\os  \right) \Big\} , 
	\end{align}
	where $\os = \frac{1}{N} \sum_{i=1}^N \sigma_i	$ is the average magnetization. 
This shows that the $p$-spin Curie-model has a single sufficient statistic $\os$ which suggests, as mentioned before, that the parameters $(\beta, h)$ cannot be estimated simultaneously. In fact, one can show that the joint ML estimates for $(\beta, h)$ might not exist with probability 1 (see Lemma \ref{mleexist} for details). However, it is possible to marginally estimate one of the parameters assuming that the other is known. Hereafter, given $\bs \sim \p_{\beta,h,p} $, we denote by $\hat{\beta}_N$ and $\hat{h}_N$ the marginal maximum likelihood estimators of $\beta$ and $h$, respectively. More precisely, for fixed $h \in \R$, $\hat{\beta}_N$ is a solution of the equation (in $\beta$)
	\begin{equation}\label{eqmle}
		\e_{\beta,h,p} \left(\overline{\sigma}^p_N \right) = \overline{\sigma}^p_N . 
	\end{equation} 
	Similarly, for fixed $\beta$, $\hat{h}_N $ is a solution of the equation (in $h$)
	\begin{equation}\label{eqmleh}
		\e_{\beta,h,p} \left(\os \right) = \os . 
	\end{equation} 
	The limiting distribution of the ML estimates of $h$ and $\beta$ therefore depend on the fluctuations of the average magnetization $\os$ across the parameter space $\Theta:=\{(\beta,h): \beta \ge 0, h \in \mathbb{R}\}$. These are described in Section \ref{sec:mle_h} and Section \ref{sec:mle_beta}, respectively. The results are summarized in a phase diagram in Section \ref{sec:mle_beta_h_II}. We begin by recalling some relevant definitions from \cite{cmp} in Section \ref{sec:samplemean}.

	\subsection{Preparations}
	\label{sec:samplemean}

For $p \geq 2$ and $(\beta,h) \in \Theta:=[0, \infty) \times \R$, define the function $H = H_{\beta,h,p}:[-1,1]\rightarrow \mathbb{R}$ as
	\begin{align}\label{eq:H}
		H(x) := \beta x^p + hx - I(x),
	\end{align} where $I(x) := \frac{1}{2}\left\{(1+x)\log(1+x) + (1-x)\log(1-x) \right\}$, for $x \in  [-1, 1]$, is the binary entropy function. The points of maxima of this function determine the typical values of $\bar \sigma_N$ and, hence, play a crucial role in our results. It follows from results in \cite{cmp} that the function $H$ can have one, two, or three global maximizers in the open interval $(-1, 1)$, which leads to the following definition:\footnote{For a smooth function $f: [-1, 1] \rightarrow \R$ and $x \in (-1, 1)$, the first and second derivatives of $f$ at the point $x$ will be denoted by $f'(x)$ and $f''(x)$, respectively. More generally, for $s \geq 3$, the $s$-th order derivative of $f$ at the point $x$ will be denoted by $f^{(s)}(x)$. } 
	
	\begin{definition}\label{punique} Fix $p \geq 2$ and $(\beta,h) \in \Theta$, and let $H$ be as defined above in \eqref{eq:H}. 
		\begin{enumerate} 
			
			\item The point $(\beta,h)$ is said to be $p$-{\it regular}, if the function $H_{\beta,h,p}$ has a unique global maximizer $m_* = m_*(\beta,h,p) \in (-1,1)$ and $H_{\beta,h,p}''(m_*) < 0$.\footnote{A point $m \in (-1, 1)$ is a global maximizer of $H$ if $H(m) > H(x)$, for all $x\in [-1,1]\setminus \{m\}$.} Denote the set of all $p$-regular points in $\Theta$ by $\cR_p$.
			
			\item The point $(\beta,h)$ is said to be $p$-{\it special}, if $H_{\beta,h,p}$ has a unique global maximizer $m_* = m_*(\beta,h,p) \in (-1,1)$ and $H_{\beta,h,p}''(m_*) = 0$.  
			
			\item The point $(\beta,h)$ is said to be $p$-{\it critical}, if $H_{\beta,h,p}$ has more than one global maximizer. 
			
		\end{enumerate}
	\end{definition} 
	
	Note that the three cases above form a disjoint partition of the parameter space $\Theta$. Hereafter, we denote the set of $p$-critical points by $\cp$, and the set of points $(\beta, h)$ where $H_{\beta, h, p}$ has exactly two global maximizers by $\cp^+$.  It follows from \cite[Lemma B.3]{cmp} that the set of points in $\cp$ form a continuous $1$-dimensional curve in the parameter space $\Theta$ (see also Figure \ref{figure:ordering1} and Figure \ref{figure:ordering2}). Next, we consider points with three global maximizers, that is $\cp \backslash \cp^+$. To this end, define 
	\begin{align}\label{eq:betatilde}	
		\tilde{\beta}_p := \sup\left\{\beta \geq 0: \sup_{x\in [-1,1]}H_{\beta,0,p}(x) = 0  \right\}.  
	\end{align} 
	Now, depending on whether $p$ is odd or even we have the following two cases (see \cite[Lemma B.1]{cmp}):

	\begin{itemize}
		
		\item $p \geq 3$ odd: In this case  for all points $(\beta, h) \in \cp$, the function $H_{\beta, h, p}$ has exactly two global  maximizers, that is, $\cp=\cp^+$. 
		
		\item $p \geq 4$ even:  Here, there is a unique point 
		$\lambda_p:=(\tilde{\beta}_p,0) \in \cp$, with $\tilde \beta_p$ as defined in \eqref{eq:betatilde}, at which the function $H_{\tilde{\beta}_p,0,p}$ has exactly three global maximizers. For all other points in $(\beta, h)  \in \cp$, $H_{\beta, h, p}$ has exactly two global maximizers, that is,  $\cp=\cp^+ \cup \{\lambda_p \}$. In the case, $p\geq 4$ is even, we will refer to the point $\lambda_p$, or, equivalently, the point $\tilde{\beta}_p$, as the $p$-{\it strongly critical} point.\footnote{Note that the point $\tilde \beta_p$ is defined for all $p \geq 2$ (even or odd) as in \eqref{eq:betatilde}. However, for $p \geq 3$ odd, this point is  $p$-critical, but not $p$-strongly critical (that means it belongs to $\sC_p^+$). On the other hand, for $p=2$ this point is 2-special (see discussion in Remark \ref{remark:theta}).} Hereafter, when the need while arise to distinguish strongly critical points from other critical points, we will refer to a point which is $p$-critical but not $p$-strongly critical, as $p$-{\it weakly critical}. Note that the collection of all $p$-weakly critical points is precisely the set $\cp^+$.
	\end{itemize} 
	It remains to describe the structure of $p$-special points. To this end, fix $p \geq 3$ and define the following quantities: 
	\begin{align}\label{eq:beta_h_special}
		\check{\beta}_p := \frac{1}{2(p-1)} \left(\frac{p}{p-2}\right)^{\frac{p-2}{2}}\quad\textrm{and}\quad\check{h}_p := \tanh^{-1}\left(\sqrt{\frac{p-2}{p}}\right) - \check{\beta}_p p \left(\frac{p-2}{p}\right)^{\frac{p-1}{2}}. 
	\end{align}
	Again, depending on whether $p$ is even or odd there are two cases (\cite[Lemma B.2]{cmp}): 
	
	\begin{itemize}
		
		\item $p \geq 3$ odd: In this case,  there is only one $p$-special point $\tau_p:=(\check{\beta}_p,\check{h}_p)$.
		
		\item $p \geq 4$ even: Here, by the symmetry of the model about $h=0$, there are two $p$-special points $\tau_p^+:=(\check{\beta}_p, \check{h}_p)$ and $\tau_p^-:=(\check{\beta}_p, -\check{h}_p)$. 
		
	\end{itemize} 
	These points are especially interesting, because, as we will see in a moment, here the average magnetization has fluctuations of order $N^{\frac{1}{4}}$ and a non-Gaussian limiting distribution.

Note that, on recalling that $\cR_p$ denotes the set of all $p$-regular points and $\sC_p^+$ the set of points $(\beta, p)$ where $H_{\beta, h, p}$ has exactly two maximizers, the discussion above can be summarized as follows: 
	\begin{align}\label{eq:parameter_space}
		\Theta= 
		\left\{
		\begin{array}{cc}
			\cR_p \bigcup \sC_p^+ \bigcup \{\tau_p\}  &  \text{ for } p \geq 3  \text{ odd},    \\
			\cR_p \bigcup \sC_p^+ \bigcup  \{\lambda_p, \tau_p^+, \tau_p^-\}  &  \text{ for } p \geq 4 \text{ even}. 
		\end{array}
		\right. 
	\end{align}
	Figure \ref{figure:ordering1} and Figure \ref{figure:ordering2} illustrates this decomposition of the parameter space for $p=4$ and $p=5$, respectively.

	\begin{remark}\label{remark:theta} Note that \eqref{eq:parameter_space} provides a complete characterization of the parameter space for $p\geq 3$. As mentioned before, in the well-studied case of $p=2$, the situation is relatively simpler \cite{dm_gibbs_notes,ellis_book}. In this case, $H_{\beta, h, p}$ can have at most two global maximizers, that is, it has no strongly critical points, hence, $\cC_2=\cC_2^+$. In fact, it follows from \cite{ellis_book} that the set of points $(\beta, h)$ with exactly two global maximizers $\cC_2^+$ is the open half-line $(0.5, \infty)\times \{0\}$. Moreover, there is a single 2-special point $(0.5, 0)$ (where there the function $H$ has a unique maximum, but the double derivative is zero), and all the remaining points $\Theta \backslash [0.5, \infty)$ are 2-regular. This shows that for $p=2$ there is no point in $\Theta$ with $h \ne 0$ that is critical. 
	\end{remark}

	\subsection{ML Estimate of $h$} 
	\label{sec:mle_h}
	
	In order to describe the asymptotic distribution of the ML estimate of $h$, we need the following definition: 
	
	\begin{definition}\label{defn:halfnormal}
		For $\sigma > 0$, the {\it positive half-normal distribution} $N^+(0,\sigma^2)$ is defined as the distribution of $|Z|$, where $Z \sim N(0,\sigma^2)$. The {\it negative half-normal distribution} $N^-(0,\sigma^2)$ is defined as the distribution of $-|Z|$, where $Z \sim N(0,\sigma^2)$.
	\end{definition}
	
	The asymptotic distribution of the ML estimate of $h$ is summarized in the theorem below. As expected, the results depend on whether $(\beta, h)$ is regular, critical, or special, which we state separately in the theorems below. In this regard, denote by $\delta_x$ the point mass at $x$.  We begin with the case when $(\beta, h)$ is regular. Throughout, $H=H_{\beta,p,h}$ will be as defined in \eqref{eq:H}.

	\begin{thm}[Asymptotic distribution of $\hat{h}_N$ at $p$-regular points]\label{cltintr3_1}  Fix $p \geq 3$ and suppose $(\beta, h) \in \Theta$ is $p$-regular. Assume $\beta$ is known and $\bs \sim \mathbb{P}_{\beta,h,p} $. Then denoting the unique maximizer of $H$ by $m_*= m_*(\beta,h,p)$, as $N \rightarrow \infty$,
		\begin{align}\label{eq:h_1}
			N^{\frac{1}{2}}(\hat{h}_N  - h) \xrightarrow{D} N\left(0, -H''(m_*)\right).
		\end{align} 
	\end{thm}

	This result shows that $\hat{h}_N $ is $N^{\frac{1}{2}}$-consistent and asymptotically normal at the regular points. Before discussing more about the implications of this theorem, we state the result for the asymptotic distribution of $\hat{h}_N$ when $(\beta, h)$ is $p$-special.

	\begin{thm}[Asymptotic distributions of $\hat{h}_N$ at $p$-special points]\label{cltintr3_2}	
		Fix $p \geq 3$ and suppose $(\beta, h) \in \Theta$ is $p$-special. Assume $\beta$ is known and $\bs \sim \mathbb{P}_{\beta,h,p} $. Then denoting the unique maximizer of $H$ by $m_*= m_*(\beta,h,p)$, as $N \rightarrow \infty$, 
		\begin{align}\label{eq:h_5} 
			N^\frac{3}{4}(\hat{h}_N - h) \xrightarrow{D} G_1, 
		\end{align} 
		where the distribution function of $G_1$ is given by  
		$$G_1(t) = F_{0,0}\left(\int_{-\infty}^{\infty} u ~\mathrm d F_{0,t}(u)\right),$$ with $F_{0, t}$ as defined in \eqref{eq:beta_h_distribution} below. 
	\end{thm}

	Finally, we consider the case $(\beta, h)$ is $p$-critical. Here, it is convenient to consider the cases $p$ is odd or even separately. 
	
	\begin{thm}[Asymptotic distribution of $\hat{h}_N$ at $p$-critical points]\label{cltintr3_III}
		Fix $p \geq 3$ and suppose $(\beta, h) \in \Theta$ is $p$-critical. Assume $\beta$ is known and $\bs \sim \mathbb{P}_{\beta,h,p} $. Denote the $K \in \{2,3\}$ maximizers of $H$ by $m_1 := m_1(\beta, h, p) < \ldots < m_K := m_K(\beta, h, p)$, and let $p_1, \ldots, p_K$ be defined as:
		\begin{equation}\label{eq:p1}
			p_k := \frac{\left[(m_k^2-1)H''(m_k)\right]^{-1/2}}{\sum_{i=1}^K \left[(m_i^2-1)H''(m_i)\right]^{-1/2}}.
		\end{equation} 
		
		\begin{itemize}

			\item[$(1)$] Suppose $p \geq 3$ is odd. In this case, the function $H$ has exactly two (asymmetric) maximizers $m_1 < m_2$ and, as $N \rightarrow \infty$,
			\begin{align}\label{eq:h_31}
				N^{\frac{1}{2}}(\hat{h}_N - h) \xrightarrow{D} \tfrac{p_1}{2} N^{-}\left(0, -H''(m_1)\right) + \tfrac{1-p_1}{2} N^{+}\left(0, -H''(m_2)\right) + \tfrac{1}{2} \delta_0, 
			\end{align} 
			where $N^{\pm}$ are the half-normal distributions as in Definition $\ref{defn:halfnormal}$.

			\item[$(2)$] Suppose $p \geq 4$ is even.  Then the following hold:  
			
			\begin{enumerate}
				
				\item[$\bullet$] If $h \ne 0$,  then the function $H$ has exactly two (asymmetric) maximizers $m_1 < m_2$ and, as $N \rightarrow \infty$,
				\begin{align}\label{eq:h_3}
					N^{\frac{1}{2}}(\hat{h}_N - h) \xrightarrow{D} \tfrac{p_1}{2} N^{-}\left(0, -H''(m_1)\right) + \tfrac{1-p_1}{2} N^{+}\left(0, -H''(m_2)\right) + \tfrac{1}{2} \delta_0.  
				\end{align}

				\item[$\bullet$] If $h=0$ and $\beta> \tilde{\beta}_p$, then the function $H$ has exactly two symmetric maximizers $m_1, m_2$, where $m_2 = -m_1 = m_*$, for some $m_* = m_*(\beta, h, p)> 0$. Then, as $N \rightarrow \infty$, 
				\begin{align}\label{eq:h_4}
					N^{\frac{1}{2}} \hat{h}_N  \xrightarrow{D} \tfrac{1}{2} N\left(0, -H''(m_*)\right)  + \tfrac{1}{2} \delta_0.
				\end{align}
				
				\item[$\bullet$] If $h=0$ and $\beta = \tilde{\beta}_p$, the function $H$ has three maximizers $m_1=-m_*$, $m_2=0$, and $m_3=m_*$, where $m_*=m_*(\beta, h, p) > 0$.  Then, as $N \rightarrow \infty$,
				\begin{align}\label{eq:h_2}
					N^{\frac{1}{2}}\hat{h}_N \xrightarrow{D} p_1 N\left(0, -H''(m_1)\right) + (1-p_1)\delta_0,
				\end{align} 
				where $p_1$ is as defined in \eqref{eq:p1}.  
			\end{enumerate} 
		\end{itemize}
		
	\end{thm}
	
	The proofs of these results are given in Section \ref{sec:mlebetahpf}. The results above show that for all points in the parameter space, the ML estimate $\hat{h}_N$ is a consistent estimate of $h$, that is, $\hat{h}_N \pto h$. Moreover, the rate of convergence is $N^{\frac{1}{2}}$, except at the $p$-special points. However, at the $p$-special point(s), 
	the rate improves to $N^\frac{3}{4}$, that is, the ML estimate of $h$ at these point(s) is {\it superefficient}, converging to the true value of $h$ faster than the usual $N^{\frac{1}{2}}$ rate at the neighboring points. Another interesting feature is that, while at the regular points $\hat h_N$ has a simple Gaussian limit, at the critical points it has a mixture distribution, consisting of (half) normals and a point mass at $0$. The reason the limiting distribution has a point mass at $0$ is because the average magnetization $\os$ is  ``discontinuous"  under the perturbed measure $\p_{\beta,h+ t/\sqrt N,p}$, as $t$ transitions from negative to positive. In fact, Lemma \ref{redunsup} (in Appendix \ref{sec:pf_concentration}) shows that under the measure $\p_{\beta,h+ t/\sqrt{N},p}$, the point where $\os$ concentrates depends on the sign of the perturbation factor $t$. Therefore, since the distribution function of $N^{\frac{1}{2}}(\hat{h}_N - h)$ evaluated at $t$ depends on the law of $\os$ under the perturbed measure $\p_{\beta,h+ t/\sqrt N, p}$ (see the calculations in Section \ref{sec:pf_beta_h_2} for details), it has a discontinuity at the point $t=0$, and, hence, a point mass at $0$ appears in the limit. 
	
	Another interesting revelation are the results in \eqref{eq:h_31} and \eqref{eq:h_3}, where the $H$ function has two (asymmetric) maximizers. In this case, the ML estimate $\hat h_N$ converges to a three component mixture, which has a point mass at zero with probability $\frac{1}{2}$ and is a mixture of two half normal distributions, with probabilities $\frac{p_1}{2}$ and $\frac{1-p_1}{2}$, respectively. This corresponds to the region of the critical curve where $h \ne 0$ (and also the point $(\tilde \beta_p, 0)$, for $p \geq 3$ odd), a striking new phenomena that emerges only when $p \geq 3$. Note that, this does not happen for $p=2$, because, in this case, $\sC_p^+=(0.5, \infty)\times \{0\}$, hence, the two maximizers at any 2-critical point are symmetric about zero, and the two half normal mixing components combine to form a single Gaussian, and the resulting limit is the mixture of a single normal and a point mass at zero, as is the case in \eqref{eq:h_4} above.

	\subsection{ML Estimate of $\beta$}
	\label{sec:mle_beta} 
	
	Here, we consider the ML estimate  $\hat{\beta}_N$ of $\beta$. As before, the results depend on whether $(\beta, h)$ is regular, critical, or special. However, the analysis here is more involved, and each of these cases breaks down into further cases, depending on the value of the maximizers, the parity of $p$, and the sign of $h$. We begin with the case when $(\beta, h)$ is regular. As always, $H=H_{\beta,p,h}$ will be as defined in \eqref{eq:H}.

	\begin{thm}[Asymptotic distributions of $\hat{\beta}_N$ at $p$-regular points]\label{thmmle1} 
		Fix $p \geq 3$ and suppose $(\beta, h) \in \Theta$ is $p$-regular. Assume $h$ is known and $\bs \sim \mathbb{P}_{\beta,h,p} $. Then denoting the unique maximizer of $H$ by $m_*= m_*(\beta,h,p)$, the following hold, 
		\begin{itemize}
			
			\item[$\bullet$]  If $m_* \neq 0$, then, as $N \rightarrow \infty$, 
			\begin{align}\label{eq:bmle_m_1}
				N^{\frac{1}{2}}(\hat{\beta}_N-\beta) \xrightarrow{D} N\left(0, -\frac{H''(m_*)}{p^{2}m_*^{2p-2}} \right). 
			\end{align}
			
			\item[$\bullet$] If $m_* = 0$, (equivalently, $h = 0$ and $\beta < \tilde{\beta}_p$), then, as $N \rightarrow \infty$,
			\begin{align}\label{eq:bmle_m_2} 
				\hat{\beta}_N \xrightarrow{D} 
				\begin{cases}
					\frac{1}{2}\delta_{\tilde{\beta}_p} + \frac{1}{2} \delta_{-\tilde{\beta}_p} &\quad\text{if}~p~\textrm{is odd},\\
					\gamma_p \delta_{-\infty} + (1-\gamma_p)\delta_{\tilde{\beta}_p} &\quad\text{if}~p~\textrm{is even},\\
				\end{cases}
			\end{align}
			where $\gamma_p := \p(Z^p \leq \e Z^p)$ with $Z \sim N(0,1)$. 
		\end{itemize}
	\end{thm}

	We will discuss the various implications of the above theorem later in this section. Now, we state the result for the asymptotic distribution of $\hat{\beta}_N$ when $(\beta, h)$ is $p$-special.

	\begin{thm}[Asymptotic distributions of $\hat{\beta}_N$ at $p$-special points]\label{thmmle2}	
		Fix $p \geq 3$ and suppose $(\beta, h) \in \Theta$ is $p$-special. Assume $h$ is known and $\bs \sim \mathbb{P}_{\beta,h,p} $. Then denoting the unique maximizer of $H$ by $m_*= m_*(\beta,h,p)$, as $N \rightarrow \infty$,
		\begin{align}\label{eq:bmle_m_3} 
			N^\frac{3}{4}(\hat{\beta}_N - \beta) \xrightarrow{D} G_2, 
		\end{align} 
		where the distribution function of $G_2$ is given by 
		$$G_2(t) = F_{0,0}\left(\int_{-\infty}^{\infty} u ~\mathrm d F_{t,0}(u)\right),$$ with $F_{t,0}$ as defined in  \eqref{eq:beta_h_distribution} below. 
	\end{thm}

	Finally, we consider the case $(\beta, h)$ is $p$-critical. The situation here is quite delicate, depending on various things like weak and strong criticality, parity of $p$, and the sign of the field $h$. 
	
	\begin{thm}[Asymptotic distribution of $\hat{\beta}_N$ at $p$-critical points]\label{thmmle_III}
		Fix $p \geq 3$ and suppose $(\beta, h) \in \Theta$ is $p$-critical. Assume $h$ is known and $\bs \sim \mathbb{P}_{\beta,h,p} $. Denote the $K \in \{2,3\}$ maximizers of $H$ by $m_1 := m_1(\beta, h, p) < \ldots < m_K := m_K(\beta, h, p)$, and let $p_1, \ldots, p_K$ be as in \eqref{eq:p1}. 
		
		\begin{itemize}
			
			\item[$(1)$] Suppose $p \geq 3$ is odd. In this case, the function has exactly two maximizers  $m_1 < m_2$. Then, as $N \rightarrow \infty$, the following hold: 
			
			\begin{enumerate}
				\item[$\bullet$] If $(\beta,h) \neq (\tilde{\beta}_p,0)$, where $\tilde{\beta}_p$ is defined in \eqref{eq:betatilde}, then 
				\begin{align}\label{eq:bmle_multimax_1} 
					N^{\frac{1}{2}}(\hat{\beta}_N - \beta)\xrightarrow{D} \tfrac{p_1}{2}N^-\left(0,-\frac{H''(m_1)}{p^2m_1^{2p-2}}\right)+ \tfrac{1-p_1}{2}N^+\left(0,-\frac{H''(m_2)}{p^2m_2^{2p-2}}\right)+\tfrac{1}{2}\delta_0.
				\end{align} 
				
				\item[$\bullet$] If $(\beta,h) = (\tilde{\beta}_p,0)$, then 
				\begin{align}\label{eq:bmle_multimax_2} 
					N^{\frac{1}{2}}(\hat{\beta}_N - \beta)\xrightarrow{D} \tfrac{p_1}{2}\delta_{-\infty} + \tfrac{1-p_1}{2}N^+\left(0,-\frac{H''(m_2)}{p^2m_2^{2p-2}}\right)+\tfrac{1}{2}\delta_0. 
				\end{align} 
			\end{enumerate}
			
			\item[$(2)$] Suppose $p \geq 4$ is even. Then the following hold, as $N \rightarrow \infty$:   
			\begin{enumerate}
				\item[$\bullet$] If $h > 0$, then 
				\begin{align}\label{eq:bmle_multimax_3}
					N^{\frac{1}{2}}(\hat{\beta}_N - \beta)\xrightarrow{D} \tfrac{p_1}{2}N^-\left(0,-\frac{H''(m_1)}{p^2m_1^{2p-2}}\right)+ \tfrac{1-p_1}{2}N^+\left(0,-\frac{H''(m_2)}{p^2m_2^{2p-2}}\right)+\tfrac{1}{2}\delta_0.
				\end{align}
				
				\item[$\bullet$]  If $h < 0$, then 
				\begin{align}\label{eq:bmle_multimax_4} 
					N^{\frac{1}{2}}(\hat{\beta}_N - \beta)\xrightarrow{D} \tfrac{p_1}{2}N^+\left(0,-\frac{H''(m_1)}{p^2m_1^{2p-2}}\right)+ \tfrac{1-p_1}{2}N^-\left(0,-\frac{H''(m_2)}{p^2m_2^{2p-2}}\right)+\tfrac{1}{2}\delta_0. 
				\end{align}
				
				\item[$\bullet$] If $h=0$ and $\beta> \tilde{\beta}_p$, there are exactly two maximizers $m_1=-m_*$ and $m_2=m_*$ of $H$, where $m_* = m_*(\beta, h, p)> 0$. In this case, 
				\begin{align}\label{eq:bmle_multimax_5} 
					N^{\frac{1}{2}}(\hat{\beta}_N - \beta)\xrightarrow{D} N\left(0,-\frac{H''(m_*)}{p^2m_*^{2p-2}}\right).
				\end{align}
				
				\item[$\bullet$] If $h=0$ and $\beta = \tilde{\beta}_p$, there are exactly three maximizers $m_1 = -m_*$, $m_2=0$, and $m_3 = m_*$ of $H$, where $m_*=m_*(\beta, h, p) > 0$. In this case, \begin{align}\label{eq:bmle_multimax_6} 
					N^{\frac{1}{2}}(\hat{\beta}_N - \beta)\xrightarrow{D} p_2 \gamma_p \delta_{-\infty} + p_1N^+\left(0,-\frac{H''(m_*)}{p^2m_*^{2p-2}}\right) + (1-p_1-p_2\gamma_p) \delta_0,
				\end{align} 
				where $\gamma_p := \p(Z^p \leq \e Z^p)$ and $Z$ is a standard normal random variable.
			\end{enumerate}
		\end{itemize}	
	\end{thm}

	The proofs of the above results are given in Section \ref{sec:mlebetahpf}. 
	Below, we summarize the main consequences of the above results and highlight the various new phenomena that emerge as one moves from the 2-spin to the $p$-spin case.

	\begin{itemize}

		\item  For $p$-regular points, Theorem \ref{thmmle1} shows that when the unique maximizer $m_* \ne 0$, then $\hat{\beta}_N$ is consistent at rate $N^{\frac{1}{2}}$ with a limiting normal distribution. On the other hand, when $m_* = 0$, which happens in the interval  $[0, \tilde \beta_p)$, the ML estimate $\hat{\beta}_N$ is inconsistent. In this regime, when $p \geq 3$ is odd, then $\hat{\beta}_N$ concentrates at $\pm \bm \tilde \beta_p$ with probability $\frac{1}{2}$, irrespective of the value of true value of $\beta \in  [0, \tilde \beta_p)$. The situation is even more strange when $p \geq 4$ is even. Here,  $\hat{\beta}_N$ concentrates at either $\bm \tilde \beta_p$ or escapes to negative infinity, that is, with positive probability $\hat{\beta}_N$ is unbounded, when $p \geq 4$ and $\beta \in  [0, \tilde \beta_p)$. The corresponding results for $p=2$ are similar in the sense that, for $\beta \in [0, 0.5)$ (recall that $\tilde{\beta_2}=0.5$), the ML estimate $\hat{\beta}_N$ is inconsistent. However, unlike in the case for $p\geq 4$ even, the ML estimate $\hat{\beta}_N$, when $p=2$, is always finite and converges to a (properly centered and rescaled) chi-squared distribution \cite[Theorem 1.4]{comets}.

		\item For $p$-special points Theorem \ref{thmmle2} shows that $\hat{\beta}_N$  converges to $\beta$ at rate $N^{-\frac{3}{4}}$, that is, it is superefficient. Recall that the same thing happens for $\hat h_N$ at $p$-special points (Theorem \ref{cltintr3_2}). In comparison, for $p=2$ at the only 2-special $(0.5, 0)$,  $\hat{h}_N$ is superefficient with rate $N^{-\frac{3}{4}}$ \cite[Theorem 1.3]{comets}, but $\hat{\beta}_N$ remains $N^{\frac{1}{2}}$-consistent \cite[Theorem 1.4]{comets}. This is because when $(\beta,h)$ is $p$-special, the unique maximizer $m_*$ of $H_{\beta,h,p}$ is $0$ when $p=2$, but non-zero, for $p \geq 3$. This creates a difference in the rate of convergence of the maxima of $H_{\beta_N,h_N,p}$ towards the maximum of $H_{\beta,h,p}$ for some suitably chosen perturbation $(\beta_N,h_N)$ of $(\beta,h)$, which is an important step in deriving the asymptotic rate of convergence of the ML estimates. Another interesting difference is that for $p=2$, the only 2-special point $(0.5, 0)$ coincides with the thermodynamic threshold of the 2-spin Curie-Weiss model. However, for $p\geq 3$, the $p$-special points (the point $(\check{\beta}_p,\check{h}_p)$, for $p \geq 3$ odd, and the points 
		$(\check{\beta}_p, \pm \check{h}_p)$, for $p \geq 4$ even), where we get the non-Gaussian limits of $\os$, ${\hat h}_N$, and ${\hat \beta}_N$, have nothing to do with the thermodynamic threshold of the $p$-spin Curie-Weiss model, but rather depends on the vanishing property of the second derivative of $H$ at its maximizer. On the contrary,  quite remarkably, the thermodynamic threshold $(\tilde{\beta}_p, 0)$ of the $p$-spin Curie-Weiss model (recall definition in \eqref{eq:betatilde}) turns out to be a $p$-weakly critical point for $p \geq 3$ odd, and the only $p$-strongly critical point for $p \geq 4$ even, another unexpected phenomenon unearthed by our results.

		\item  The landscape is much more delicate for $p$-critical points, as can be seen from Theorem \ref{thmmle_III}. In this case, the limiting distribution of $\hat{\beta}_N$ converges to various mixture distributions, depending on, among other things, the sign of $h$ and the parity of $p$. As in the case of $\hat{h}_N$, a particularly interesting new phenomena is the three component mixture that arises in the limiting distribution of $\hat{\beta}_N$ when the critical curve $\sC_p^+$ intersects the region $h \ne 0$. This corresponds to the result \eqref{eq:bmle_multimax_1} for $p \geq 3$ odd, and results in \eqref{eq:bmle_multimax_3} and \eqref{eq:bmle_multimax_4}  for $p \geq 4$ even. Recall, from the discussion following Theorem \ref{cltintr3_III}, that this does not happen for $p=2$, because, in this case, $\sC_p^+=(0.5, \infty)\times \{0\}$, hence, the two maximizers at any 2-critical point are symmetric about zero, and the two half normal mixing components combine to form a single Gaussian. As a result, the limit is the mixture of a single normal and a point mass at zero. Interestingly, this also happens for $p \geq 4$ even, when the critical curve intersects the line $h=0$ and is strictly above the threshold $\tilde{\beta}_p$, as seen in  \eqref{eq:bmle_multimax_5} above.

		\item The final bit in the puzzle is the point of thermodynamic phase transition $(\tilde \beta_p, 0)$. Here, the ML estimate $\hat{\beta}_N$  is not $N^{\frac{1}{2}}$-consistent. More precisely, in the limit, $N^{\frac{1}{2}}(\hat \beta_N-\beta)$ has a point mass at negative infinity with positive probability,  and is a mixture of a folded normal and a point mass with the remaining probability (as described in \eqref{eq:bmle_multimax_2} and \eqref{eq:bmle_multimax_6}). In contrast, as explained in the second case above, when $p=2$, then at the point of thermodynamic phase transition ($\tilde{\beta}_2=0.5$) the ML estimate ${\hat \beta}_N$ is $N^{\frac{1}{2}}$-consistent. 
		
	\end{itemize}

	\subsection{Summarizing the Phase Diagram}  
	\label{sec:mle_beta_h_II}

	The results above can be compactly summarized and better visualized in a phase diagram, which shows the  partition of the parameter space described in \eqref{eq:parameter_space}. The phase diagrams for $p=4$ and $p=5$, obtained by numerical optimization of the function $H$ over a fine grid of parameter values, are shown in Figure \ref{figure:ordering1} and Figure \ref{figure:ordering2}, respectively.   The limiting distributions that arise in the different regions of the phase diagram are described in the figure legends.

	\begin{figure}[h]
		\centering
		\begin{minipage}[c]{1.0\textwidth}
			\centering
			\includegraphics[width=6.25in,height=3.05in]
			{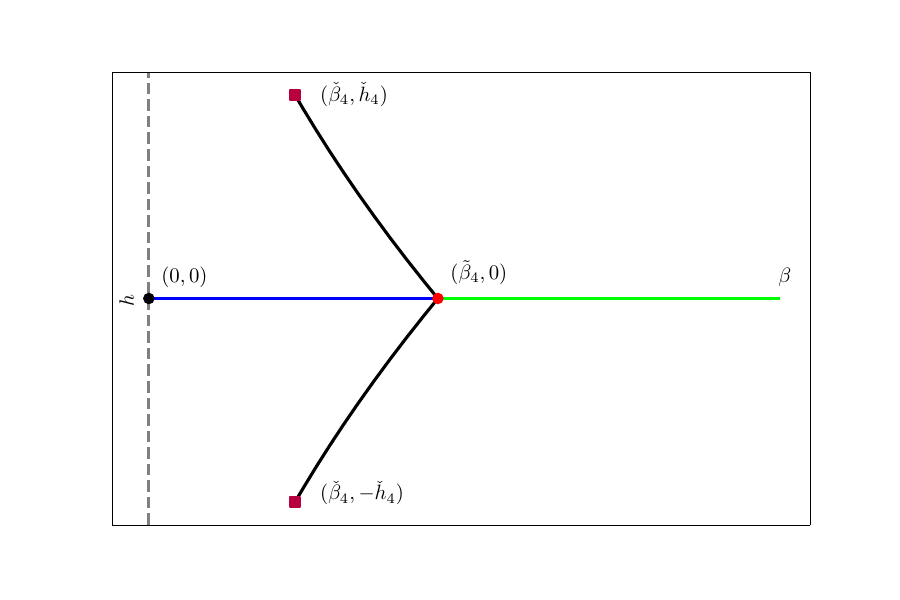}\\
			\caption{\small{The phase diagram for $p=4$: The properties of the ML estimates in the  different regions of the parameter space $\Theta= [0, \infty) \times \R$ are as follows: }}
			\label{figure:ordering1}
			\small{
				\begin{itemize} 
					\item[] 
					\begin{itemize}
						
						\item The $\mysquare[white]$ (white) region: These are the $p$-regular points where $H$ has a unique global maximizer $m_* \ne 0$  and $H''(m_*) < 0$.  Hence, $\hat \beta_N$ and $\hat h_N$  are both $N^{\frac{1}{2}}$-consistent and asymptotically normal, by \eqref{eq:bmle_m_1} and \eqref{eq:h_1}, respectively. 
						
						\item The \textcolor{blue}{\rule{0.75cm}{0.75mm}} line: These are the $p$-regular points where $H$ has a unique global maximizer $m_* = 0$  and $H''(0) < 0$. Hence, $\hat \beta_N$ is inconsistent by \eqref{eq:bmle_m_2},  but $\hat h_N$ is $N^{\frac{1}{2}}$-consistent and asymptotically normal by \eqref{eq:h_1}. 
						
						\item The $\mysquare[purple]$ points: These are the $p$-special points. Here, $H$ has a unique maximizer $m_*$, but $H''(m_*)=0$. Hence,  $\hat \beta_N$ and $\hat h_N$  are both superefficient, converging at rate $N^{\frac{3}{4}}$ to non-Gaussian distributions, by \eqref{eq:bmle_m_3} and \eqref{eq:h_5}, respectively.

						\item The \textcolor{black}{\rule{0.75cm}{0.75mm}} curve: These are $p$-weakly critical points where $h \ne 0$. Here, $H$ has two global (non-symmetric) maximizers. Both $\hat \beta_N$ and $\hat h_N$  are $N^{\frac{1}{2}}$-consistent and asymptotically a three component mixture (comprising of two half normal normal distributions and a point mass at zero), by  \eqref{eq:bmle_multimax_3}, \eqref{eq:bmle_multimax_4},  and \eqref{eq:h_3}, respectively.

						\item The \textcolor{green}{\rule{0.75cm}{0.75mm}} line: These are $p$-weakly critical points where $h = 0$. Here, $H$ has two global symmetric maximizers. Hence, $\hat \beta_N$ is $N^{\frac{1}{2}}$-consistent and asymptotically normal by \eqref{eq:bmle_multimax_5}, and $\hat h_N$ is $N^{\frac{1}{2}}$-consistent and asymptotically a mixture of a normal distribution and a point mass at zero, by \eqref{eq:h_4}. 
						
						\item The \tikz\draw[red,fill=red] (0,0) circle (.8ex); point: This is the $p$-strongly critical point. Here, $H$ has three global maximizers. Hence, $\hat{\beta}_N$ is not $N^{\frac{1}{2}}$-consistent, by \eqref{eq:bmle_multimax_6}, but $\hat{h}_N$ is $N^{\frac{1}{2}}$-consistent and asymptotically a mixture of normal distribution and point mass at 0, by \eqref{eq:h_2}. 	
					\end{itemize}
				\end{itemize}
			}
		\end{minipage}
	\end{figure}

	\begin{figure}[h]
		\centering
		\begin{minipage}[c]{1.0\textwidth}
			\centering
			\includegraphics[width=6.25in,height=3.05in]
			{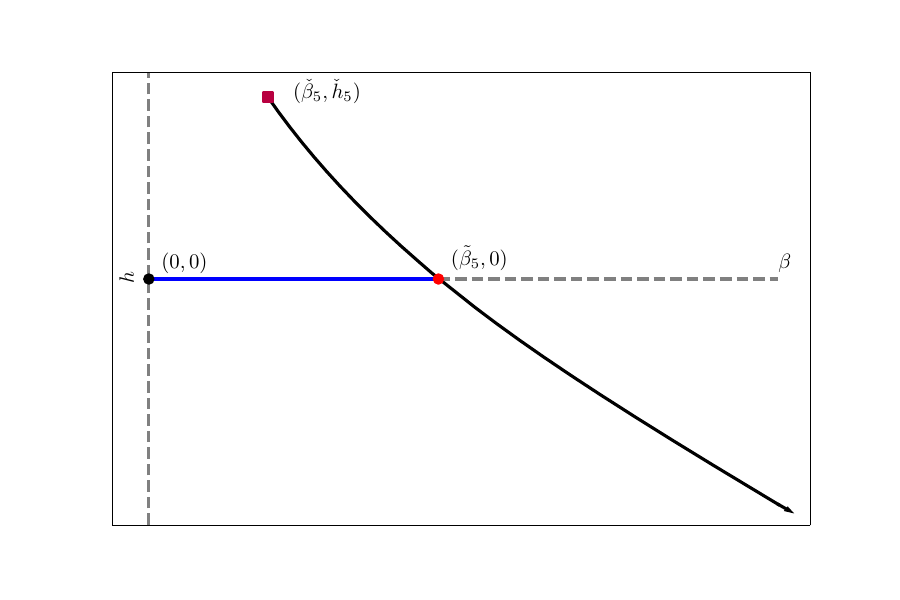}\\
			\caption{\small{The phase diagram for $p=5$: The properties of the ML estimates in the  different regions of the parameter space $\Theta= [0, \infty) \times \R$ are as follows: }}
			\label{figure:ordering2}
			\small{
				\begin{itemize} 
					\item[] 
					\begin{itemize}
						
						\item The $\mysquare[white]$ (white) region: These are the $p$-regular points where $H$ has a unique global maximizer $m_* \ne 0$  and $H''(m_*) < 0$.  Hence, $\hat \beta_N$ and $\hat h_N$  are both $N^{\frac{1}{2}}$-consistent and asymptotically normal, by \eqref{eq:bmle_m_1} and \eqref{eq:h_1}, respectively. 
						
						\item The \textcolor{blue}{\rule{0.75cm}{0.75mm}} line: These are the $p$-regular points where $H$ has a unique global maximizer $m_* = 0$  and $H''(0) < 0$. Hence, $\hat \beta_N$ is inconsistent by \eqref{eq:bmle_m_2},  but $\hat h_N$   $N^{\frac{1}{2}}$-consistent and asymptotically normal by \eqref{eq:h_1}. 
						
						\item The $\mysquare[purple]$ point: This is the only $p$-special point. Here, $H$ has a unique maximizer $m_*$, but $H''(m_*)=0$. Hence,  $\hat \beta_N$ and $\hat h_N$  are both superefficient, converging at rate $N^{\frac{3}{4}}$ to non-Gaussian distributions, by \eqref{eq:bmle_m_3} and \eqref{eq:h_5}, respectively.

						\item The \textcolor{black}{\rule{0.75cm}{0.75mm}} curve: These are $p$-weakly critical points where $h \ne 0$. Here, $H$ has two global (non-symmetric) maximizers. Both, $\hat \beta_N$ and $\hat h_N$  are $N^{\frac{1}{2}}$-consistent and asymptotically a three component mixture (comprising two half normal normal distributions and a point mass at zero), by  \eqref{eq:bmle_multimax_1} and \eqref{eq:h_31}, respectively.

						\item The \tikz\draw[red,fill=red] (0,0) circle (.8ex); point: This is the $p$-weakly critical point with $h=0$. Here, $H$ has two (non-symmetric) global maximizers. Hence, $\hat h_N$  is asymptotically a three component mixture by  \eqref{eq:h_31}, but $\hat{\beta}_N$ is not $N^{\frac{1}{2}}$-consistent by \eqref{eq:bmle_multimax_2}. 
					\end{itemize}
				\end{itemize}
			}
		\end{minipage}
	\end{figure}

	\section{Proofs of the Main Results} 
	\label{sec:mlebetahpf}


	In this section we prove the results described in Sections \ref{sec:mle_h}  and  \ref{sec:mle_beta}. 
	To this end, recall the ML equations \eqref{eqmle} and \eqref{eqmleh}, and for notational convenience, we introduce the following definition , for $m \geq 1$: $$u_{N,m}(\beta, h, p) := \e_{\beta,h,p}  \overline{\sigma}^m_N. $$ 
The derivation the asymptotic distribution of the ML estimates proceeds as follows:

	\begin{itemize}
		
		\item  The first step is to express the distribution functions of $\hat{\beta}_N$ and $\hat{h}_N$ in terms of the average magnetization $\os$. This follows from the ML equations \eqref{eqmle} and \eqref{eqmleh} and the monotonicity of the function $u_{N,m}$ (proved in Lemma \ref{increasing}). To this end, define $a_N = N^{\frac{1}{2}}$, if $(\beta,h)$ not $p$-special and $a_N=N^{\frac{3}{4}}$ if $(\beta,h)$ is $p$-special. Now, note that, fixing $t \in \R$,  
		\begin{align}\label{intui1}
			\left\{a_N(\hat{h}_N -h) \leq t\right\} = \left\{\hat{h}_N  \leq h + \frac{t}{a_N}\right\} 
			&= \left\{\os \leq \e_{\beta, h+ \frac{t}{a_N},p}  (\os)\right\},
		\end{align}
		by the monotonicity of the function $u_{N, 1}(\beta, \cdot, p)$ (using Lemma \ref{increasing}) and the ML equation \eqref{eqmleh}. Similarly, 
		\begin{align}\label{intui2}
			\left\{a_N(\hat{\beta}_N-\beta) \leq t\right\} = \left\{\hat{\beta}_N \leq \beta + \frac{t}{a_N}\right\} 
			&= \left\{\overline{\sigma}^p_N \leq \e_{\beta+ \frac{t}{a_N},h,p} (\overline{\sigma}^p_N)\right\}.
		\end{align}
		by the monotonicity of the function $u_{N, p}(\cdot, h, p)$ (using Lemma \ref{increasing}) and the ML equation \eqref{eqmle}. 
		
		\item The next step is to write the event in \eqref{intui1}  as $$\left\{\frac{N}{a_N}(\os-c) \leq \e_{\beta, h+ \frac{t}{a_N},p}  \left[ \frac{N}{a_N} (\os-c)\right]\right\},$$ 
		for some appropriately chosen centering $c$, and similarly, for the event \eqref{intui2}. Now, if the point $(\beta,h)$ is $p$-regular or $p$-special, the centering $c$ will be the unique global maximizer of $H_{\beta,h,p}$, around which $\os$ concentrates. However, if $(\beta,h)$ is $p$-critical, then the situation is more tricky. In that case, one needs to look at the sign of $t$, and choose the centering $c$ to be that global maximizer of $H_{\beta,h,p}$ around which $\os$ concentrates, under the measure $\p_{\beta, h+ t/a_N, p}$ (for $\hat{h}_N$), and the measure $\p_{\beta+t/a_N, h,p}$ (for $\hat{\beta}_N$). Therefore, to obtain the limiting distribution of the ML estimates we need to derive the asymptotic distribution of $\os$ at perturbed parameter values $(\beta, h+ t/a_N)$ and $(\beta+t/a_N, h)$, for some appropriately chosen sequence $a_N \rightarrow \infty$. This is described in Theorem \ref{cltun} below.  
		
	\end{itemize}

Suppose $\{(\beta_N, h_N)\}_{N \geq 1}$ is a sequence of parameters such that $\beta_N \rightarrow \beta$ and $h_N \rightarrow h$. For notational convenience, henceforth we will denote $\p_{\beta_N, h_N,p}, ~Z_N(\beta_N,h_N,p)$, and $F_N(\beta_N,h_N,p)$,  by $\bar{\p}, \bar{Z}_N$,  and $\bar{F}_N$, respectively. The asymptotic distribution of $\os$ in the different cases at the appropriately perturbed parameter values is given in the following result. The proof is given in Appendix \ref{perturbed_asm}.

	\begin{thm}[Asymptotic distribution of $\os$ under perturbed parameters]\label{cltun} Fix $p \geq 3$, $(\beta, h) \in \Theta$, and $\bar{\beta},\bar{h}\in \mathbb{R}$. Then with $H=H_{\beta,p,h}$ as defined in \eqref{eq:H} the following hold: 
		
		\begin{itemize}
			
			\item[$(1)$] Suppose $(\beta, h)$ is $p$-regular and denote the unique maximizer of $H$ by $m_*= m_*(\beta,h,p)$. Then, for $\bs \sim \mathbb{P}_{\beta+N^{-\frac{1}{2}}\bar{\beta},~h + N^{-\frac{1}{2}}\bar{h},~p}$, as $N \rightarrow \infty$,
			\begin{align}\label{eq:cltun_I} 
				N^{\frac{1}{2}}\left(\os - m_*(\beta,h,p)\right)\xrightarrow{D} N\left(-\frac{\bar{h}+\bar{\beta}pm_*(\beta,h,p)^{p-1}}{H''(m_*)},~-\frac{1}{H''(m_*)}\right). 
			\end{align} 
			
			\item[$(2)$] Suppose $(\beta, h)$ is $p$-critical and denote the $K \in \{2, 3\}$ maximizers  of $H$ denoted by $m_1:=m_1(\beta,h,p),$$\ldots,m_K:=m_K(\beta,h,p)$. Then, for $\bs \sim \P_{\beta, h, p}$, as $N \rightarrow \infty$, 
			\begin{align}\label{eq:cltun_p1}
				\os \xrightarrow{D} \sum_{k=1}^K p_k \delta_{m_k}, 
			\end{align}
			where $p_1, \ldots, p_K$ are as defined in \eqref{eq:p1}. Moreover, if $m$ is any local maximizer of $H$ contained in the interior of an interval $A \subseteq [-1,1]$, such that $H(m) > H(x)$ for all $x\in A\setminus \{m\}$, then for $\bs \sim \mathbb{P}_{\beta+N^{-\frac{1}{2}}\bar{\beta},~h + N^{-\frac{1}{2}}\bar{h},~p}$, as $N \rightarrow \infty$, 
			\begin{equation}\label{eq:cltun_II}
				N^{\frac{1}{2}}\left(\os - m\right)\Big\vert \{ \os \in A\} \xrightarrow{D} N\left(-\frac{\bar{h}+\bar{\beta}pm^{p-1}}{H''(m)},~-\frac{1}{H''(m)}\right).
			\end{equation}

			\item[$(3)$] Suppose $(\beta, h)$ is $p$-special and denote the unique maximizer of $H$ by $m_*= m_*(\beta,h,p)$. Then, for $\bs \sim \mathbb{P}_{\beta+ N^{-\frac{3}{4}}\bar{\beta},~h + N^{-\frac{3}{4}}\bar{h},~p}$, as $N \rightarrow \infty$, 
			\begin{align}\label{eq:cltun_III} 
				N^\frac{1}{4}(\os - m_*(\beta,h,p)) \xrightarrow{D} F_{\bar{\beta},\bar{h}},
			\end{align} 
			where density of $F$ with respect to the Lebesgue measure is given by 
			\begin{align}\label{eq:beta_h_distribution}
				\frac{\mathrm dF_{\bar{\beta},\bar{h}} (x) }{\mathrm d x} ~\propto~ \exp\left(\frac{H^{(4)}(m_*)}{24} x^4 + (\bar{\beta}pm_*^{p-1} + \bar{h})x\right). 
			\end{align}
		\end{itemize} 
	\end{thm}

\begin{remark}
Note that by substituting $\bar{\beta}=0$ and $\bar{h}=0$ in Theorem \ref{cltun} one gets the limiting distribution of the average magnetization $\os$ at all fixed parameter values $(\beta, h) \in \Theta$, thus recovering the result in \cite[Theorem 2.1]{cmp}. The outline of the proof of Theorem \ref{cltun} is similar to that of \cite[Theorem 2.1]{cmp}, however, the calculations now are more delicate because the parameters depend on $N$. The details are given in Appendix \ref{perturbed_asm}. 
\end{remark}

We now show how Theorem \ref{cltun} can be used to derive the asymptotic distribution of the ML estimates $\hat{\beta}_N$ and $\hat{h}_N$. The rest of this section is organized as follows. The asymptotic distributions in the $p$-regular and $p$-special cases (which include Theorems \ref{cltintr3_1}, \ref{cltintr3_2}, \ref{thmmle1}, and \ref{thmmle2}) are given in Section \ref{sec:pf_beta_h_1}. The asymptotic distribution of 
	$\hat{h}_N$ in the $p$-critical case (Theorem \ref{cltintr3_III}) is given Section \ref{sec:pf_beta_h_2}. Finally, the results for $\hat{\beta}_N$ in the $p$-critical case (Theorem \ref{thmmle_III}) are proved in Section \ref{sec:pf_beta_III}.
	
	\subsection{Proofs of Theorems \ref{cltintr3_1}, \ref{cltintr3_2}, \ref{thmmle1} and \ref{thmmle2}} 
	\label{sec:pf_beta_h_1} 
	
	We will only prove the case $(\beta,h)$ is $p$-regular, which includes  Theorems \ref{cltintr3_1} and \ref{thmmle1}. The proofs for the $p$-special case, that is, Theorems \ref{cltintr3_2} and \ref{thmmle2}, follow similarly from part (3) of Theorem \ref{cltun}.

	\subsubsection{Proof of Theorem \ref{cltintr3_1}}

	For any $t \in \mathbb{R}$, we have by \eqref{eqmleh}, Lemma \ref{increasing}, and Theorem \ref{cltun}, together with the fact that pointwise convergence of moment generating functions on $\mathbb{R}$ imply convergence of moments, 
	\begin{align}\label{trclt2}
		\p_{\beta,h,p}  \left(N^{\frac{1}{2}}(\hat{h}_N -h) \leq t\right) &= \p_{\beta,h,p} \left(\hat{h}_N  \leq h + \frac{t}{N^{\frac{1}{2}}}\right)\nonumber\\
		&= \p_{\beta,h,p} \left(u_{N,1}(\beta,\hat{h}_N ,p) \leq u_{N,1}\left(\beta, h + \frac{t}{N^{\frac{1}{2}}},p\right) \right)\nonumber\\
		&= \p_{\beta,h,p} \left(\os \leq \e_{\beta, h+ N^{-\frac{1}{2}} t,p}  (\os) \right)\nonumber\\
		&= \p_{\beta,h,p} \left(N^{\frac{1}{2}}(\os-m_*) \leq \e_{\beta, h+ N^{-\frac{1}{2}} t,p}  (N^{\frac{1}{2}}(\os-m_*)) \right)\nonumber\\
		& \rightarrow \p_{\beta,h,p} \left(N\left(0, -\frac{1}{H''(m_*)}\right)\leq -\frac{t}{H''(m_*)} \right)\nonumber\\
		&= \p_{\beta,h,p} \left(N\left(0, -H''(m_*)\right)\leq t \right).
	\end{align}
	Now, the proof of Theorem \ref{cltintr3_1} follows from \eqref{trclt2}.

	\subsubsection{Proof of Theorem \ref{thmmle1}}
	
	We begin with the case $m_* \ne 0$. By Theorem \ref{cltun}, $\left(\os - m_*\right)^s =  O_{P}(N^{-\frac{s}{2}}) = O_{P}(N^{-1})$, for every $s \geq 2$ under  $\p = \p_{\beta+ \bar{\beta}/\sqrt N, h, p}$. Further, since pointwise convergence of moment generating functions on $\mathbb{R}$ imply convergence of moments, we also have $\e_{\beta+ \bar{\beta}/\sqrt N, h, p}(\os - m_*)^s  = O(N^{-1})$, for every $s \geq 2$. Now,
	\begin{equation}\label{binoexp}
		N^{\frac{1}{2}}(\overline{\sigma}^p_N - m_*^p) =N^{\frac{1}{2}}pm_*^{p-1} (\os-m_*) + N^{\frac{1}{2}}\sum_{s=2}^p \binom{p}{s} m_*^{p-s}(\os-m_*)^s.
	\end{equation}
	It follows from Theorem \ref{cltun} and \eqref{binoexp} that under $\p_{\beta+{N^{-\frac{1}{2}}\bar{\beta}},h,p} $,
	\begin{equation}\label{mpl1}
		N^{\frac{1}{2}}(\overline{\sigma}^p_N - m_*^p) \xrightarrow{D} N\left(-\frac{\bar{\beta}p^2m_*^{2p-2}}{H''(m_*)},-\frac{p^2m_*^{2p-2}}{H''(m_*)}\right) ,
	\end{equation} 
	and 
	\begin{equation}\label{mpl1_N}
		\e\left[N^{\frac{1}{2}}(\overline{\sigma}^p_N - m_*^p)\right] \rightarrow -\frac{\bar{\beta}p^2m_*^{2p-2}}{H''(m_*)}.
	\end{equation} 
	Now, note that for any $t \in \mathbb{R}$, we have by \eqref{eqmle} and the monotonicity of the function $u_{N, p}(\cdot, h, p)$ (Lemma \ref{increasing}), we have
	\begin{align}\label{trclt}
		\p_{\beta,h,p}  \left(N^{\frac{1}{2}}(\hat{\beta}_N-\beta) \leq t\right) &= \p_{\beta,h,p} \left(\hat{\beta}_N \leq \beta + \frac{t}{N^{\frac{1}{2}}}\right)\nonumber\\
		&= \p_{\beta,h,p} \left(u_{N,p}(\hat{\beta}_N,h,p) \leq u_{N,p}\left(\beta + \frac{t}{N^{\frac{1}{2}}}, h,p\right) \right)\nonumber\\
		&= \p_{\beta,h,p} \left(\overline{\sigma}^p_N \leq \e_{\beta + N^{-\frac{1}{2}} t, h,p}  (\overline{\sigma}^p_N) \right)\nonumber\\ 
		&= \p_{\beta,h,p} \left(N^{\frac{1}{2}}(\overline{\sigma}^p_N-m_*^p) \leq \e_{\beta + N^{-\frac{1}{2}} t, h,p}  (N^{\frac{1}{2}}(\overline{\sigma}^p_N-m_*^p)) \right).
	\end{align} 
	Now, weak convergence to a continuous distribution implies uniform convergence of the distribution functions, by \eqref{mpl1}, \eqref{mpl1_N}, and \eqref{trclt}, it follows that under $\p_{\beta,h,p} $, 
	\begin{align*}
		\p_{\beta,h,p}  \left(N^{\frac{1}{2}}(\hat{\beta}_N-\beta) \leq t\right) &\rightarrow \p_{\beta,h,p} \left(N\left(0,-\frac{p^2m_*^{2p-2}}{H''(m_*)}\right) \leq -\frac{tp^2m_*^{2p-2}}{H''(m_*)} \right)  \nonumber \\ 
		& = \p_{\beta,h,p}\left(N\left(0, -\frac{H''(m_*)}{p^2m_*^{2p-2}}\right)\leq t \right) .
	\end{align*}
	This completes the proof of \eqref{eq:bmle_m_1}. 
	
	Next, we consider the case $m_* = 0$. This implies that $\sup_{x \in [-1,1]} H_{\beta,h,p}(x) = 0$. Hence, by part (1) of \cite[Lemma B.1]{cmp}, $h=0$, and then, \eqref{eq:betatilde} implies $\beta \leq \tilde{\beta}_p$. However, the point $(\tilde{\beta}_p,0)$ is $p$-critical, and hence, we must have $\beta < \tilde{\beta}_p$. Now, for every $t \in \mathbb{R}$,we have by \eqref{eqmle} and Lemma \ref{increasing}, 
	\begin{equation}\label{fct1}
		\p_{\beta,0,p}\left(\hat{\beta}_N > t\right) = \p_{\beta,0,p}\left({\left(N^{\frac{1}{2}}\os\right)}^p > N^{\frac{p}{2}}u_{N,p}(t,0,p)\right).
	\end{equation}
	First, fix $t \in (\tilde{\beta}_p,\infty)$ and note that:
	\begin{equation}\label{adnlst1}
		u_{N,p}(t,0,p) = \frac{1}{N} \frac{\partial}{\partial \underline{\beta}}F_N(\underline{\beta},0,p)\Big|_{\underline{\beta} = t}. 
	\end{equation}
	Now, by the mean value theorem and the fact that $F_N(0,0,p) = 0$, we have:
	\begin{equation}\label{adnlstp2}
		F_N(t,0,p) = t\frac{\partial}{\partial \underline{\beta}}F_N(\underline{\beta},0,p)\Big|_{\underline{\beta} = \xi}
	\end{equation}
	for some $\xi \in (0,t)$. By Lemma \ref{increasing}, we have:
	\begin{equation}\label{adnlstp3}
		\frac{\partial}{\partial \underline{\beta}}F_N(\underline{\beta},0,p)\Big|_{\underline{\beta} = \xi} \leq \frac{\partial}{\partial \underline{\beta}}F_N(\underline{\beta},0,p)\Big|_{\underline{\beta} = t}. 
	\end{equation}
	Combining \eqref{adnlst1}, \eqref{adnlstp2} and \eqref{adnlstp3}, we have:
	\begin{equation}\label{adnlstp4}
		u_{N,p}(t,0,p) \geq t^{-1} N^{-1} F_N(t,0,p). 
	\end{equation}
	Now, \eqref{logpartex} in  Lemma \ref{ex} (for odd $p$) and \eqref{eq:multi_max_II} in Lemma \ref{lm:condpart} (for even $p$) implies that\footnote{For two positive sequences $\{a_n\}_{n \geq 1}$ and $\{b_n\}_{n \geq 1}$, $a_n = \Omega(b_n)$, if there exists a positive constant $C$, such that $a_n \geq C b_n$, for all large $n$.} $$N^{-1} F_N(t,0,p) = \Omega(1).$$
	This, together with \eqref{adnlstp4} implies that:
	\begin{equation}\label{udifc}
		u_{N,p}(t,0,p) = \Omega(1).
	\end{equation}
	Since, by Theorem \ref{cltun}, $N^{1/2}\os \xrightarrow{D} N(0,1)$ under $\p_{\beta,0,p}$,   \eqref{fct1} and \eqref{udifc} implies, as $N \rightarrow \infty$,
	\begin{equation}\label{tgeqbet}
		\p_{\beta,0,p}\left(\hat{\beta}_N > t\right) \rightarrow 0.
	\end{equation}
	Next, fix $t \in [0,\tilde{\beta}_p)$. Since we have pointwise convergence of moment generating functions in part (1) of Theorem \ref{cltun}, we get:
	\begin{equation}\label{adl}
		N^{\frac{p}{2}}u_{N,p}(t,0,p) = \e_{t,0,p}[(N^{\frac{1}{2}}\os)^p] \rightarrow \e Z^p. 
	\end{equation}
	Hence by \eqref{fct1},	
	\begin{equation}\label{tgeqbet1}
		\p_{\beta,0,p}\left(\hat{\beta}_N \leq t\right) \rightarrow \gamma_p.
	\end{equation}
	Finally, fix $t\in (-\infty,0)$. If $p$ is odd, the function $\beta \mapsto F_N(\beta,0,p)$ becomes an even function (recall \eqref{ptnepn}), and hence, its partial derivative with respect to $\beta$ becomes an odd function. Consequently, $$u_{N,p}(t,0,p) = -u_{N,p}(-t,0,p). $$ Now, if $t <  -\tilde{\beta}_p$, then $-t \in ( \tilde{\beta}_p,\infty)$, so by \eqref{udifc}, $N^{\frac{p}{2}}u_{N,p}(-t,0,p)$ converges to $\infty$ , i.e. $$\lim_{N\rightarrow\infty} N^{\frac{p}{2}}u_{N,p}(t,0,p) = -\infty. $$ If $t > -\tilde{\beta}_p$, then $-t \in (0,\tilde{\beta}_p)$, and hence, by \eqref{adl} (note that $\mathbb{E} Z^p = 0$ when $p$ is odd)
	$$\lim_{N\rightarrow\infty} N^{\frac{p}{2}}u_{N,p}(t,0,p) = -\lim_{N\rightarrow\infty} N^{\frac{p}{2}}u_{N,p}(-t,0,p) = 0. $$                                
	Hence, we have from \eqref{fct1}, as $N \rightarrow \infty$,
	
	\[   
	\p_{\beta,0,p}\left(\hat{\beta}_N \leq t\right) \rightarrow 
	\begin{cases}
		0 &\quad\text{if}~t < -\tilde{\beta}_p,\\
		\frac{1}{2} &\quad\text{if}~t > -\tilde{\beta}_p\\
	\end{cases}
	\]
	This, combined with \eqref{tgeqbet} and \eqref{tgeqbet1}, shows that $\hat{\beta}_N \xrightarrow{D} \frac{1}{2}\delta_{\tilde{\beta}_p} + \frac{1}{2} \delta_{-\tilde{\beta}_p}$ if $p$ is odd.
	
	Now, assume that $p \geq 4$ is even. Then, for $t\in (-\infty,0)$, 
	\begin{align}\label{myt1}
		\bigg|N^{\frac{p}{2}-1} \frac{\partial}{\partial \underline \beta} F_N(\underline \beta,0,p)\Big|_{\underline \beta = t} - & \mathbb{E}_{0,0,p} \left[N^{\frac{p}{2}}\overline{\sigma}^p_N\right]\bigg|   \nonumber\\ 
		&= N^{\frac{p}{2}-1}\left\{\frac{\partial}{\partial \underline \beta} F_N(\underline \beta,0,p)\Big|_{\underline \beta = 0} - \frac{\partial}{\partial \underline \beta} F_N(\underline \beta,0,p)\Big|_{\underline \beta = t}\right \} \nonumber\\
		&   \leq -tN^{\frac{p}{2}-1} \sup_{\zeta \in [t,0]} \frac{\partial^2}{\partial \underline \beta^2} F_N(\underline \beta, 0, p)\Big|_{\underline \beta=\zeta}  \nonumber\\
		&=-tN^{\frac{p}{2}-1} \sup_{\zeta \in [t,0]} \mathrm{Var}_{\zeta,0,p} \left(N\overline{\sigma}^p_N\right)\nonumber\\ 
		&= -tN^{1-\frac{p}{2}}\sup_{\zeta \in [t,0]} \mathrm{Var}_{\zeta,0,p} \left(N^{\frac{p}{2}}\overline{\sigma}^p_N\right)\nonumber\\
		& \leq -tN^{1-\frac{p}{2}}\sup_{\zeta \in [t,0]} \mathbb{E}_{\zeta,0,p}  (N^{p}\overline{\sigma}^{2p}_N) \nonumber\\
		&= -t N^{1-\frac{p}{2}}\sup_{\zeta \in [t,0]} \mathbb{E}_{0,0,p}  \left[N^{p}\overline{\sigma}^{2p}_N e^{\zeta N\overline{\sigma}^p_N-F_N(\zeta,0,p)}\right].
	\end{align}
	Next, for every $\zeta \in [t,0]$, since the map $\beta \mapsto \frac{\partial}{\partial \beta} F_N(\beta,0,p)$ is increasing,
	$$-F_N(\zeta,0,p) \leq -\zeta \frac{\partial}{\partial \underline \beta} F_N(\underline \beta,0,p)\Big|_{\underline \beta=0} \leq -t N\mathbb{E}_{0,0,p} \overline{\sigma}^p_N = o(1)\implies \sup_{N \geq 1}\sup_{\zeta \in [t,0]} e^{-F_N(\zeta,0,p)} := B < \infty.$$ We thus have from \eqref{myt1},
	\begin{equation*}
		\left|N^{\frac{p}{2}-1} \frac{\partial}{\partial \underline \beta} F_N(\underline \beta,0,p)\Big|_{\underline \beta = t} - \mathbb{E}_{0,0,p} \left[N^{\frac{p}{2}}\overline{\sigma}^p_N\right]\right| \leq -tBN^{1-\frac{p}{2}} \mathbb{E}_{0,0,p}\left[N^p\overline{\sigma}^{2p}_N\right] = o(1).
	\end{equation*}
	Hence, $N^{\frac{p}{2}} u_{N,p}(t,0,p) = N^{\frac{p}{2}-1} \frac{\partial}{\partial \tilde{\beta}} F_N(\tilde{\beta},0,p)\Big|_{\tilde{\beta} = t} \rightarrow \e Z^p$ as $N \rightarrow \infty$. Consequently, as $N \rightarrow \infty$,
	\begin{equation}\label{lstpeven}
		\p_{\beta,0,p}\left(\hat{\beta}_N \leq t\right) \rightarrow \gamma_p.
	\end{equation}
	We conclude from \eqref{tgeqbet}, \eqref{tgeqbet1} and \eqref{lstpeven}, that $\hat{\beta}_N \xrightarrow{D} \gamma_p \delta_{-\infty} + (1-\gamma_p) \delta_{\tilde{\beta}_p}$ if $p$ is even. This completes the proof of \eqref{eq:bmle_m_2}. \qed

	\begin{remark}[Efficiency of the ML estimates at $p$-regular points]\label{remark1} 
		An interesting consequence of the results proved above is that, at the $p$-regular points, the  limiting variance of the ML estimates equals the limiting inverse Fisher information, that is, the ML estimates are asymptotically efficient. To see this, note that the Fisher information of $\beta$ and $h$ (scaled by $N$) in the model \eqref{cwwithmg} are given by 
		$$I_N(\beta) = \frac{1}{N} \e_{\beta,h,p}\left[\left(\frac{\partial}{\partial \beta} \log \p_{\beta,h,p}(\bs)\right)^2\right] = \mathrm{Var}_{\beta,h,p}(N^{\frac{1}{2}} \os^p)$$ 
		and 
		$$I_N(h) = \frac{1}{N} \e_{\beta,h,p}\left[\left(\frac{\partial}{\partial h} \log \p_{\beta,h,p}(\bs)\right)^2\right] = \mathrm{Var}_{\beta,h,p}(N^{\frac{1}{2}} \os),$$ 
		respectively. It follows from the proof of Theorem \ref{cltun}, that for a $p$-regular point $(\beta,h)$, the moment generating of $\sqrt{N}\left(\os- m_*\right)$ converges pointwise to that of the centered Gaussian distribution with variance $-\left[H''(m_*)\right]^{-1}$. Hence, 
		\begin{equation*}
			\lim_{N \rightarrow \infty}I_N(h) =  -\left[H''(m_*)\right]^{-1}~.
		\end{equation*} 
		Also, it follows from \eqref{binoexp} and \eqref{mpl1} and the fact $\e_{\beta,h,p}\left[(\os- m_*)^s\right] = O(N^{-s/2})$, for each $s \geq 1$, that
		\begin{equation*}
			\lim_{N \rightarrow \infty}I_N(\beta) =  -\frac{p^2 m_*^{2p-2}}{H''(m_*)}~.
		\end{equation*} 
		Therefore, by Theorem \ref{cltintr3_1}, at a $p$-regular point $(\beta,h)$, $\hat{h}_N$ is an efficient estimator of $h$, and by Theorem \ref{thmmle1}, if $(\beta,h)$ is a $p$-regular point with $m_*\neq 0$, then $\hat{\beta}_N$ is an efficient estimator of $\beta$.
	\end{remark}

	\subsection{Proof of Theorem \ref{cltintr3_III}}
	\label{sec:pf_beta_h_2} 
	
	Recall the definitions of the sets $A_k~(1\leq k\leq K)$ from the proof of  Lemma \ref{lem:multiple_max}. Now, fixing $t <0$, we have similar to the proof of \eqref{eq:h_1},
	\begin{align*}
		\p_{\beta,h,p}  \left(N^{\frac{1}{2}}(\hat{h}_N -h) \leq t\right)  &= \p_{\beta,h,p} \left(N^{\frac{1}{2}}(\os-m_1) \leq \e_{\beta, h+ N^{-\frac{1}{2}} t,p}  (N^{\frac{1}{2}}(\os-m_1)) \right) \nonumber\\
		& = T_1 + T_2, 
	\end{align*}
	where 
	\begin{align*}
		T_1&= \p_{\beta,h,p} \left(N^{\frac{1}{2}}(\os-m_1) \leq \e_{\beta, h+ N^{-\frac{1}{2}} t,p}  (N^{\frac{1}{2}}(\os-m_1))\Big| \os \in A_1\right) \p_{\beta,h,p} (\os \in A_1) , \\
		T_2 & = \p_{\beta,h,p} \left(N^{\frac{1}{2}}(\os-m_1) \leq \e_{\beta, h+ N^{-\frac{1}{2}} t,p}  (N^{\frac{1}{2}}(\os-m_1))\Big| \os \in A_1^c\right) \p_{\beta,h,p} (\os \in A_1^c). 
	\end{align*} 
	Now, by the law of iterated expectations, we have
	\begin{align}\label{lie1}
		\e_{\beta, h+ N^{-\frac{1}{2}} t,p}  (N^{\frac{1}{2}}(\os-m_1))  & =  S_1 + S_2 , 
	\end{align}
	where 
	\begin{align}\label{lie2}
		S_1:= \e_{\beta, h+ N^{-\frac{1}{2}} t,p}  \left(N^{\frac{1}{2}}(\os-m_1)\Big|\os \in A_1\right)  \p_{\beta, h+ N^{-\frac{1}{2}} t,p} (\os \in A_1) 
	\end{align}
	and
	\begin{align}\label{lie3}
		S_2:= \e_{\beta, h+ N^{-\frac{1}{2}} t,p}  \left(N^{\frac{1}{2}}(\os-m_1)\Big|\os \in A_1^c\right) \p_{\beta, h+ N^{-\frac{1}{2}} t,p} (\os \in A_1^c). 
	\end{align}  
	Note that by \eqref{moment generating functionnonun}, $$ \e_{\beta, h+ N^{-\frac{1}{2}} t,p} \left(N^{\frac{1}{2}}(\os-m_1)\Big|\os \in A_1\right) \rightarrow - \frac{t}{H''(m_1)},$$ 
	as $N \rightarrow \infty$. Also, by Lemma \ref{redunsup}, $\p_{\beta, h+ t/\sqrt N,p} (\os \in A_1^c) \leq C_1 e^{-C_2 N^{\frac{1}{2}}}$ for positive constants $C_1,C_2$ not depending on $N$. Hence, \eqref{lie2} converges to $-t/H''(m_1)$ and \eqref{lie3} converges to $0$. Consequently, \eqref{lie1} converges to $-t/H''(m_1)$. 
	
	Next, under $\p_{\beta, h,p}(~\cdot~|\os \in A_1^c)$, $N^{\frac{1}{2}}(\os - m_1) \xrightarrow{P} \infty,$ by Lemma \ref{lem:multiple_max}. Hence, $T_2 \rightarrow 0$. Also, by Theorem \ref{cltun}, $N^{\frac{1}{2}}(\os - m_1) \xrightarrow{D} N\left(0,-1/H''(m_1)\right)$ under $\p_{\beta,h,p} (~\cdot~|\os \in A_1)$. Hence, $T_1$ converges to
	$$p_1\p(N\left(0,-\frac{1}{H''(m_1)}) \leq -\frac{t}{H''(m_1)}\right) = p_1\p\left(N(0,-H''(m_1)) \leq t\right).$$ Hence,  
	\begin{equation}\label{htl0}
		\p_{\beta,h,p}  \left(N^{\frac{1}{2}}(\hat{h}_N -h) \leq t\right)  \rightarrow p_1\p\left(N(0,-H''(m_1)) \leq t\right)\quad\quad \textrm{for all}\quad t<0.
	\end{equation}
	Next, fix $t >0$, whence we have
	\begin{equation*}
		\p_{\beta,h,p}  \left(N^{\frac{1}{2}}(\hat{h}_N -h) > t\right) = T_3+ T_4,
	\end{equation*}
	where 
	\begin{align*}
		T_3&= \p_{\beta,h,p} \left(N^{\frac{1}{2}}(\os-m_K) > \e_{\beta, h+ N^{-\frac{1}{2}} t,p}  (N^{\frac{1}{2}}(\os-m_K))\Big| \os \in A_K\right) \p_{\beta,h,p} (\os \in A_K), \\
		T_4 & = \p_{\beta,h,p} \left(N^{\frac{1}{2}}(\os-m_K) > \e_{\beta, h+ N^{-\frac{1}{2}} t,p}  (N^{\frac{1}{2}}(\os-m_K))\Big| \os \in A_K^c\right) \p_{\beta,h,p} (\os \in A_K^c).
	\end{align*} 
	By the same arguments as before, it follows that $$\e_{\beta, h+  N^{-\frac{1}{2}} t, p}  (N^{\frac{1}{2}}(\os-m_K)) \rightarrow -\frac{t}{H''(m_K)}.$$ Next, under $\p_{\beta, h,p} (~\cdot~|\os \in A_K^c )$, $N^{\frac{1}{2}}(\os - m_K) \xrightarrow{P} -\infty$ by Lemma \ref{lem:multiple_max}. Hence, $T_4\rightarrow 0$. Also, by Theorem \ref{cltun} (2), $N^{\frac{1}{2}}(\os - m_K) \xrightarrow{D} N(0,-1/H''(m_K))$ under $\p_{\beta,h,p} (~\cdot~|\os \in A_K)$. Hence, $T_3$ converges to
	$$p_K\p\left(N\left(0,-\frac{1}{H''(m_K)}\right) > -\frac{t}{H''(m_K)}\right) = p_K\p\left(N(0,-H''(m_K)) > t\right).$$ Hence,
	\begin{equation}\label{htl1}
		\p_{\beta,h,p}  \left(N^{\frac{1}{2}}(\hat{h}_N -h) > t\right) \rightarrow p_K\p\left(N(0,-H''(m_K)) > t\right)\quad\quad \textrm{for all}\quad t>0.
	\end{equation}
	Combining \eqref{htl0} and \eqref{htl1}, we conclude that for all $p$-critical points $(\beta,h)$, under $\p_{\beta,h,p} $,
	\begin{equation}\label{htl3}
		N^{\frac{1}{2}}(\hat{h}_N -h) \xrightarrow{D} \tfrac{p_1}{2} N^{-}(0,-H''(m_1)) + \tfrac{p_K}{2} N^+(0,-H''(m_K)) + \left(1-\frac{p_1+p_K}{2}\right)\delta_0.	
	\end{equation}
	Theorem \ref{cltintr3_III} follows from \eqref{htl3} on observing that if $p\geq 4$ is even and $(\beta,h) = (\tilde{\beta}_p,0)$, then $K=3$, $m_3 = -m_1$ and $p_1=p_3$, and otherwise, $K=2$.
	
	\subsection{Proof of Theorem \ref{thmmle_III}} 
	\label{sec:pf_beta_III} 
	
	We first deal with the case $p \geq 3$ is odd.
	
	\noindent\textit{Proof of \eqref{eq:bmle_multimax_1}:} In this case, $0$ is not a global maximizer of $H_{\beta,p,h}$. Fixing $t < 0$, we have
	\begin{align*}
		\p_{\beta,h,p}  \left(N^{\frac{1}{2}}(\hat{\beta}_N-\beta) \leq t\right) &= \p_{\beta,h,p} \left(N^{\frac{1}{2}}(\overline{\sigma}^p_N-m_1^p) \leq \e_{\beta+N^{-\frac{1}{2}} t, h ,p}  (N^{\frac{1}{2}}(\overline{\sigma}^p_N-m_1^p)) \right)\nonumber\\&= T_5+T_6,
	\end{align*}
	where
	\begin{align*}
		T_5&= \p_{\beta,h,p} \left(N^{\frac{1}{2}}(\overline{\sigma}^p_N-m_1^p) \leq \e_{\beta+N^{-\frac{1}{2}} t, h ,p}  (N^{\frac{1}{2}}(\overline{\sigma}^p_N-m_1^p))\Big| \os \in A_1\right)\p_{\beta,h,p} (\os \in A_1), \\
		T_6 & = \p_{\beta,h,p} \left(N^{\frac{1}{2}}(\overline{\sigma}^p_N-m_1^p) \leq \e_{\beta+N^{-\frac{1}{2}} t, h ,p}  (N^{\frac{1}{2}}(\overline{\sigma}^p_N-m_1^p))\Big| \os \in A_1^c\right)\p_{\beta,h,p} (\os \in A_1^c). 
	\end{align*} 
	Now, by the law of iterated expectations, we have
	\begin{align}
		\e_{\beta+N^{-\frac{1}{2}} t, h ,p}  (N^{\frac{1}{2}}(\overline{\sigma}^p_N-m_1^p)) = S_3 + S_4, \label{lie11} 
	\end{align} 
	where 
	\begin{align} 
		S_3:=\e_{\beta+N^{-\frac{1}{2}} t, h,p}  \left(N^{\frac{1}{2}}(\overline{\sigma}^p_N-m_1^p)\Big|\os \in A_1\right) \p_{\beta+ N^{-\frac{1}{2}} t, h,p} (\os \in A_1)\label{lie21}
	\end{align} 
	and 
	\begin{align} 
		S_4:=\e_{\beta+N^{-\frac{1}{2}} t, h,p}  \left(N^{\frac{1}{2}}(\overline{\sigma}^p_N-m_1^p)\Big|\os \in A_1^c\right) \p_{\beta+ N^{-\frac{1}{2}} t, h,p} (\os \in A_1^c).\label{lie31}
	\end{align} 
	
	\noindent From Theorem \ref{cltun} and a simple binomial expansion (see \eqref{binoexp}), it follows that
	\begin{equation*}
		\e_{\beta+N^{-\frac{1}{2}} t, h,p}  \left(N^{\frac{1}{2}}(\overline{\sigma}^p_N-m_1^p)\Big|\os \in A_1\right) \rightarrow -\frac{tp^2m_1^{2p-2}}{H''(m_1)},
	\end{equation*}
	and under $\p_{\beta,h,p} (~\cdot~|\os \in A_1)$,
	\begin{equation}\label{eqel2}
		N^{\frac{1}{2}}(\overline{\sigma}^p_N - m_1^p) \xrightarrow{D} N\left(0,-\frac{p^2m_1^{2p-2}}{H''(m_1)}\right). 
	\end{equation}
	
	\noindent By Lemma \ref{redunsup2}, $\p_{\beta+ t/\sqrt N, h,p} (\os \in A_1^c) \leq C_1 e^{-C_2 \sqrt N}$ for positive constants $C_1,C_2$ not depending on $N$. Hence, \eqref{lie21} converges to $-tp^2m_1^{2p-2}/H''(m_1)$ and \eqref{lie31} converges to $0$. Consequently, \eqref{lie11} converges to $-tp^2m_1^{2p-2}/H''(m_1)$.  
	
	Next, under $\p_{\beta, h,p} (~\cdot~|\os \in A_1^c )$, $N^{\frac{1}{2}}(\overline{\sigma}^p_N - m_1^p) \xrightarrow{P} \infty$ by Lemma \ref{lem:multiple_max}. Hence, $T_6 \rightarrow 0$. Then, by \eqref{eqel2}, $T_5$ converges to
	$$p_1\p\left(N\left(0,-\frac{p^2m_1^{2p-2}}{H''(m_1)}\right) \leq -\frac{tp^2m_1^{2p-2}}{H''(m_1)}\right) = p_1\p\left(N\left(0,-\frac{H''(m_1)}{p^2 m_1^{2p-2}}\right) \leq t\right).$$ Hence, 
	\begin{equation}\label{htl03}
		\p_{\beta,h,p}  \left(N^{\frac{1}{2}}(\hat{\beta}_N-\beta) \leq t\right) \rightarrow p_1\p\left(N\left(0,-\frac{H''(m_1)}{p^2 m_1^{2p-2}}\right)\leq t\right)\quad\quad \textrm{for all}\quad t<0.
	\end{equation}
	Next, fix $t >0$, whence we have
	\begin{equation*}
		\p_{\beta,h,p}  \left(N^{\frac{1}{2}}(\hat{\beta}_N-\beta) > t\right) = T_7 + T_8,
	\end{equation*}
	where
	\begin{align*}
		T_7&= \p_{\beta,h,p} \left(N^{\frac{1}{2}}(\overline{\sigma}^p_N-m_2^p) > \e_{\beta+ N^{-\frac{1}{2}} t, h,p}  (N^{\frac{1}{2}}(\overline{\sigma}^p_N-m_2^p))\Big| \os \in A_2\right)\p_{\beta,h,p} (\os \in A_2),\\
		T_8 & = \p_{\beta,h,p} \left(N^{\frac{1}{2}}(\overline{\sigma}^p_N-m_2^p) > \e_{\beta+ N^{-\frac{1}{2}} t, h,p}  (N^{\frac{1}{2}}(\overline{\sigma}^p_N-m_2^p))\Big| \os \in A_2^c\right)\p_{\beta,h,p} (\os \in A_2^c). 
	\end{align*}
	By the same arguments as before, it follows that $$\e_{\beta+ N^{-\frac{1}{2}} t, h,p}  (N^{\frac{1}{2}}(\overline{\sigma}^p_N-m_2^p)) \rightarrow -\frac{tp^2 m_2^{2p-2}}{H''(m_2)}.$$ Next, under $\p_{\beta, h,p} (~\cdot~\big|\os \in A_2^c )$, $N^{\frac{1}{2}}(\overline{\sigma}^p_N - m_2^p) \xrightarrow{P} -\infty$ by Lemma \ref{lem:multiple_max}. Hence, $T_8 \rightarrow 0$. Also, we know that  $N^{\frac{1}{2}}(\overline{\sigma}^p_N - m_2^p) \xrightarrow{D} N\left(0,-p^2m_2^{2p-2}/H''(m_2)\right)$ under $\p_{\beta,h,p} \left(~\cdot~\big|\os \in A_2\right)$. Hence, $T_7$ converges to
	$$p_2\p\left(N\left(0,-\frac{p^2m_2^{2p-2}}{H''(m_2)}\right) > -\frac{tp^2m_2^{2p-2}}{H''(m_2)}\right) = p_2\p\left(N\left(0,-\frac{H''(m_2)}{p^2 m_2^{2p-2}}\right) > t\right).$$ Hence,
	\begin{equation}\label{htl13}
		\p_{\beta,h,p}  \left(N^{\frac{1}{2}}(\hat{\beta}_N-\beta) > t\right) \rightarrow p_2\p\left(N\left(0,-\frac{H''(m_2)}{p^2 m_2^{2p-2}}\right) > t\right), \quad \textrm{for all } t>0.
	\end{equation}
	Combining \eqref{htl03} and \eqref{htl13}, we conclude that if $p \geq 3$ is odd, then for all $p$-critical points $(\beta,h)$, under $\p_{\beta,h,p} $,
	\begin{equation}\label{htl33}
		N^{\frac{1}{2}}(\hat{\beta}_N-\beta) \xrightarrow{D} \frac{p_1}{2} N^{-}\left(0,-\frac{H''(m_1)}{p^2 m_1^{2p-2}}\right) + \frac{p_2}{2} N^+\left(0,-\frac{H''(m_2)}{p^2 m_2^{2p-2}}\right) + \left(1-\frac{p_1+p_2}{2}\right)\delta_0.	
	\end{equation}
	\eqref{eq:bmle_multimax_1} now follows from \eqref{htl33} on observing that $p_2=1-p_1$. \\ 
	
	\medskip
	
	\noindent \textit{Proof of \eqref{eq:bmle_multimax_2}:} In this case, $m_1 = 0$. We can write for any $t < 0$,
	\begin{align}
		& \p_{\beta,h,p}  \left(N^{\frac{1}{2}}(\hat{\beta}_N-\beta) \leq t\right)\nonumber\\& =  \p_{\tilde{\beta}_p,0,p} \left(N^\frac{p}{2}\overline{\sigma}^p_N \leq \e_{\tilde{\beta}_p+N^{-\frac{1}{2}} t, 0 ,p}  (N^\frac{p}{2}\overline{\sigma}^p_N)\Big| \os \in A_1\right)\p_{\tilde{\beta}_p,0,p} (\os \in A_1)\label{m111}\\&+  \p_{\tilde{\beta}_p,0,p} \left(N^\frac{p}{2}\overline{\sigma}^p_N \leq \e_{\tilde{\beta}_p+N^{-\frac{1}{2}} t, 0 ,p}  (N^\frac{p}{2}\overline{\sigma}^p_N)\Big| \os \in A_1^c\right)\p_{\tilde{\beta}_p,0,p} (\os \in A_1^c)\label{m211}.
	\end{align}
	By Theorem \ref{cltun} under both $\p_{\tilde{\beta}_p,0,p} (~\cdot~\big|\os \in A_1)$ and $\p_{\tilde{\beta}_p + t/\sqrt N, 0,p} (~\cdot~\big|\os \in A_1)$, $N^\frac{p}{2} \overline{\sigma}^p_N$ converges to $Z^p$ in distribution and in moments, where $Z \sim N(0,1)$. Consequently, $\e_{\tilde{\beta}_p+ t/\sqrt N, 0 ,p}  (N^\frac{p}{2}\overline{\sigma}^p_N) \rightarrow 0$ by arguments similar to before, since $\p_{\tilde{\beta}_p+  t/\sqrt N, 0 ,p} (\os \in A_1^c)$ decays to $0$ exponentially fast. Hence, \eqref{m111} converges to $p_1/2$. Also, under $\p_{\tilde{\beta}_p,0,p} (~\cdot~\big| \os \in A_1^c )$, $N^\frac{p}{2}\overline{\sigma}^p_N \xrightarrow{P} \infty$ and hence, \eqref{m211} converges to $0$. This shows that for all $t < 0$,
	\begin{equation}\label{consprob}
		\p_{\beta,h,p}  \left(N^{\frac{1}{2}}(\hat{\beta}_N-\beta) \leq t\right) \rightarrow \frac{p_1}{2}.
	\end{equation}
	Of course, \eqref{htl13} still remains valid. \eqref{eq:bmle_multimax_2} now follows from \eqref{htl13} and \eqref{consprob}.
	
	Now, assume that $p \geq 4$ is even. If $h \neq 0$, then $K=2$. Also, $m_1<m_2<0$ if $h <0$ and $0<m_1<m_2$ if $h> 0$. Hence, $m_1^p < m_2^p$ if $h >0$ and $m_1^p > m_2^p$ if $h<0$. We can now use Lemma \ref{redunsup3} to derive \eqref{eq:bmle_multimax_3} and \eqref{eq:bmle_multimax_4}, and the proof is so similar to that for the $p \geq 3$ odd case, that we skip it. We now prove \eqref{eq:bmle_multimax_5} and \eqref{eq:bmle_multimax_6}. \\ 
	
	\noindent\textit{Proof of \eqref{eq:bmle_multimax_5}:}~By Theorem \ref{cltun} and a standard binomial expansion (see \eqref{binoexp}), it follows that for any $\bar{\beta} \in \mathbb{R}$ and $i \in \{1,2\}$, under the conditional measure $\p_{\beta+N^{-\frac{1}{2}}\bar{\beta}, 0,p}\left(~\cdot~\big|\os \in A_i\right)$,  
	\begin{equation}\label{iin12}
		N^{\frac{1}{2}}(\overline{\sigma}^p_N - m_*^p) \xrightarrow{D} N\left(-\frac{\bar{\beta}p^2 m_*^{2p-2}}{H''(m_*)}, -\frac{p^2 m_*^{2p-2}}{H''(m_*)}\right) \textrm{ and }  \e\left(N^{\frac{1}{2}}(\overline{\sigma}^p_N - m_*^p)\right) \rightarrow -\frac{\bar{\beta}p^2m_*^{2p-2}}{H''(m_*)}.
	\end{equation}
	Since $A_1 \sqcup A_2 = [-1,1]$, \eqref{iin12} also holds under the unconditional measure $\p_{\beta+\bar{\beta}/\sqrt N, 0,p}$. The result in \eqref{eq:bmle_multimax_5} now follows easily, since for every $t \in \mathbb{R}$, we have
	\begin{align*}
		\p_{\beta,0,p}\left(N^{\frac{1}{2}}(\hat{\beta}_N - \beta) \leq t\right) &= \p_{\beta,0,p}\left(N^{\frac{1}{2}}(\overline{\sigma}^p_N - m_*^p) \leq \e_{\beta+ N^{-\frac{1}{2}} t, 0,p} \left(N^{\frac{1}{2}}(\overline{\sigma}^p_N - m_*^p) \right)\right)\nonumber\\&\rightarrow \p\left(N\left(0, -\frac{p^2 m_*^{2p-2}}{H''(m_*)}\right) \leq -\frac{tp^2m_*^{2p-2}}{H''(m_*)}\right)\nonumber\\&= \p\left( N\left(0, -\frac{H''(m_*)}{p^2 m_*^{2p-2}}\right) \leq t\right).
	\end{align*}
	\noindent \textit{Proof of \eqref{eq:bmle_multimax_6}:}~Fix $t < 0$. By \eqref{cc3} in Lemma \ref{redunsup3} and Theorem \ref{cltun}, we have
	\begin{equation*}
		\e_{\beta + N^{-\frac{1}{2}}t,0,p} \left(N^\frac{p}{2}\overline{\sigma}^p_N\right) = \e_{\beta + N^{-\frac{1}{2}}t,0,p} \left(N^\frac{p}{2}\overline{\sigma}^p_N\Big| \os \in A_2\right)(1-o(1)) + o(1) = \e Z^p + o(1).
	\end{equation*}
	This, together with the fact that $N^\frac{p}{2} \overline{\sigma}^p_N \xrightarrow{P} \infty$ under $\p_{\beta,0,p}(~\cdot~\big|\os \in A_2^c)$, implies that
	\begin{align}\label{neg}
		\p_{\beta,0,p}\left(N^{\frac{1}{2}}(\hat{\beta}_N - \beta) \leq t\right) &= \p_{\beta,0,p}\left(N^\frac{p}{2} \overline{\sigma}^p_N \leq \e Z^p + o(1)\Big|\os \in A_2\right)\p_{\beta,0,p}(\os \in A_2) + o(1)\nonumber\\&\rightarrow p_2\gamma_p.
	\end{align}
	Next, fix $t > 0$. Note that for any $\bar{\beta} \in \mathbb{R}$ and $i \in \{1,3\}$, we have under $\p_{\beta+\bar{\beta}/\sqrt N, 0,p}\left(~\cdot~\big|\os \in A_i\right)$,  
	\begin{equation}\label{iin121}
		N^{\frac{1}{2}}(\overline{\sigma}^p_N - m_*^p) \xrightarrow{D} N\left(-\frac{\bar{\beta}p^2 m_*^{2p-2}}{H''(m_*)}, -\frac{p^2 m_*^{2p-2}}{H''(m_*)}\right) \textrm{ and } \e\left(N^{\frac{1}{2}}(\overline{\sigma}^p_N - m_*^p)\right) \rightarrow -\frac{\bar{\beta}p^2m_*^{2p-2}}{H''(m_*)}.
	\end{equation} 
	By \eqref{cc2} in Lemma \ref{redunsup3}, $\p_{\beta+ t/\sqrt N, 0, p}(\os \in A_2) \leq Ce^{-DN^{\frac{1}{2}}}$ for some positive constants $C$ and $D$. It thus follows from the second convergence in \eqref{iin121}, that
	\begin{equation}\label{expconv}
		\e_{\beta+N^{-\frac{1}{2}}t, 0, p} \left(N^{\frac{1}{2}}(\overline{\sigma}^p_N - m_*^p)\right) \rightarrow -\frac{tp^2m_*^{2p-2}}{H''(m_*)}.
	\end{equation} 
	Next, observe that $N^{\frac{1}{2}}(\overline{\sigma}^p_N - m_*^p) \xrightarrow{P} - \infty$ under $\p_{\beta, 0,p}\left(~\cdot~\big|\os \in A_2\right)$. Combining this with \eqref{expconv} and using the fact that $p_1=p_3$, we have by the first convergence in \eqref{iin121},
	\begin{align}\label{final}
		&\p_{\beta,0,p}\left(N^{\frac{1}{2}}(\hat{\beta}_N - \beta)> t\right)\nonumber\\ 
		& =  \p_{\beta,0,p} \left(N^{\frac{1}{2}}(\overline{\sigma}^p_N - m_*^p) > -\frac{tp^2m_*^{2p-2}}{H''(m_*)} + o(1)\Bigg|\os \in A_1\right)\p_{\beta,0,p}(\os \in A_1)\nonumber\\
		&+  \p_{\beta,0,p} \left(N^{\frac{1}{2}}(\overline{\sigma}^p_N - m_*^p) > -\frac{tp^2m_*^{2p-2}}{H''(m_*)} + o(1)\Bigg|\os \in A_3\right)\p_{\beta,0,p}(\os \in A_3) + o(1) \nonumber\\
		&\rightarrow 2p_1\p_{\beta,0,p}\left(N\left(0, -\frac{p^2 m_*^{2p-2}}{H''(m_*)}\right) > -\frac{tp^2m_*^{2p-2}}{H''(m_*)}\right)\nonumber\\& =  p_1 \p_{\beta,0,p}\left(N^+\left(0,-\frac{H''(m_*)}{p^2m_*^{2p-2}}\right) > t\right).
	\end{align}
	The result in \eqref{eq:bmle_multimax_6} now follows from \eqref{neg} and \eqref{final}.

	\section{Constructing Confidence Intervals}
	\label{sec:applications}

	In this section, we discuss how the limiting distributions for the ML estimates $\hat{\beta}_N$ and $\hat{h}_N$ obtained above can be used to construct asymptotically valid confidence intervals for the respective parameters. One complication towards using the above results directly is that the limiting distributions $\hat{\beta}_N$ and $\hat{h}_N $ depend on the actual position of the true parameter $(\beta,h) \in \Theta$.  However, if there were an oracle that told us that the unknown parameter $(\beta,h)$ is $p$-regular, then using the results in \eqref{eq:h_1} and \eqref{eq:bmle_m_1} we would be able to easily construct  confidence intervals for the parameters with asymptotic coverage probability $1-\alpha$, as follows:

	\begin{itemize}

		\item {\it Confidence interval for $h$ at the regular points}: Suppose $\beta \geq 0$ is known and $(\beta, h)$ is $p$-regular. Denote the unique maximizer of the function $H$ by $m_*= m_*(\beta, h, p)$. Note that, by  Theorem \ref{cltun}, the average magnetization $\os \pto m_*$, under $\P_{\beta, h, p}$. Therefore, by \eqref{eq:h_1}, 
		\begin{equation}\label{eq:cih1}
			I_{\mathrm{reg}} := \left(\hat{h}_N - \sqrt{\frac{-H''(\os)}{N}} z_{1-\frac{\alpha}{2}},~\hat{h}_N + \sqrt{\frac{-H''(\os)}{N}} z_{1-\frac{\alpha}{2}}\right), 
		\end{equation} 
		is an interval which contains $h$ with asymptotic coverage probability $1-\alpha$, whenever $(\beta, h)$- is $p$-regular.\footnote{Note that $z_\alpha$ is the $\alpha$-th quantile of the standard normal distribution, that is, $\mathbb{P}(N(0, 1) \leq z_\alpha) = \alpha$.} More precisely, $\P_{\beta, h, p}(h \in I_{\mathrm{reg}}) \rightarrow 1-\alpha$, for  $(\beta, h) \in \Theta$ which is regular.

		\item {\it Confidence interval for $\beta$ at the regular points}: Suppose $h \ne 0$ is known and $(\beta, h)$ is $p$-regular. As before, denote the unique maximizer of the function $H$ by $m_*= m_*(\beta, h, p)$. Therefore, by \eqref{eq:bmle_m_1}, 
		\begin{equation}\label{eq:cibeta1}
			J_{\mathrm{reg}}: = \left(\hat{\beta}_N - \frac{\os^{1-p}}{p}\sqrt{\frac{-H''(\os)}{N}} z_{1-\frac{\alpha}{2}},~\hat{\beta}_N + \frac{\os^{1-p}}{p}\sqrt{\frac{-H''(\os)}{N}} z_{1-\frac{\alpha}{2}}\right),
		\end{equation}   
		is an interval which contains $\beta$ with asymptotic coverage probability $1-\alpha$, whenever $(\beta, h)$ is $p$-regular. Note that the assumption $h \ne 0$ is essential, since \ref{eq:bmle_m_2} shows that the ML estimate $\hat{\beta}_N$ may be inconsistent otherwise. 
	\end{itemize}
	Note that the length of $I_{\mathrm{reg}}$ does not depend upon the true value of $\beta$, and 
	the length of $J_{\mathrm{reg}}$ does not depend on the true $h$. 
	
	Now, we discuss how the intervals in \eqref{eq:cih1} and \eqref{eq:cibeta1} can be modified so that they are valid at all parameter points. To this end, let $\overline{\cp}$ denote the closure of the curve $\cp$ with respect to the Euclidean topology on $\Theta$, that is, the union of $\cp$ with the $p$-special point(s) (recall \eqref{eq:beta_h_special}).
	
	\begin{itemize}
		
		\item {\it Confidence interval for $h$ for all points}: Suppose $\beta \geq 0$ is known, and define $S_p(\beta) :=  \{\underline{h} \in \mathbb{R}: (\beta,\underline{h}) \in  \overline{\cp}\}$. Note that $S_p(\beta)$ is either empty, a singleton or a doubleton (recall Figures \ref{figure:ordering1} and  \ref{figure:ordering2}), and is free of the unknown parameter $h$. Then 
		\begin{align}\label{eq:cih2}
			I := I_{\mathrm{reg}} \bigcup S_p(\beta), 
		\end{align}  
		is an interval with the same length (Lebesgue measure) as the regular interval $I_{\mathrm{reg}}$, and contains $h$ with asymptotic probability at least $1-\alpha$, {\it for all $(\beta, h) \in \Theta$}.  This is because the asymptotic coverage probability is guaranteed to be $1-\alpha$ when $(\beta, h)$ is $p$-regular by the discussion following \eqref{eq:cih1} above. On the other hand, 
		if $(\beta, h)$ is not $p$-regular, by definition $h \in S_p(\beta)$, and hence, by \eqref{eq:cih2}, $\P_{\beta,h, p}(I \ni h) = 1$.

		\item {\it Confidence interval for $\beta$ for all points with $h \ne 0$}: 
		Fix $h \neq 0$, and define $T_p(\beta) :=  \{\underline{\beta} \geq 0: (\underline{\beta},h) \in  \overline{\cp}\}$. Note that $T_p(h)$ is either empty or a singleton, and is free of the unknown parameter $\beta$. Then, as above,  
		\begin{align}\label{eq:cibeta2}
			J:= J_{\mathrm{reg}} \bigcup T_p(h), 
		\end{align}   
		is an interval with the same length (Lebesgue measure) as the regular interval $J_{\mathrm{reg}}$, and contains $\beta$ with asymptotic probability at least $1-\alpha$, {\it for all $(\beta, h) \in \Theta$.} 
		
	\end{itemize}
	
	Figure \ref{fig:sigma_interval_beta} shows 100 realizations of the 95\% confidence interval for $\beta$ at the $3$-regular point $(\beta, h)=(0.5, 0.2)$, with $N=10,000$. The green horizontal line represents the true parameter $\beta=0.5$ and the intervals not containing the true parameter are shown in red. \\ 
	
	\begin{figure}[h]
		\centering
		\begin{minipage}[c]{1.0\textwidth}
			\centering
			\includegraphics[width=5.25in,height=2.75in]
			{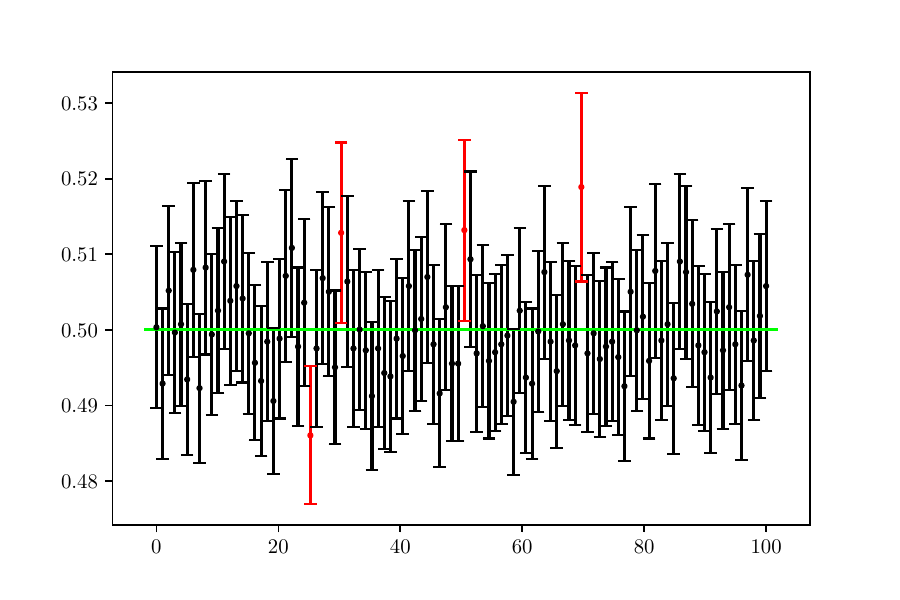}\\
		\end{minipage}
		\caption{\small{100 realizations of the 95\% confidence interval for $\beta$ at the $3$-regular point $(\beta, h)=(0.5, 0.2)$, with $N=10,000$. The intervals not containing the true parameter $\beta=0.5$ are shown in red. }}
		\label{fig:sigma_interval_beta}
	\end{figure}

	\small{
\noindent\textbf{Acknowledgements.} B. B. Bhattacharya was supported by NSF CAREER grant DMS 2046393 and a Sloan Research Fellowship. }

	\normalsize

	\appendix

	\section{Proof of Theorem \ref{cltun}}\label{perturbed_asm}

	
	\subsection{Proof of Theorem \ref{cltun} when $(\beta, h)$ is $p$-regular} 
	\label{sec:pfcltun_I}
	
	Fix a $p$-regular point $(\beta, h) \in \Theta$ and consider a sequence $(\beta_N, h_N) \in \Theta$ (to be specified later) such that $\beta_N \rightarrow \beta$ and $h_N \rightarrow h$. It has been shown in  Lemma \ref{derh44} that the function $H_N(x) := H_{\beta_N,h_N,p}(x)$ will have a unique global maximizer $m_*(N)$, for all large $N$, and $m_*(N)\rightarrow m_*$ as $N \rightarrow \infty$. Choose this maximizer $m_*(N)$ and define, for $\alpha \in (0, 1)$, 
	\begin{align}\label{eq:an_maximizer}
		A_{N,\alpha} := \left(m_*(N)-N^{-\frac{1}{2} + \alpha},m_*(N) +N^{-\frac{1}{2} + \alpha}\right). 
	\end{align}
	The first step in the  proof of Theorem \ref{cltun} when $(\beta, h)$ is $p$-regular,  is to show that under $\bar{\p}$, the average magnetization $\os$ concentrates around $m_*(N)$ at rate $N^{-\frac{1}{2} + \alpha}$, for any $\alpha > 0$. 
	
	\begin{lem}\label{conc}
		Suppose  $(\beta, h) \in \Theta$ is $p$-regular. Then for $\alpha \in \left(0,\frac{1}{6}\right]$ and  $A_{N,\alpha}$ as defined above in \eqref{eq:an_maximizer},\footnote{For any set $A$, $A^c$ denotes the complement of the set $A$.}
		\begin{equation*}
			\bar{\mathbb P}\left(\os \in A_{N,\alpha}^c\right) = \exp\left\{\frac{1}{3}N^{2\alpha} H''(m_*) \right\}O(N^{\frac{3}{2}}).
		\end{equation*}
	\end{lem}
	
	\begin{proof} 	Recalling \eqref{cwss} note that the average magnetization $\os$ has the probability mass function, 
	\begin{align*}
		\P_{\beta,h,p}(\overline{\sigma}_N=m)= \frac{1}{2^{N}Z_N(\beta,h,p)} \binom{N}{\frac{N(1+m)}{2}}e^{N(\beta m^p + hm)} , 	\end{align*} 
for $$m \in \mathcal{M}_N:=\left\{-1,-1+\dfrac{2}{N},\ldots, 1-\dfrac{2}{N},1\right\}.$$	It  follows from \cite{talagrand}, Equation (5.4), that for any $m \in \mathcal{M}_N$, the cardinality of the set
		$$A_m := \left\{\bs \in \sa_N: \os = m\right\}$$ can be bounded by
		\begin{equation}\label{boundtal}
			\frac{2^N}{LN^{\frac{1}{2}}} \exp\left\{-N I(m) \right\}\leq |A_m| \leq 2^N\exp\left\{-N I(m) \right\}
		\end{equation}
		for some universal constant $L$ (recall that $I(\cdot)$ is the binary entropy function). Hence, we have from \eqref{boundtal},
		\begin{align}\label{sb}
			\bar{\p}(\os \in A_{N,\alpha}^c) & =  \frac{\sum_{m \in \mathcal{M}_N \bigcap A_{N,\alpha}^c} |A_m|\exp\left\{N(\beta_N m^p + h_Nm) \right\}}{\sum_{m \in \mathcal{M}_N} |A_m|\exp\left\{N(\beta_N m^p + h_Nm) \right\}}\nonumber\\& \leq  \frac{LN^{\frac{1}{2}}(N+1) \sup_{x \in A_{N,\alpha}^c}e^{NH_N(x)}}{\sup_{x\in [-1,1]}e^{NH_N(x)} }\nonumber\\& =  \exp\left\{N\left(\sup_{x\in A_{N,\alpha}^c} H_N(x) - H_N\left(m_*(N)\right)\right)\right\} O(N^{\frac{3}{2}}).
		\end{align}
		By  \cite[Lemma B.11]{cmp}, we know that for all large $N$, $\sup_{x\in A_{N,\alpha}^c} H_N(x)$ is either $H_N(m_*(N)-N^{-\frac{1}{2} + \alpha})$ or $H_N(m_*(N)+N^{-\frac{1}{2} + \alpha})$. Since $H_N'\left(m_*(N)\right) = 0$ and the functions $H_N^{(3)}$ are uniformly bounded on any closed interval contained in $(-1,1)$, Taylor's theorem gives us:
		\begin{align}
			H_N\left(m_*(N)\pm N^{-\frac{1}{2} + \alpha}\right) - H_N(m_*(N)) & =  \frac{1}{2}N^{-1+2\alpha} H_N''(m_*(N)) + O\left(N^{-\frac{3}{2} + 3\alpha}\right)\label{taylorH1}\\& \leq  \frac{1}{3}N^{-1+2\alpha} H''(m_*) +  O\left(N^{-\frac{3}{2} + 3\alpha}\right)\label{taylorH2}.
		\end{align}
		Note that \eqref{taylorH2} follows from \eqref{taylorH1} since $H_N''(m_*(N)) \rightarrow H''(m_*) <0$.  The proof of Lemma \ref{conc} is now complete, in view of \eqref{sb}. 
	\end{proof}

	Lemma \ref{conc} shows that almost all contribution to $\bar{Z}_N$ comes from configurations whose average lies in a vanishing neighborhood of the maximizer $m_*(N)$ of $H_N$. This enables us to accurately approximate  the partition function $\bar{Z}_N$. This involves a Riemann approximation of the sum of the mass function $\p_{\beta_N,h_N,p}(\bs)$ over all $\bm \sigma$ whose mean lies in a vanishing neighborhood of $m_*$, followed by a further saddle-point approximation of the resulting integral. 
	
	\begin{lem}\label{ex} Suppose  $(\beta, h) \in \Theta$ is $p$-regular. Then for $\alpha > 0$ and $N$ large enough, the partition function can be expanded as,
		\begin{equation}\label{partex}
			\bar{Z}_N =  \frac{e^{NH_N(m_*(N))}}{\sqrt{(m_*(N)^2-1)H_N''(m_*(N))}}\left(1+O\left(N^{-\frac{1}{2} +\alpha}\right)\right),
		\end{equation}
		where $m_*(N)$ is the unique maximizer of the function $H_N$. Moreover, for $N$ large enough, the log-partition function can be expanded as, 
		\begin{equation}\label{logpartex}
			\bar{F}_N =  N H_N(m_*(N)) - \tfrac{1}{2}\log \left[(m_*(N)^2-1)H_N''(m_*(N))\right] + O\left(N^{-\frac{1}{2} +\alpha}\right).
		\end{equation}
	\end{lem}
	
	\begin{proof}[Proof of Lemma \ref{ex}]
		Without loss of generality, let $\alpha \in \left(0,\frac{1}{6}\right]$ and note that
		\begin{equation}\label{fststep}
			\bar{\mathbb P}(\os \in A_{N,\alpha}) = \bar{Z}_N^{-1}\sum_{m \in \mathcal{M}_N\bigcap A_{N,\alpha}} \binom{N}{N(1+m)/2} \exp\left\{N(\beta_N m^p + h_Nm-\log 2) \right\}.
		\end{equation}
		By Lemma \ref{conc}, $\bar{\mathbb P}(\os \in A_{N,\alpha}) = 1-O\left(e^{-N^\alpha}\right)$ and hence \eqref{fststep} gives us
		\begin{align}\label{secstep}
			\bar{Z}_N & =  \left(1+ O\left(e^{-N^\alpha}\right) \right)\sum_{m \in \mathcal{M}_N\bigcap A_{N,\alpha}} \binom{N}{N(1+m)/2} \exp\left\{N(\beta_N m^p + h_Nm-\log 2) \right\}\nonumber\\& =  \left(1+ O\left(e^{-N^\alpha}\right) \right) \sum_{m \in \mathcal{M}_N\bigcap A_{N,\alpha}} \zeta(m)
		\end{align}
		where $\zeta:[-1,1] \rightarrow \mathbb{R}$ is defined as
		\begin{equation}\label{xidef}
			\zeta(x) := \binom{N}{N(1+x)/2} \exp\left\{N(\beta_N x^p + h_Nx - \log 2)\right\},
		\end{equation}
		where $\binom{N}{N(1+x)/2}$ is interpreted as a continuous binomial coefficient.\footnote{For real numbers $x\geq y>0$, the continuous binomial coefficient $x$ choose $y$ is defined as $\binom{x}{y} := \frac{\Gamma(x+1)}{\Gamma(y+1)\Gamma(x-y+1)}$.} In the next lemma we bound the derivative of the function $\zeta$ in a neighborhood of the point $m_*(N)$.

	\begin{lem}\label{imp1} For any sequence $x \in (-1,1)$  bounded away from both $1$ and $-1$, 
		\begin{align}\label{eq:imp11} 
		\zeta(x) = \sqrt{\frac{2}{\pi N (1-x^2)}} e^{NH_N(x)}\left(1+O(N^{-1})\right) . 
		\end{align} 
Moreover, for every $\alpha \geq 0$ and $p$-regular point $(\beta,h)$, we have the following bound:
		\begin{align}\label{eq:imp12} 
		\sup_{x\in A_{N,\alpha}} |\zeta'(x)| = \zeta(m_*(N))O\left(N^{\frac{1}{2}+\alpha}\right), 
		\end{align}
		where $m_*(N)$ is the global maximizer of $H_N$ and $A_{N,\alpha} := \left(m_*(N)-N^{-\frac{1}{2} + \alpha},m_*(N) +N^{-\frac{1}{2} + \alpha}\right)$. 
	\end{lem}

The proof of this Lemma \ref{imp1} is given in Section \ref{sec:lempf}. We use this lemma and to approximate the sum in \eqref{secstep} by an integral, using \cite[Lemma A.2]{cmp}. Note that Lemma A.2 in \cite{cmp} can be applied with $n = \Theta(N^{\frac{1}{2}+\alpha})$ to obtain (using Lemma \ref{imp1}),
		\begin{align}\label{rax}
			\left|\int_{A_{N,\alpha}} \zeta(x)\mathrm d x - \frac{2}{N} \sum_{m \in \mathcal{M}_N\bigcap A_{N,\alpha}} \zeta(m)\right| &\leq \Theta(N^{-\frac{1}{2}+\alpha})N^{-1}\sup_{x\in A_{N,\alpha}}|\zeta'(x)|\nonumber\\
			&= O\left(N^{-\frac{1}{2}+\alpha}\cdot N^{-1}\cdot N^{\frac{1}{2}+\alpha}\right)\zeta(m_*(N))\nonumber\\
			&= O\left(N^{-1 + 2\alpha}\right)\zeta(m_*(N)).
		\end{align}

		It now follows from \eqref{rax},  Lemma A.3 and Lemma A.5 in \cite{cmp},  and Lemma \ref{imp1}, that
		\begin{align}\label{fin1}
			&\sum_{m \in \mathcal{M}_N\bigcap A_{N,\alpha}} \zeta(m)\nonumber\\ 
			&= \frac{N}{2} \int_{A_{N,\alpha}} \zeta(x)\mathrm d x + O(N^{2\alpha}) \zeta(m_*(N))\nonumber\\
			&= \frac{N^{\frac{1}{2}}}{2}\left(1+O(N^{-1})\right)\int_{A_{N,\alpha}} e^{NH_N(x)}\sqrt{\frac{2}{\pi (1-x^2)}}\mathrm d x + O(N^{2\alpha}) \zeta(m_*(N))\nonumber\\
			&= \frac{N^{\frac{1}{2}}}{2} \sqrt{\frac{2\pi}{N|H_N''(m_*(N))|}}\sqrt{\frac{2}{\pi(1-m_*(N)^2)}} e^{N H_N(m_*(N))}\left(1+O\left(N^{-\frac{1}{2}+3\alpha}\right)\right) + O(N^{2\alpha}) \zeta(m_*(N))\nonumber\\
			&= \frac{e^{N H_N(m_*(N))}}{\sqrt{(m_*(N)^2-1)H_N''(m_*(N))}}\left(1+O\left(N^{-\frac{1}{2}+3\alpha}\right)\right)\nonumber\\ 
			&+ \sqrt{\frac{2}{\pi N (1-m_*(N)^2)}} e^{NH_N(m_*(N))}\left(1+O(N^{-1})\right)O(N^{2\alpha})\nonumber\\ 
			&= \frac{e^{N H_N(m_*(N))}}{\sqrt{(m_*(N)^2-1)H_N''(m_*(N))}}\left(1+O\left(N^{- \frac{1}{2}+3\alpha}\right)\right).
		\end{align}
		Combining \eqref{secstep} and \eqref{fin1}, we have:
		\begin{align}\label{fin2}
			\bar{Z}_N & =  \left(1+ O\left(e^{-N^\alpha}\right) \right)\left(1+O\left(N^{- \frac{1}{2}+3\alpha}\right)\right)\frac{e^{N H_N(m_*(N))}}{\sqrt{(m_*(N)^2-1)H_N''(m_*(N))}}\nonumber\\& =  \left(1+O\left(N^{- \frac{1}{2}+3\alpha}\right)\right)\frac{e^{N H_N(m_*(N))}}{\sqrt{(m_*(N)^2-1)H_N''(m_*(N))}}.
		\end{align}
		This completes the proof of \eqref{partex}. If we take logarithm on all sides in \eqref{fin2} and use the fact that $\log\left(1+ O(a_n)\right) = O(a_n)$ for any sequence $a_n = o(1)$, then we get \eqref{logpartex}, completing the proof. 
	\end{proof}

\subsubsection{Completing the Proof of \eqref{eq:cltun_I}} 
\label{sec:regularpf} 
We now have all the necessary ingredients in order to derive the CLT for $\os$ when $(\beta, h)$ is $p$-regular. Throughout this subsection, we take 
	$$\beta_N = \beta+N^{-\frac{1}{2}}\bar{\beta} \quad \text{and} \quad  h_N = h + N^{-\frac{1}{2}}\bar{h},$$ for some fixed $\beta\geq 0$ and $\bar{\beta}, h, \bar{h} \in \mathbb{R}$.  Now, recall that $H_N := H_{\beta_N,h_N,p}$ and $m_*=m_*(\beta,h,p)$ is the unique maximizer of $H$. To complete the proof we will show that the moment generating function of $N^{\frac{1}{2}}\left(\os - m_*\right)$ under $\p_{\beta_N,h_N,p}$ converges pointwise to the moment generating function of the Gaussian distribution with mean $-\bar{h}/H''(m_*)$ and variance $-1/H''(m_*)$. Towards this, fix $t \in \mathbb{R}$ and note that the moment generating function of $N^{\frac{1}{2}}\left(\os - m_*\right)$ at $t$ can be expressed as 
	\begin{equation}\label{cltst1}
		\mathbb{E}_{\beta_N,h_N,p}e^{tN^{\frac{1}{2}}\left(\os - m_*\right)} = e^{-tN^{\frac{1}{2}}m_*}\frac{Z_N\left(\beta_N,h_N+N^{-\frac{1}{2}}t,p\right)}{Z_N(\beta_N,h_N,p)}.
	\end{equation}
	Using Lemma \ref{ex} and the fact that $m_*(N) \rightarrow m_*$, the right side of \eqref{cltst1} simplifies to 
	\begin{equation}\label{cltst2}
		(1+o(1)) e^{-tN^{\frac{1}{2}}m_* +  N\left\{H_{\beta_N, h_N+N^{-\frac{1}{2}}t,p} \left(m_*\left(\beta_N, h_N+N^{-\frac{1}{2}}t,p\right) \right) - H_{\beta_N, h_N,p} \left(m_*\left(\beta_N, h_N,p\right) \right) \right\} } .
	\end{equation}
	Now, using \cite[Lemma B.5]{cmp} and a simple Taylor expansion gives us
	\begin{align}\label{cltst3}
		m_*\left(\beta_N, h_N+N^{-\frac{1}{2}}t,p\right) - m_*\left(\beta_N, h_N,p\right) &= N^{-\frac{1}{2}}t~ \frac{\partial}{\partial{\underline{h}}} m_*(\beta_N,\underline{h},p)\Big|_{\underline{h}=h_N} + O(N^{-1})\nonumber\\
		&= -\frac{t}{N^{\frac{1}{2}} H_N''(m_*(\beta_N,h_N,p))} + O(N^{-1}). 
	\end{align}
	Using \eqref{cltst3} and a further Taylor expansion, we have
	\begin{align}\label{cltst4}
		&N\left\{H_{\beta_N, h_N,p} \left(m_*\left(\beta_N, h_N+N^{-\frac{1}{2}}t,p\right) \right) - H_{\beta_N, h_N,p} \left(m_*\left(\beta_N, h_N,p\right) \right) \right\}\nonumber\\
		&\quad \quad \quad \quad = \frac{N}{2}\left\{m_*\left(\beta_N, h_N+N^{-\frac{1}{2}}t,p\right) - m_*\left(\beta_N, h_N,p\right)\right\}^2 H_N''\left(m_*\left(\beta_N, h_N,p\right)\right) + o(1)\nonumber\\ 
		&\quad \quad \quad \quad = \frac{t^2}{2H_N''\left(m_*\left(\beta_N, h_N,p\right)\right)} + o(1) \nonumber\\ 
		&\quad \quad \quad \quad = \frac{t^2}{2H''(m_*)} + o(1).
	\end{align}
	Next, by \cite[Lemma B.5]{cmp} and a Taylor expansion,
	\begin{align}\label{cltst5}
		tN^{\frac{1}{2}} m_*\left(\beta_N, h_N+N^{-\frac{1}{2}}t,p\right) &= tN^{\frac{1}{2}} m_*(\beta_N,h,p) + t(t+\bar{h}) \frac{\partial}{\partial{\underline{h}}} m_*(\beta_N,\underline{h},p)\Big|_{\underline{h} = h} + o(1)\nonumber\\
		&= tN^{\frac{1}{2}} m_* + t\bar{\beta}\frac{\partial}{\partial \underline{\beta}} m_*(\underline{\beta},h,p)\Big|_{\underline{\beta} = \beta}  - \frac{t(t+\bar{h})}{H''(m_*)} + o(1)\nonumber\\ 
		&= tN^{\frac{1}{2}}m_* -\frac{t\bar{\beta}pm_*^{p-1}}{H''(m_*)} -\frac{t(t+\bar{h})}{H''(m_*)} + o(1).
	\end{align}
	Adding \eqref{cltst4} and \eqref{cltst5}, we have:
	\begin{align}\label{cltst6}
		& N\left\{H_{\beta_N, h_N+N^{-\frac{1}{2}}t,p} \left(m_*\left(\beta_N, h_N+N^{-\frac{1}{2}}t,p\right) \right) - H_{\beta_N, h_N,p} \left(m_*\left(\beta_N, h_N,p\right) \right) \right\}\nonumber\\ 
		& \quad \quad \quad \quad = tN^{\frac{1}{2}}m_* -\frac{t(\bar{h}+\bar{\beta}pm_*^{p-1})}{H''(m_*)} -\frac{t^2}{2H''(m_*)} + o(1).
	\end{align}
	Using \eqref{cltst6}, the expression in \eqref{cltst2} becomes
	\begin{equation}\label{cltst7}
		\exp\left\{ -\frac{t(\bar{h}+\bar{\beta}pm_*^{p-1})}{H''(m_*)} -\frac{t^2}{2H''(m_*)}\right\} + o(1).
	\end{equation}
	The constant in expression \eqref{cltst7} is easily recognizable as the moment generating function of $N(-\frac{\bar{h}+\bar{\beta}pm_*^{p-1}}{H''(m_*)},~-\frac{1}{H''(m_*)})$ evaluated at $t$. This completes the proof of Theorem \ref{cltun}. \qed

\subsubsection{Proof of Lemma \ref{imp1}} 
\label{sec:lempf}

The result in \eqref{eq:imp11} follows from \cite[Lemma A.5]{cmp}. Hence, it suffices to prove \eqref{eq:imp12}. We begin with the following lemma: 
		\begin{lem}\label{xiprime}
			For any sequence $x \in (-1,1)$ that is bounded away from both $1$ and $-1$, we have
			$$\zeta'(x) =\zeta(x)\left(NH_N'(x) + \frac{x}{1-x^2} + O(N^{-1})\right).$$
		\end{lem}
		
		\begin{proof}
			By Lemma A.1 and Equation (A.1) in \cite{cmp}, we have	
			\begin{align}\label{ff1}
				& \dfrac{\mathrm d}{\mathrm d x} \binom{N}{N(1+x)/2}\nonumber\\ & = \frac{N}{2}\binom{N}{N(1+x)/2}\left[\psi\left(1-\frac{Nx}{2}+\frac{N}{2}\right)-\psi\left(1+\frac{Nx}{2}+\frac{N}{2}\right)\right]\nonumber\\
				& = \frac{N}{2}\binom{N}{N(1+x)/2} \left(\log\left(\frac{ N}{2}(1-x)\right)-\log\left(\frac{N}{2}(1+x)\right)+\frac{1}{N(1-x)}-\frac{1}{N(1+x)} + O(N^{-2})\right)\nonumber\\
				& = \binom{N}{N(1+x)/2}\left[-N \tanh^{-1}(x)+\frac{x}{1-x^2}+O(N^{-1})\right].
			\end{align} 
			We thus have by the product rule of differential calculus and \eqref{ff1},
			\begin{align*}
				\zeta'(x) & =  \zeta(x)(N\beta_N px^{p-1} + Nh_N) + \exp\left\{N(\beta_N x^p + h_Nx - \log 2)\right\}\frac{\mathrm d}{\mathrm d x} \binom{N}{N(1+x)/2}\\& =  \zeta(x)(N\beta_N px^{p-1} + Nh_N) + \zeta(x)\left[-N \tanh^{-1}(x)+\frac{x}{1-x^2}+O(N^{-1})\right]\\& =  \zeta(x)\left(NH_N'(x) + \frac{x}{1-x^2} + O(N^{-1})\right),
			\end{align*}
			completing the proof of Lemma \ref{xiprime}.
		\end{proof}
		
		Now, we proceed with the proof of Lemma \ref{imp1}. First note that, since $H_N'(m_*(N)) = 0$, we have by the mean value theorem,
		\begin{equation}\label{imp1eq1}
			\sup_{x \in A_{N,\alpha}}\big|H_N'(x)\big| \leq \sup_{x \in A_{N,\alpha}}\big|x-m_*(N)\big|\sup_{x\in A_{N,\alpha}}|H_N''(x)|=O\left(N^{-\frac{1}{2}+\alpha}\right).
		\end{equation}
		It follows from \eqref{imp1eq1} and Lemma \ref{xiprime} that
		\begin{equation}\label{imp1eq2}
			\sup_{x \in A_{N,\alpha}}|\zeta'(x)| \leq O\left(N^{\frac{1}{2}+\alpha}\right)\sup_{x \in A_{N,\alpha}}\zeta(x) .
		\end{equation} 
		Using \eqref{eq:imp11} gives, 
		\begin{equation}\label{imp1eq3}
			\sup_{x \in A_{N,\alpha}} \zeta(x) \leq \left(1+O(N^{-1})\right) \zeta(m_*(N))\sup_{x\in A_{N,\alpha}} \sqrt{\frac{1-m_*(N)^2}{1-x^2}} = \zeta(m_*(N))O(1).
		\end{equation}
		Lemma \ref{imp1} now follows from \eqref{imp1eq2} and \eqref{imp1eq3}.  \qed

	\subsection{Proof of Theorem \ref{cltun} when $(\beta, h)$ is $p$-special}\label{sec:pfcltun_II} 
	
	When $(\beta, h)$ is $p$-special, we consider local perturbations of the parameters 
	$$(\beta_N,h_N) :=\left(\beta + \bar{\beta} N^{-\frac{3}{4}}, h + \bar{h} N^{-\frac{3}{4}}\right)$$
	as in the statement of Theorem \ref{cltun} (3).  
	Note that in this case the function $H_{\beta, h, p}$ still has a unique maximizer $m_*=m_*(\beta, h, p)$, but $H_{\beta,h,p}''(m_*) = 0$. The proof strategy here has the same broad roadmap as in the $p$-regular case, with relevant modifications while taking Taylor expansions, since $H_{\beta,h,p}''(m_*) = 0$. As before, the first step is to prove the concentration of $\os$ within a vanishing neighborhood of $m_*$ (Lemma \ref{irrconch}). Here, the concentration window turns out to be a little more inflated, that is, its length is of order $N^{-\frac{1}{4} + \alpha}$, for $\alpha > 0$. Next, we approximate the partition function $\bar Z_N$, where, since the second derivative of $H$ is zero at the maximizer, we need to consider derivatives up to order four to accurately approximate $\bar Z_N$ (Lemma \ref{ex2}). The details of the proof are presented below. 
	
	Throughout this section, as usual, we will denote $H_{\beta,h,p}$ by $H$, $H_{\beta_N, h_N,p}$ by $H_N$, the unique global maximizer of $H_{\beta_N, h_N,p}$ (for large $N$) by $m_*(N)$, $\p_{\beta_N,h_N,p} $ by $\bar{\p}$, $Z_N(\beta_N,h_N,p)$ by $\bar{Z}_N$ and $F_N(\beta_N,h_N,p)$ by $\bar{F}_N$. As outlined above, the first step in the proof of Theorem \ref{cltun} when $(\beta, h)$ is $p$-special, is to show the concentration of $\os$ within a vanishing neighborhood of $m_*(N)$. In the $p$-special case, this is more delicate, because it requires Taylor expansions up to the fourth order term. Here, the concentration window turns out to be a bit more inflated as well, and is given by:
	\begin{equation}\label{concwind}
		\ma_{N,\alpha} := (m_*(N)-N^{-\frac{1}{4} + \alpha},m_*(N) +N^{-\frac{1}{4} + \alpha} ). 
	\end{equation}

	\begin{lem}\label{irrconch} Suppose  $(\beta, h) \in \Theta$ is $p$-special. Fix $\alpha \in \left(0,\frac{1}{20}\right]$ and let $\ma_{N,\alpha}$ be as in \eqref{concwind}. Then,
		\begin{equation*}
			\bar{\mathbb P}\left(\os \in \ma_{N,\alpha}^c\right) = \exp\left\{\frac{1}{24}N^{4\alpha} H^{(4)}(m_*)(1+o(1)) \right\}O(N^{\frac{3}{2}}).
		\end{equation*}	
	\end{lem}
	
	\begin{proof}
		It follows from the proof of Lemma \ref{conc} and the fact $H_N''(m_*(N)) \leq 0$, that
		\begin{align}\label{simtoconc}
			\bar{\mathbb P} & (\os \in \ma_{N,\alpha}^c)   \nonumber \\ 
			&\leq  \exp\left\{N\left( H_N\left(m_*(N) \pm N^{-\frac{1}{4}+\alpha}\right) - H_N\left(m_*(N)\right)\right)\right\} O(N^{\frac{3}{2}})\nonumber\\
			&\leq  \exp\left\{\frac{1}{6} N^{\frac{1}{4}+3\alpha}H_N^{^{(3)}} (m_*(N)) + \frac{1}{24}N^{4\alpha} H_N^{(4)}(m_*(N)) + O\left(N^{-\frac{1}{4} + 5\alpha}\right)\right\} O(N^{\frac{3}{2}}). 
		\end{align}	
		Now, it follows from Lemma \ref{mnminusm}, that $|H_N^{(3)}(m_*(N))| =  O(N^{-1/4})$. Hence, $N^{(1/4) + 3\alpha} H_N^{(3)}(m_*(N)) + N^{4\alpha}H_N^{(4)}(m_*(N)) = N^{4\alpha}H^{(4)}(m_*)(1+o(1))$, and Lemma \ref{irrconch} follows from \eqref{simtoconc}. 
	\end{proof}

Similar to Lemma \ref{imp1} in the $p$-regular case, we need an estimate on the function $\zeta$ (recall \eqref{xidef}) in the $p$-special case. In this case, the bound on $\zeta'$ is better, and holds on a slightly larger region, as well.  
	
	\begin{lem}\label{xiprimeirreg}
		Let $m_*(N)$ be the unique global maximizer of $H_N := H_{\beta_N,h_N,p}$, where $(\beta_N,h_N) := \left(\beta + \bar{\beta}N^{-3/4}, h +\bar{h}N^{-3/4}\right)$ for some $\bar{\beta},\bar{h}\in \mathbb{R}$, and $(\beta,h)$ is a $p$-special point. Then, for all $\alpha \geq 0$,
		$$\sup_{x\in \mathcal{A}_{N,\alpha}} |\zeta'(x)| = \zeta(m_*(N))O\left(N^{\frac{1}{4}+3\alpha}\right)$$
		where $\mathcal{A}_{N,\alpha} := \left(m_*(N) - N^{-\frac{1}{4}+\alpha}, m_*(N) + N^{-\frac{1}{4} + \alpha}\right)$.
	\end{lem}

	\begin{proof}
		The proof of Lemma \ref{xiprimeirreg} is similar to that of Lemma \ref{imp1}, the only difference being a change in the estimate of $\sup_{x \in \mathcal{A}_{N,\alpha}} |H_N'(x)|$ from the estimate in \eqref{imp1eq1}. Note that
		\begin{align*}
			\sup_{x\in \mathcal{A}_{N,\alpha}} |H_N''(x)| &= \sup_{x\in \mathcal{A}_{N,\alpha}} |H''(x)| + O\left(N^{-\frac{3}{4}}\right)\\  
			&\leq \sup_{x \in \mathcal{A}_{N,\alpha}} \tfrac{1}{2}(x-m_*)^2\sup_{x\in \mathcal{I}(\mathcal{A}_{N,\alpha}\cup \{m_*\})}H^{(4)} (x) + O\left(N^{-\frac{3}{4}}\right) = O\left(N^{-\frac{1}{2}+ 2\alpha}\right),
		\end{align*}
		where $m_*$ denotes the global maximizer of $H_{\beta,h,p}$ and for a set $A \subseteq \mathbb{R}$, $\mathcal{I}(A)$ denotes the smallest interval containing $A$. The last equality follows from the observation $$\sup_{x\in \mathcal{A}_{N,\alpha}} |x-m_*| \leq \sup_{x\in \mathcal{A}_{N,\alpha}} |x-m_*(N)| + |m_*(N) - m_*| \leq N^{-\frac{1}{4} + \alpha} + O\left(N^{-\frac{1}{4}}\right) = O\left(N^{-\frac{1}{4} + \alpha}\right),$$ by Lemma \ref{mnminusm}. Following \eqref{imp1eq1}, we have
		$$\sup_{x\in \mathcal{A}_{N,\alpha}} |H_N'(x)| = O\left(N^{-\frac{3}{4} + 3\alpha}\right).$$ The rest of the proof is exactly same as that of Lemma \ref{imp1}. 
	\end{proof}

	The next step in the proof of Theorem \ref{cltun} when $(\beta, h)$ is $p$-special is the approximation of the partition function.

	\begin{lem}\label{ex2} Suppose  $(\beta, h) \in \Theta$ is $p$-special, and let $(\beta_N,h_N) = (\beta+ N^{-\frac{3}{4}}\bar{\beta},~h + N^{-\frac{3}{4}}\bar{h})$. Then for $N$ large enough, the partition function $\bar{Z}_N$ can be expanded as
		\begin{equation*}
			\bar{Z}_N =  \frac{N^\frac{1}{4}e^{N H_N(m_*(N))}}{\sqrt{2\pi(1-m_*(N)^2)}}\int_{-\infty}^\infty e^{\eta_{\bar{\beta},\bar{h},p}(y)} \mathrm d y \left(1+o(1)\right),
		\end{equation*}
		where $\eta_{\bar{\beta},\bar{h},p}(y)= a y^2 +  b y^3 + c y^4$, with 	
		\begin{equation*}
			a:=\frac{(6(\bar{\beta}pm_*^{p-1} + \bar{h}))^{\frac{2}{3}}\left(H^{(4)}(m_*)\right)^{\frac{1}{3}}}{4}, ~ b:= - \frac{(6(\bar{\beta}pm_*^{p-1} + \bar{h}))^{\frac{1}{3}}\left(H^{(4)}(m_*)\right)^{\frac{2}{3}}}{6}, ~c:= \frac{H^{(4)}(m_*)}{24}.
		\end{equation*}
	\end{lem}

	\begin{proof}
		Once again, as in the proof of Lemma \ref{ex}, it follows from Lemma \ref{irrconch}, that for $\alpha \in \left(0,\frac{1}{20}\right]$,
		\begin{equation}\label{secstep22}
			\bar{Z}_N = \left(1+ O\left(e^{-N^\alpha}\right) \right) \sum_{m \in \mathcal{M}_N\bigcap \mathcal{A}_{N,\alpha}} \zeta(m),
		\end{equation}
		where $\zeta:[-1,1] \rightarrow \mathbb{R}$ is defined in \eqref{xidef} and $\mathcal{A}_{N,\alpha}$ is defined in \eqref{concwind}.
		It also follows from \cite[Lemma A.2]{cmp} and Lemma \ref{xiprimeirreg}, exactly as in the proof of Lemma \ref{ex}, that
		\begin{equation}\label{rax2}
			\left|\int_{\ma_{N,\alpha}} \zeta(x)\mathrm d x - \frac{2}{N} \sum_{m \in \mathcal{M}_N\bigcap \ma_{N,\alpha}} \zeta(m)\right| = O\left(N^{-1 + 4\alpha}\right)\zeta(m_*(N)).
		\end{equation}
		Hence, we have from \cite[Lemma A.4]{cmp}, \eqref{rax2}, Lemma \ref{imp1}, and Lemma \ref{mnminusm}, 
		\begin{align}\label{fin3}
			& \sum_{m \in \mathcal{M}_N\bigcap \ma_{N,\alpha}} \zeta(m)\nonumber\\
			& =  \frac{N}{2} \int_{\ma_{N,\alpha}} \zeta(x)\mathrm d x + O(N^{4\alpha}) \zeta(m_*(N))\nonumber\\
			& =  \frac{N^{\frac{1}{2}}}{2}\left(1+O(N^{-1})\right)\int_{\ma_{N,\alpha}} e^{NH_N(x)}\sqrt{\frac{2}{\pi (1-x^2)}}\mathrm d x + O(N^{4\alpha}) \zeta(m_*(N))\nonumber\\
			& =   O(N^{4\alpha}) \zeta(m_*(N))+\frac{N^\frac{1}{4}}{\sqrt{2\pi(1-m_*(N)^2)}} e^{N H_N(m_*(N))}  \int_{-N^\alpha}^{N^\alpha}e^{\eta_{\bar{\beta},\bar{h},p}(y)}\mathrm d y \left(1+O\left(N^{-\frac{1}{4}+5\alpha}\right)\right) \nonumber\\
			& =  \frac{N^\frac{1}{4}e^{N H_N(m_*(N))}}{\sqrt{2\pi (1-m_*(N)^2)}}\int_{-\infty}^\infty e^{\eta_{\bar{\beta},\bar{h},p}(y)}\mathrm d y  (1+o(1))\left(1+O\left(N^{-\frac{1}{4}+5\alpha}\right)\right)\nonumber\\
			&+  \sqrt{\frac{2}{\pi N (1-m_*(N)^2)}} e^{NH_N(m_*(N))}\left(1+O(N^{-1})\right)O(N^{4\alpha})\nonumber\\
			& =  \frac{N^\frac{1}{4}e^{N H_N(m_*(N))}}{\sqrt{2\pi(1-m_*(N)^2)}}\int_{-\infty}^\infty e^{\eta_{\bar{\beta},\bar{h},p}(y)}\mathrm d y \left(1+o(1)\right).
		\end{align}
		Combining \eqref{secstep22} and \eqref{fin3}, we have:
		\begin{align*}
			\bar{Z}_N & =  \left(1+ O\left(e^{-N^\alpha}\right) \right)\left(1+o(1)\right)\frac{N^\frac{1}{4}e^{N H_N(m_*(N))}}{\sqrt{2\pi(1-m_*(N)^2)}}\int_{-\infty}^\infty e^{\eta_{\bar{\beta},\bar{h},p}(y)}\mathrm d y \nonumber\\& =  \left(1+o(1)\right)\frac{N^\frac{1}{4}e^{N H_N(m_*(N))}}{\sqrt{2\pi(1-m_*(N)^2)}}\int_{-\infty}^\infty e^{\eta_{\bar{\beta},\bar{h},p}(y)}\mathrm d y .
		\end{align*}
		This completes the proof of Lemma \ref{ex2}. 
	\end{proof}

\subsubsection{Completing the Proof of \eqref{eq:cltun_III}} As before, we start by computing the limiting moment generating function of $$N^{\frac{1}{4}}\left(\os - m_*(\beta,h,p)\right),$$ in the following lemma. 
	
	\begin{lem}\label{cltun133pr}
		For every $p$-special point $(\beta,h)\in \Theta$ and $\bar{\beta}, \bar{h} \in \mathbb{R}$, if $\bs \sim \mathbb{P}_{\beta+ N^{-\frac{3}{4}}\bar{\beta},~h + N^{-\frac{3}{4}}\bar{h},~p} $, then
		\begin{align}\label{complexp}
			\lim_{N \rightarrow \infty} & \e_{\beta+ N^{-\frac{3}{4}}\bar{\beta},~h + N^{-\frac{3}{4}}\bar{h},~p}  \left[e^{t N^{\frac{1}{4}}\left(\os - m_*(\beta,h,p)\right)} \right] \nonumber \\ 
			& \quad \quad = C_p(\bar \beta, \bar h, t) \exp\Big\{- t R_p(\bar \beta, \bar h, t)  + \eta_{\bar{\beta},\bar{h},p}\left( R_p(\bar \beta, \bar h, 0) -  R_p(\bar \beta, \bar h, t) \right) \Big\},
		\end{align}
		where $\eta_{\bar{\beta},\bar{h},p}$ is defined in the statement of Lemma \ref{ex2}, 
		$$C_p(\bar \beta, \bar h, t):=\frac{\int_{-\infty}^\infty e^{ \eta_{\bar{\beta},\bar{h}+t,p}(y) } \mathrm d y }{\int_{-\infty}^\infty e^{\eta_{\bar{\beta},\bar{h},p}(y) } \mathrm d y },$$
		and $R_p(\bar \beta, \bar h, t):=\left(\frac{6(\bar{\beta}pm_*^{p-1} +\bar{h}+t)}{H^{(4)}(m_*)}\right)^\frac{1}{3}$.
	\end{lem}
	\begin{proof} Once again, throughout this proof, we will denote $m_*(\beta,h,p)$ by $m_*$, $\beta + N^{-\frac{3}{4}} \bar{\beta}$ by $\beta_N$, and $h + N^{-\frac{3}{4}} \bar{h}$ by $h_N$. Fix $t \in \mathbb{R}$ and note that the moment generating function of $N^{\frac{1}{4}}\left(\os - m_*\right)$ at $t$ can be expressed as 
		\begin{equation}\label{cltstr1}
			\mathbb{E}_{\beta_N,h_N,p}e^{tN^{\frac{1}{4}}\left(\os - m_*\right)} = e^{-tN^{\frac{1}{4}}m_*}\frac{Z_N\left(\beta_N,h_N+N^{-\frac{3}{4}}t,p\right)}{Z_N(\beta_N,h_N,p)}.
		\end{equation}
		Using Lemma \ref{ex2} and the facts that $m_*(\beta_N,h_N,p) \rightarrow m_*$ and $m_*(\beta_N,h_N+N^{-\frac{3}{4}}t,p) \rightarrow m_*$, the right side of \eqref{cltstr1} simplifies to 
		\begin{align}\label{cltstr2}
			C_p(\bar \beta, \bar h, t)e^{-tN^{\frac{1}{4}}m_* + N\left\{H_{\beta_N, h_N+N^{-\frac{3}{4}}t,p} \left(m_*\left(\beta_N, h_N+N^{-\frac{3}{4}}t,p\right) \right) - H_{\beta_N, h_N,p} \left(m_*\left(\beta_N, h_N,p\right) \right) \right\} } (1+o(1)).
		\end{align}
		By Lemma \ref{mnminusm}, we have:
		\begin{equation}\label{cltstr3}
			N^\frac{1}{4} \left(m_*\left(\beta_N, h_N+N^{-\frac{3}{4}}t,p\right) - m_*\right) =  -R_p(\bar{\beta},\bar{h},t) + o(1).
		\end{equation}
		By a further Taylor expansion and using \eqref{mnminusm}, we have (denoting $H_N=H_{\beta_N, h_N,p} $), 
		\begin{align}\label{cltstr4}
			N\left\{H_{N} \left(m_*\left(\beta_N, h_N+N^{-\frac{3}{4}}t,p\right) \right) - H_{N} \left(m_*\left(\beta_N, h_N,p\right) \right) \right\}  = T_1 + T_2 + T_3 + T_4,
		\end{align} 
		where 
		\begin{align*}
			T_1&:=\tfrac{N}{2} \left\{ m_*\left(\beta_N, h_N+N^{-\frac{3}{4}}t,p\right) - m_*\left(\beta_N, h_N,p\right) \right \} ^2 H_{\beta_N,h_N,p}''\left(m_*\left(\beta_N, h_N,p\right)\right)\nonumber\\
			& = \tfrac{1}{2}\left\{ R_p(\bar \beta, \bar h, 0) - R_p(\bar \beta, \bar h, t) \right\} ^2\cdot \frac{1}{2}(6(\bar{\beta}pm_*^{p-1} + \bar{h}))^\frac{2}{3}\left(H^{(4)}(m_*)\right)^\frac{1}{3} + o(1), 
		\end{align*} 
		\begin{align*} 
			T_2 & := \tfrac{N}{6}\left\{ m_*\left(\beta_N, h_N+N^{-\frac{3}{4}}t,p\right) - m_*\left(\beta_N, h_N,p\right) \right \} ^3 H_{\beta_N,h_N,p}^{(3)}\left(m_*\left(\beta_N, h_N,p\right)\right)\nonumber\\
			& = - \tfrac{1}{6} \left\{ R_p(\bar \beta, \bar h, 0) - R_p(\bar \beta, \bar h, t) \right\} ^3(6(\bar{\beta}pm_*^{p-1} + \bar{h}))^\frac{1}{3}\left(H^{(4)}(m_*)\right)^\frac{2}{3} + o(1) , 
		\end{align*} 
		\begin{align*} 
			T_3 & := \tfrac{N}{24}\left\{ m_*\left(\beta_N, h_N+N^{-\frac{3}{4}}t,p\right) - m_*\left(\beta_N, h_N,p\right)\right\} ^4 H_{\beta_N,h_N,p}^{(4)}\left(m_*\left(\beta_N, h_N,p\right)\right)\nonumber\\ 
			& = \tfrac{1}{24} \left\{ R_p(\bar \beta, \bar h, 0) - R_p(\bar{\beta},\bar{h},t) \right\} ^4 H^{(4)}(m_*) + o(1), 
		\end{align*}
		and  
		$$T_4  := O(N\{ m_*(\beta_N, h_N+N^{-\frac{3}{4}}t,p) - m_*(\beta_N, h_N,p)\} ^5) = o(1).$$
		
		Now, using both \eqref{cltstr3} and \eqref{cltstr4}, we have 
		\begin{align}
			& N\left[H_{\beta_N, h_N+N^{-\frac{3}{4}}t,p} \left(m_*\left(\beta_N, h_N+N^{-\frac{3}{4}}t,p\right) \right) - H_{\beta_N, h_N,p} \left(m_*\left(\beta_N, h_N,p\right) \right) \right]\nonumber\\ 
			&= tN^{\frac{1}{4}}m_* - tR_p(\bar \beta, \bar h, t)  + \eta_{\bar{\beta},\bar{h},p}\left(R_p(\bar \beta, \bar h, 0) - R_p(\bar \beta, \bar h, t)\right) + o(1). \nonumber 
		\end{align}
		Using the above with  \eqref{cltstr1} and \eqref{cltstr2} Lemma \ref{cltun133pr} follows. 
	\end{proof}
	
	Although \eqref{complexp} is not readily recognizable as the moment generating function of any probability distribution, we will show below that it is indeed the moment generating function of the distribution $F_{\bar{\beta},\bar{h}}$ defined in \eqref{eq:beta_h_distribution}.
	
	\begin{lem}\label{wint}
		Let $F_{\bar{\beta},\bar{h}}$ be the distribution defined in \eqref{eq:beta_h_distribution}. Then, 
		\begin{align} 
			\label{cltstr6pr}
			&\int e^{tx} \mathrm d F_{\bar{\beta},\bar{h}}(x) =  C_p(\bar \beta, \bar h, t) \exp\Big\{- t R_p(\bar \beta, \bar h, t)  + \eta_{\bar{\beta},\bar{h},p}\left( R_p(\bar \beta, \bar h, 0) -  R_p(\bar \beta, \bar h, t) \right) \Big\}, 
		\end{align}
		with notations as in Lemma \ref{cltun133pr}.  
	\end{lem}
	\begin{proof} 
		Let us denote the right side of \eqref{cltstr6pr} by $M(t)$. Define
		\begin{align*}
			\Delta(t,y) :=  - t R_p(\bar \beta, \bar h, t)  + \eta_{\bar{\beta},\bar{h},p}\left( R_p(\bar \beta, \bar h, 0) -  R_p(\bar \beta, \bar h, t) \right)  + \eta_{\bar{\beta},\bar{h}+t,p}(y),
		\end{align*}
		Note that
		\begin{equation}\label{num}
			M(t) =\frac{\int_{-\infty}^\infty e^{\Delta(t,y)}\enskip \!\! \mathrm d y }{\int_{-\infty}^\infty e^{\eta_{\bar{\beta},\bar{h},p}(y) } \mathrm d y }.
		\end{equation}
		Using the change of variables $u = y- R_p(\bar \beta, \bar h, t)$ and a straightforward algebra, we have 
		\begin{equation}\label{integrand}
			\Delta(t,y) = \frac{H^{(4)}(m_*)}{24} u^4 + (\bar{\beta} p m_*^{p-1}+\bar{h})u + tu + \frac{(6(\bar{\beta} p m_*^{p-1}+\bar{h}))^{\frac{4}{3}}}{8 (H^{(4)}(m_*))^{\frac{1}{3}}} 
		\end{equation}
		and 
		\begin{equation}\label{integrand2}
			\eta_{\bar{\beta},\bar{h},p}(y)=\frac{H^{(4)}(m_*)}{24}u^4 + (\bar{\beta} p m_*^{p-1}+\bar{h}) u + \frac{\left(6(\bar{\beta}pm_*^{p-1}+\bar{h})\right)^{\frac{4}{3}}}{8(H^{(4)}(m_*))^{\frac{1}{3}}}.
		\end{equation}
		
		Lemma \ref{wint} now follows on substituting \eqref{integrand} and \eqref{integrand2} in \eqref{num}.
	\end{proof}
	
	The proof of \eqref{eq:cltun_III} now follows from Lemmas \ref{cltun133pr} and \ref{wint}. This completes the proof of Theorem \ref{cltun} when $(\beta, h)$ is $p$-special.

	\subsection{Proof of Theorem \ref{cltun} when $(\beta, h)$ is $p$-critical}\label{nonpuniqueclt}
	
	Throughout this section we assume that $(\beta, h) \in \Theta$ is $p$-critical. This means, by definition and \cite[Lemma B.1]{cmp}, that the function $H = H_{\beta,h,p}$ has $K \in \{2, 3\}$ global maximizers, which we denote by  $m_1< \ldots<m_K$. It also follows from Lemma \ref{derh44}, that for sequences $\beta_N \rightarrow \beta$ and $h_N \rightarrow h$, the function $H_N := H_{\beta_N,h_N,p}$, for all large $N$, has local maximizers at $m_1(N), \ldots,m_K(N)$ such that $m_k(N) \rightarrow m_k$, as $N \rightarrow \infty$, for all $1\leq k \leq K$. As before, $\bar{\p}$ and $\bar{Z}_N$ will denote $\p_{\beta_N,h_N,p} $ and $Z_N(\beta_N,h_N,p)$, respectively.

	In presence of multiple global maximizers, the magnetization $\os$ will concentrate around the set of all global maximizers.  In fact, we can prove the following stronger result: Consider an open interval $A$ around a local maximizer $m$ such that $m$ is the unique global maximizer of $H$ over $A$. Then conditional on the event  $\os \in A$ (which is a rare event if $m$ is not a global maximizer), $\os$ concentrates around $m$. This is the first step in the proof of Theorem \ref{cltun} when $(\beta, h)$ is $p$-critical. To state the result precisely, assume that $m$ is a local maximizer of $H$ and let $m(N)$ be local maximizers of $H_N$ converging to $m$, which exist by Lemma \ref{derh44}. Define 
	\begin{align}\label{eq:An}
		A_{N,\alpha}(m(N)) = \left(m(N)-N^{-\frac{1}{2} + \alpha},m(N) +N^{-\frac{1}{2} + \alpha}\right). 
	\end{align} 
	The following lemma gives the conditional and, hence, the unconditional, concentration result of $\os$ around local maximizers.\footnote{The unconditional concentration derived in  \eqref{eq:multi_max_local_II} is not required in the proof of Theorem \ref{cltun}. Nevertheless, we include it for the sake of completeness.}
	
	\begin{lem}\label{lem:multiple_max} Suppose $(\beta, h) \in \Theta$ is $p$-critical. Then for $\alpha \in \left(0,\frac{1}{6}\right]$ fixed and $A_{N,\alpha}(m(N))$ as defined in \eqref{eq:An}, 
		\begin{equation}\label{eq:multi_max_local_I}
			\bar{\mathbb P}\left(\os \in A_{N,\alpha}(m(N))^c\big|
			\os \in A\right) = \exp\left\{\frac{1}{3}N^{2\alpha} H''(m) \right\}O(N^{\frac{3}{2}}),
		\end{equation} 
		for any interval $A \subseteq [-1,1]$ such that $m \in \textrm{int}(A)$ and $H(m) > H(x)$, for all $x \in \mathrm{cl}(A) \setminus \{m\}$\footnote{For any set $A \subseteq \R$, $\textrm{int}(A)$ and $\mathrm{cl}(A)$ denote the topological interior and closure of $A$, respectively.} As a consequence, for $A_{N,\alpha,K} := \bigcup_{k=1}^K A_{N,\alpha}(m_k(N))$, 	
		\begin{equation}\label{eq:multi_max_local_II}
			\bar{\mathbb P}\left(\os \in A_{N,\alpha,K}^c\right) = \exp\left\{\frac{1}{3}N^{2\alpha} \max_{1\leq k\leq K} H''(m_k) \right\}O(N^{\frac{3}{2}}).
		\end{equation}
	\end{lem} 
	
	\begin{proof}
		It follows from Lemma \ref{derh44}, that for all $N$ sufficiently large, $H_N(m(N)) > H_N(x)$ for all $x \in \mathrm{cl}(A) \setminus \{m(N)\}$, whence we can apply \cite[Lemma B.11]{cmp} to conclude that $$\sup_{x\in A\setminus A_{N,\alpha}(m(N))} H_N(x) = H_N\left(m(N) \pm N^{-\frac{1}{2}+\alpha}\right),$$ for all large $N$ such that $A_{N,\alpha}(m(N))\subset A$, as well. Following the proof of Lemma \ref{conc}, we have for all large N ,
		\begin{align}\label{coinnonun1}
			\bar{\mathbb P}\left(\os \in A_{N,\alpha}(m(N))^c\big|
			\os \in A\right) & \leq   \exp\left\{N\left(\sup_{x\in A\setminus A_{N,\alpha}(m(N))} H_N(x) - \sup_{x\in A} H_N(x)\right)\right\} O(N^{\frac{3}{2}})\nonumber\\& =  \exp\left\{N\left( H_N\left(m(N) \pm N^{-\frac{1}{2}+\alpha}\right) - H_N\left(m(N)\right)\right)\right\} O(N^{\frac{3}{2}})\nonumber\\& \leq  \exp\left\{\frac{N}{3}\left(N^{-1+2\alpha} H''(m) +  O\left(N^{-\frac{3}{2} + 3\alpha}\right) \right)\right\} O(N^{\frac{3}{2}}).
		\end{align} 
		The result \eqref{eq:multi_max_local_I} now follows from \eqref{coinnonun1}.
		
		Next, we proceed to prove \eqref{eq:multi_max_local_II}. Let $A_1 := [-1,(m_1+m_2)/2)$, $A_K := [(m_{K-1}+m_K)/2,1]$ and for $1< k< K$, $A_k := [(m_{k-1} + m_k)/2,(m_k+m_{k+1})/2)$. Then, $A_1,A_2,\ldots, A_K$ are disjoint intervals uniting to $[-1,1]$, $m_k \in \textrm{int}(A_k)$, and $H(m_k) > H(x)$ for all $x\in \mathrm{cl}(A_k) \setminus \{m_k\}$ and all $1\leq k\leq K$. Hence, by Lemma \ref{lem:multiple_max}, $$\bar{\mathbb P}\left(\os \in A_{N,\alpha}(m_k(N))^c\big|
		\os \in A_k\right) = \exp\left\{\frac{1}{3}N^{2\alpha} H''(m_k) \right\}O(N^{\frac{3}{2}})\quad\textrm{for all}~ 1\leq k\leq K.$$ Since $A_{N,\alpha}(m_k(N)) \subset A_k$ for all $1\leq k \leq K$, for all large $N$, we have $A_{N,\alpha}(m_k(N))^c \bigcap A_k = A_{N,\alpha,K}^c \bigcap A_k$ for all $1\leq k \leq K$, for all large $N$ (recall the definition of $ A_{N,\alpha,K}$ from the statement of Lemma \ref{lem:multiple_max}). Hence, $\bar{\mathbb P}\left(\os \in A_{N,\alpha}(m_k(N))^c\big|
		\os \in A_k\right) = \bar{\mathbb P}\left(\os \in A_{N,\alpha,K}^c\big|
		\os \in A_k\right)$ for all $1\leq k \leq K$, for all large $N$. Hence, for all large $N$, we have
		\begin{equation}\label{condtouncond}
			\bar{\mathbb P}\left(\os \in A_{N,\alpha,K}^c\big|
			\os \in A_k\right) = \exp\left\{\frac{1}{3}N^{2\alpha} H''(m_k) \right\}O(N^{\frac{3}{2}})\quad\textrm{for all}~ 1\leq k\leq K.
		\end{equation}
		It follows from \eqref{condtouncond} that for all large $N$,
		\begin{align}\label{lstp1}
			\bar{\p}(\os \in A_{N,\alpha,K}^c) & =  \sum_{k=1}^K \bar{\mathbb P}\left(\os \in A_{N,\alpha,K}^c\big|
			\os \in A_k\right) \bar{\p} (\os \in A_k)\nonumber\\& \leq  \exp\left\{\frac{1}{3}N^{2\alpha} \max_{1\leq k\leq K}H''(m_k) \right\}O(N^{\frac{3}{2}})\sum_{k=1}^K \bar{\p} (\os \in A_k)\nonumber\\& =  \exp\left\{\frac{1}{3}N^{2\alpha} \max_{1\leq k\leq K}H''(m_k) \right\}O(N^{\frac{3}{2}}).
		\end{align}
		The result in \eqref{eq:multi_max_local_II} now follows from \eqref{lstp1}, completing the proof of Lemma \ref{lem:multiple_max}. 
	\end{proof}
	In order to derive a conditional CLT of $\os$ around the local maximizer $m$, given that $m$ is in $A$ (where $A$ is as in Lemma \ref{lem:multiple_max} above), we need precise estimates of the \textit{restricted partition functions} defined as
	$$\bar{Z}_N\big|_A := \frac{1}{2^N}\sum_{\bs\in \sa_N: \os \in A} \exp\left\{N(\beta_N \overline{\sigma}^p_N + h_N\os)  \right\}.$$ Note that $\bar{Z}_N\big|_A$ is the partition function of the conditional measure $\bar{\mathbb P}\left(\bs \in \cdot\big|\os \in A \right)$, in the sense that for any $\bt = (\tau_1, \tau_2, \ldots, \tau_N) \in \sa_N$ such that $\bar{\bt} \in A$, we have
	$$\bar{\mathbb P}\left(\bs = \bt\big|\os \in A \right) = \frac{1}{2^{N} \bar{Z}_N\big|_{A}} \exp\left\{N(\beta_N \overline{\sigma}^p_N + h_N\os) \right\}.$$ The following lemma gives an approximation of the restricted and, hence, the unrestricted partition functions. To this end, recall that $m(N)$ is a local maximizer of $H_N$ converging to $m$.

	\begin{lem}\label{lm:condpart} Suppose  $(\beta, h) \in \Theta$ is $p$-critical. Then for $\alpha > 0$ and $N$ large enough, the restricted partition function can be expanded as
		\begin{equation}\label{eq:multi_max_I}
			\bar{Z}_N\big|_A =  \frac{e^{NH_N(m(N))}}{\sqrt{(m(N)^2-1)H_N''(m(N))}}\left(1+O\left(N^{-\frac{1}{2} +\alpha}\right)\right),
		\end{equation}
		where the set $A$ is as in Lemma \ref{lem:multiple_max}. This implies, for every $\alpha > 0$ and $N$ large enough, the (unrestricted) partition function can be expanded as 
		\begin{equation}\label{eq:multi_max_II}
			\bar{Z}_N =  \sum_{k=1}^K\frac{e^{NH_N(m_k(N))}}{\sqrt{(m_k(N)^2-1)H_N''(m_k(N))}}\left(1+O\left(N^{-\frac{1}{2} +\alpha}\right)\right). 
		\end{equation}
	\end{lem}
	
	\begin{proof}
		The arguments below are meant for all sufficiently large $N$. Without loss of generality, let $\alpha \in \left(0,\frac{1}{6}\right]$ and note that
		\begin{align}\label{exp1}
			& \bar{\p}\left(\os \in A_{N,\alpha}(m(N))\Big|\os \in A\right)\nonumber\\ & =  \bar{Z}_N\big|_A^{-1} \sum_{m \in \mathcal{M}_N \bigcap A_{N,\alpha}(m(N))} \binom{N}{N(1+m)/2} \exp\left\{N(\beta_N m^p + h_N m - \log 2) \right\}.
		\end{align}
		By Lemma \ref{lem:multiple_max}, $\bar{\p}\left(\os \in A_{N,\alpha}(m(N))\Big|\os \in A\right) = 1-O(e^{-N^\alpha})$ and hence \eqref{exp1} gives us
		\begin{equation}\label{secc}
			\bar{Z}_N\big|_A = \left(1+O(e^{-N^\alpha})\right) \sum_{m \in \mathcal{M}_N \bigcap A_{N,\alpha}(m(N))} \binom{N}{N(1+m)/2} \exp\left\{N(\beta_N m^p + h_N m - \log 2) \right\}.
		\end{equation}
		Since $m(N)$ is the unique global maximizer of $H_N$ over the interval $A_{N,\alpha}(m(N))$,  by mimicking the proof of Lemma \ref{ex} on the interval $A_{N,\alpha}(m(N))$, it follows that 
		\begin{align}\label{nonunmim1}
			&\sum_{m \in \mathcal{M}_N\bigcap A_{N,\alpha}(m(N))} \binom{N}{N(1+m)/2} \exp\left\{N(\beta_N m^p + h_Nm-\log 2) \right\}\nonumber\\&= \frac{e^{N H_N(m(N))}}{\sqrt{(m(N)^2-1)H_N''(m(N))}}\left(1+O\left(N^{- \frac{1}{2}+3\alpha}\right)\right).
		\end{align} 
		The result in \eqref{eq:multi_max_I} now follows from \eqref{secc} and \eqref{nonunmim1}. 
		
		For each $1\leq k\leq K$, (\ref{eq:multi_max_I}) immediately gives us 
		\begin{equation}\label{tobeadded}
			\bar{Z}_N\big|_{A_k} =  \frac{e^{NH_N(m_k(N))}}{\sqrt{(m_k(N)^2-1)H_N''(m_k(N))}}\left(1+O\left(N^{-\frac{1}{2} +\alpha}\right)\right),
		\end{equation}
		where the sets $A_1, \ldots,A_K$ are as defined in the proof of \eqref{eq:multi_max_local_II}. The result in \eqref{eq:multi_max_II} now follows from \eqref{tobeadded} on observing that $\bar{Z}_N = \sum_{k=1}^K \bar{Z}_N\big|_{A_k}$. 
	\end{proof}

	Lemmas \ref{lem:multiple_max} and \ref{lm:condpart} can now be used to complete the proof of Theorem \ref{cltun} (2). \\

	\noindent\textbf{\textit{Completing the Proof of Theorem $\ref{cltun}$ when $(\beta, h)$ is $p$-critical}}: For each $\varepsilon > 0$ and $1\leq s \leq K$, define $B_{s, \varepsilon} = (m_s -\varepsilon, m_s+\varepsilon)$. Then for all $\varepsilon>0$ small enough, $H(m_s) > H(x)$, for all $x \in B_{s,\varepsilon}\setminus \{m_s\}$. Now, for each $1\leq s \leq K$, we have
	\begin{equation}\label{smallprob}
		\p_{\beta,h,p} (\os \in B_{s,\varepsilon}) = \frac{Z_N(\beta,h,p)\big|_{B_{s,\varepsilon}}}{Z_N(\beta,h,p)}.
	\end{equation}
	By Lemma \ref{lm:condpart} we have 
	\begin{equation}\label{smallprob1}
		Z_N(\beta,h,p)\big|_{B_{s,\varepsilon}} =  \frac{e^{N\sup_{x\in[-1,1]} H(x)}}{\sqrt{(m_s^2-1)H''(m_s)}}\left(1+o(1)\right)\quad\textrm{for all} ~1\leq s \leq K, 
	\end{equation} 
	and 
	\begin{equation}\label{smallprob2}
		Z_N(\beta,h,p) = e^{N\sup_{x\in[-1,1]} H(x)} \sum_{s=1}^K \frac{1}{\sqrt{(m_s^2-1)H''(m_s)}}\left(1+o(1)\right).
	\end{equation}
	The result in \eqref{eq:cltun_p1} now follows from \eqref{smallprob}, \eqref{smallprob1} and \eqref{smallprob2}.
	
	Now, we proceed we prove \eqref{eq:cltun_II}. Hereafter, let $\beta_N := \beta+ N^{-\frac{1}{2}} \bar{\beta}$ and $h_N := h + N^{-\frac{1}{2}} \bar{h}$. A direct calculation reveals that 
	\begin{equation}\label{condexp}
		\e_{\beta_N,h_N,p}  \left[e^{tN^{\frac{1}{2}}(\os-m)}\Big| \os \in A  \right] =  e^{-tN^{\frac{1}{2}}m} \frac{Z_N(\beta_N,h_N + N^{-\frac{1}{2}} t,p)\big|_A}{Z_N(\beta_N,h_N,p)\big|_A}.
	\end{equation}
	Using Lemma \ref{lm:condpart}, the right side of \eqref{condexp} simplifies to
	\begin{equation*}
		(1+o(1)) e^{-tN^{\frac{1}{2}}m +  N\left\{H_{\beta_N, h_N+N^{-\frac{1}{2}}t,p} \left(m\left(\beta_N, h_N+N^{-\frac{1}{2}}t,p\right) \right) - H_{\beta_N, h_N,p} \left(m\left(\beta_N, h_N,p\right) \right)  \right\}},
	\end{equation*}
	where $m(\beta_N, h_N,p )$ and $m(\beta_N, h_N+N^{-\frac{1}{2}}t,p)$ are the local maximizers of the functions $H_{\beta_N, h_N,p}$ and $H_{\beta_N, h_N+N^{-\frac{1}{2}}t,p}$ respectively, converging to $m$. We can mimic the proof of Theorem \ref{cltun} verbatim from this point onward, to conclude that as $N \rightarrow \infty$,   
	\begin{equation}\label{moment generating functionnonun}
		\e_{\beta_N,h_N,p}  \left[e^{tN^{\frac{1}{2}}(\os-m)}\Big| \os \in A  \right] \rightarrow 	\exp\left\{ -\frac{t(\bar{h}+\bar{\beta}pm^{p-1})}{H''(m)} -\frac{t^2}{2H''(m)}\right\}.
	\end{equation}
	The result in \eqref{eq:cltun_II} now follows from \eqref{moment generating functionnonun}. \qed

	\section{Perturbative Concentration Lemmas when $(\beta, h)$ is $p$-critical}
	\label{sec:pf_concentration}
	
	It was shown in \eqref{eq:cltun_p1} that for $(\beta,h) \in \Theta$ which is $p$-critical, the limiting distribution of $\os$ assigns positive mass to each of the global maximizers $m_1,m_2,\ldots,m_K$. However, to use this result to obtain the limiting distribution of the ML estimates, we need to derive a similar concentration for $\os$ under $\p_{\beta_N,h_N,p}$. In particular, is it the case that $\os$ assigns positive mass to each of $m_1,m_2,\ldots,m_K$, or is the asymptotic support of $\os$ in this case a proper subset of $\{m_1,m_2,\ldots,m_K\}$ (we already know from \eqref{eq:multi_max_local_II} that the asymptotic support of $\os$ is a subset of  $\{m_1,m_2,\ldots,m_K\}$)? The answer to this question depends upon the rate of convergence of $(\beta_N,h_N)$ to $(\beta,h)$. This section is devoted to deriving these concentration results, which will be essential in proving the asymptotic distributions of $\hat{\beta}_N$ and $\hat{h}_N$ at the critical points.
	
	In what follows, assume  $(\beta,h) \in \Theta$ which is $p$-critical and let $m_1<m_2<\ldots<m_K$ be the global maximizers of $H_{\beta,h,p}$, and let $A_1,A_2,\ldots,A_K$ be the sets defined in the proof of \eqref{eq:multi_max_local_II}.  The following lemma shows that keeping $\beta$ fixed, if $h$ is perturbed at a rate slower than $1/N$, then under the perturbed sequence of measures, $\os$ concentrates around the largest/smallest global maximizer according as the perturbation is in the positive/negative direction, respectively.

	\begin{lem}\label{redunsup}
		For any positive sequence $y_N$ satisfying $N^{-1} \ll  y_N \ll 1$, there exist positive constants $C_1$ and $C_2$ not depending on $N$, such that
		$$\p_{\beta,h+\bar{h}y_N,p}\left(\os \in A_{\bm{1}(\bar{h}<0) + K \bm{1}(\bar{h}>0)}^c\right) \leq C_1e^{-C_2 Ny_N}.$$
	\end{lem} 
	
	\begin{proof} Let $H_N:= H_{\beta,h+\bar{h}y_N,p}$ and $m_k(N)$ be the local maximizers of $H_N$ converging to $m_k$. In what follows, for two positive sequences $\phi_N$ and $\psi_N$, we will use the notation $\phi_N \lesssim \psi_N$ to denote that $\phi_N \leq C \psi_N$ for all $N$ and some constant $C$ not depending on $N$. Let $t := \bm{1}(\bar{h}<0) + K \bm{1}(\bar{h}>0)$. Then for any $s \neq t$, we have by Lemma \ref{lm:condpart} and Lemma \ref{bdhbest},
		\begin{align}\label{compact}
			\p_{\beta,h+\bar{h}y_N,p}(\os \in A_s) & =  \frac{Z_N(\beta,h+\bar{h}y_N,p)\big|_{A_s}}{Z_N(\beta,h+\bar{h}y_N,p)}\nonumber\\& \leq   \frac{Z_N(\beta,h+\bar{h}y_N,p)\big|_{A_s}}{Z_N(\beta,h+\bar{h}y_N,p)\big|_{A_t}}\nonumber \\&\lesssim  \sqrt{\frac{(m_t(N)^2-1)H_N''(m_t(N))}{(m_s(N)^2-1)H_N''(m_s(N))}} e^{N\left[H_N(m_s(N)) - H_N(m_t(N)) \right]}\nonumber\\&\lesssim  e^{N\left[\bar{h}y_N(m_s-m_t) + O(y_N^2)\right]} = e^{Ny_N\left[\bar{h}(m_s-m_t) + o(1)\right]}.
		\end{align} 
		Lemma \ref{redunsup} now follows from \eqref{compact}, since $\bar{h}(m_s-m_t) <0$ for every $s \neq t$, by definition. 
	\end{proof}

	The situation becomes a bit trickier when $h$ is fixed and $\beta$ is perturbed, as two cases arise depending upon the parity of $p$. The case $p \geq 3$ is odd, is the easier one, and is exactly similar to the previous setting. Note that in this case, $K = 2$.
	\begin{lem}\label{redunsup2}
		Suppose that $p\geq 3$ is odd. Then, for any positive sequence $x_N$ satisfying $N^{-1} \ll  x_N \ll 1$, there exist positive constants $C_1$ and $C_2$ not depending on $N$, such that
		$$\p_{\beta+\bar{\beta}x_N,h,p}\left(\os \in A_{\bm{1}\{\bar{\beta}<0\} + 2 \cdot \bm{1}\{\bar{\beta}>0\}}^c\right) \leq C_1e^{-C_2 Nx_N}.$$
	\end{lem}

	\begin{proof} 
		Let $H_N:= H_{\beta+\bar{\beta}x_N,h,p}$ and $m_k(N)$ be the local maximizers of $H_N$ converging to $m_k$. Then for any $s \neq t := \bm{1}\{\bar{\beta}<0\} + 2 \cdot \bm{1}\{\bar{\beta}>0\}$, by exactly following the proof of Lemma \ref{redunsup}, one gets
		\begin{equation}\label{compact2}
			\p_{\beta+\bar{\beta}x_N,h,p}(\os \in A_s) \lesssim e^{Nx_N\left[\bar{\beta}(m_s^p-m_t^p) + o(1)\right]}.
		\end{equation}
		Lemma \ref{redunsup2} now follows from \eqref{compact2}, since $\bar{\beta}(m_s^p-m_t^p) <0$ for every $s \neq t$, by definition.  
	\end{proof}

	In the following lemma, we deal with the case $p\geq 4$ even. The result is presented in two cases, depending upon whether $h=0$ or not. Note that, if $h \neq 0$, then $K=2$. On the other hand, if $h = 0$, then we may assume that $\beta \geq \tilde{\beta}_p$, since otherwise, $(\beta,h)$ is $p$-regular.  In this case, $K=2$ if $\beta > \tilde{\beta}_p$ and $K=3$ if $\beta = \tilde{\beta}_p$.

	\begin{lem}\label{redunsup3}
		The following hold when $p \geq 4$ is even.
		\begin{enumerate}
			\item[$(1)$]  Suppose that $h \neq 0$. Then, for any positive sequence $x_N$ satisfying $N^{-1} \ll  x_N \ll 1$, there exist positive constants $C_1$ and $C_2$ not depending on $N$, such that the following hold.
			\begin{itemize}
				\item[$\bullet$] If $h>0$, then
				$$\p_{\beta+\bar{\beta}x_N,h,p}\left(\os \in A_{\bm{1}\{\bar{\beta}<0\} + 2 \cdot \bm{1}\{\bar{\beta}>0\}}^c\right) \leq C_1e^{-C_2 Nx_N}. $$
				\item[$\bullet$] If $h < 0$, then
				$$\p_{\beta+\bar{\beta}x_N,h,p}\left(\os \in A_{\bm{1}\{\bar{\beta}>0\} + 2 \cdot \bm{1}\{\bar{\beta}<0\}}^c\right) \leq C_1e^{-C_2 Nx_N}. $$
			\end{itemize}
			\item[$(2)$] Suppose that $h = 0$. 
			\begin{itemize}
				\item[$\bullet$] If $\beta > \tilde{\beta}_p$, then for any sequence $(\beta_N,h_N) \rightarrow (\beta,h)$, there exists a positive constant $C$ not depending on $N$, such that
				\begin{equation}\label{cc1}
					\max\left\{\left|\p_{\beta_N,h_N,p}\left(\os \in A_1\right) -\frac{1}{2}\right| ,~ \left|\p_{\beta_N,h_N,p}\left(\os \in A_2\right) -\frac{1}{2}\right|\right\} \leq Ce^{-N^{\frac{1}{6}}}.
				\end{equation}
				\item[$\bullet$] If $\beta = \tilde{\beta}_p$ and $\bar{\beta} > 0$, then for any positive sequence $x_N$ satisfying $N^{-1} \ll  x_N \ll 1$, there exist positive constants $C_1$ and $C_2$ not depending on $N$, such that
				\begin{equation}\label{cc2}
					\max\left\{\left|\p_{\beta_N,h,p}\left(\os \in A_1\right)-\frac{1}{2}\right|,~\left|\p_{\beta_N, h,p}\left(\os \in A_3\right) - \frac{1}{2}\right|\right\} \leq C_1e^{-C_2 Nx_N},
				\end{equation}
				where $\beta_N=\beta + \bar{\beta} x_N$. 
				\item[$\bullet$] If $\beta = \tilde{\beta}_p$ and $\bar{\beta} < 0$, then for any positive sequence $x_N$ satisfying $N^{-1} \ll  x_N \ll 1$, there exist a positive constants $C_1$ and $C_2$ not depending on $N$, such that
				\begin{equation}\label{cc3}
					\p_{\beta+\bar{\beta}x_N,h,p}\left(\os \in A_2^c\right) \leq C_1 e^{-C_2 N x_N}.
				\end{equation}
			\end{itemize}
		\end{enumerate}
		
	\end{lem} 
	
	\begin{proof}
		The proof of (1) is exactly similar to that of Lemma \ref{redunsup2}, and hence we ignore it. One only needs to observe that $m_1<m_2<0$ if $h <0$, and $0<m_1<m_2$ if $h> 0$. Hence, $m_1^p < m_2^p$ if $h >0$, and $m_1^p > m_2^p$ if $h<0$.

		Next, we prove (2). Note that \eqref{cc1} follows directly from \eqref{eq:multi_max_local_II} in Lemma \ref{lem:multiple_max} (taking $\alpha = \frac{1}{6}$) and using the fact that for even $p$ and $h=0$, $\os \stackrel{D}{=} -\os$. Next, note that if $\beta = \tilde{\beta}_p$ and $\bar{\beta} > 0$, then from \eqref{compact2},
		$\p_{\beta + \bar{\beta} x_N,h,p}\left(\os \in A_2\right) \lesssim e^{-C_2 Nx_N}$ for some positive constant $C_2$ not depending on $N$, and \eqref{cc2} follows from the symmetry of the distribution of $\os$. Finally, if $\beta = \tilde{\beta}_p$ and $\bar{\beta} < 0$, then once again from \eqref{compact2}, $\p_{\beta + \bar{\beta} x_N,h,p}\left(\os \in A_k\right) \lesssim e^{-C_2 N x_N}$ for some positive constant $C_2$ not depending on $N$ and $k \in \{1,3\}$. This gives \eqref{cc3} and completes the proof of Lemma \ref{redunsup3}. 
	\end{proof}

	\section{Monotonicity of the Likelihood and Existence of ML Estimates}
\label{sec:increasingML}

In this section we collect some results about the monotonicity of the likelihood equations and existence of the ML estimates. To this end, define $u_{N,p}$ and $u_{N,1}$ to be the functions appearing in the LHS of appearing in the equations \eqref{eqmle} and \eqref{eqmleh}, respectively, that is, 
\begin{equation}\label{eq:ubetah}
u_{N,p}(\beta, h, p) := \e_{\beta,h,p} \left(\overline{\sigma}^p_N \right)  \text{ and }  u_{N, 1}(\beta, h, p) := \e_{\beta,h,p} \left(\os \right)  . 
\end{equation}

	\begin{lem}\label{increasing}
		For every fixed $h$, the function $\beta \mapsto F_N(\beta,h,p)$ is strictly convex, and for every fixed $\beta$, the function $h \mapsto F_N(\beta,h,p)$ is strictly convex. Consequently, the maps $u_{N,1}$ and $u_{N,p}$ defined in \eqref{eq:ubetah} are strictly increasing in both $\beta$ and $h$.
	\end{lem}
	
	\begin{proof}
		Let $\psi_N(\beta,h) := F_N(\beta,h,p) + N\log 2 = \log \sum_{\bs \in \sa_N} e^{N\beta \overline{\sigma}^p_N + Nh\os}$. Then for every $\beta_1, \beta_2, h$ and $\lambda \in (0,1)$, we have by H\"older's inequality,
		\begin{align*}
			\psi_N(\lambda \beta_1 + (1-\lambda)\beta_2,h)& =  \log \sum_{\bs \in \sa_N} e^{N\lambda (\beta_1 \overline{\sigma}^p_N + h\os)} e^{N(1-\lambda)(\beta_2 \overline{\sigma}^p_N + h \os)}\\& <  \log \left[\left(\sum_{\bs \in \sa_N} e^{N\beta_1 \overline{\sigma}^p_N + Nh\os} \right)^\lambda \left(\sum_{\bs \in \sa_N} e^{N\beta_2 \overline{\sigma}^p_N + Nh\os} \right)^{1-\lambda}\right]\\& =  \lambda \psi_N(\beta_1,h) + (1-\lambda)\psi_N(\beta_2,h).
		\end{align*}
		Similarly, for every $h_1,h_2,\beta$ and $\lambda \in (0,1)$, we have by H\"older's inequality,
		\begin{align*}
			\psi_N(\beta,\lambda h_1 + (1-\lambda)h_2)& =  \log \sum_{\bs \in \sa_N} e^{N\lambda (\beta \overline{\sigma}^p_N + h_1\os)} e^{N(1-\lambda)(\beta \overline{\sigma}^p_N + h_2 \os)}\\& <  \log \left[\left(\sum_{\bs \in \sa_N} e^{N\beta \overline{\sigma}^p_N + Nh_1\os} \right)^\lambda \left(\sum_{\bs \in \sa_N} e^{N\beta \overline{\sigma}^p_N + Nh_2\os} \right)^{1-\lambda}\right]\\& =  \lambda \psi_N(\beta,h_1) + (1-\lambda)\psi_N(\beta,h_2).
		\end{align*}
		This shows strict convexity of the functions $\beta \mapsto F_N(\beta,h,p)$ and $h \mapsto F_N(\beta,h,p)$. Now, note that $$\frac{\partial}{\partial \beta} F_N(\beta,h,p) = Nu_{N,p}(\beta,h,p)\quad\textrm{and}\quad \frac{\partial}{\partial h} F_N(\beta,h,p) = Nu_{N,1}(\beta,h,p).$$ Lemma \ref{increasing} now follows from the fact that the first derivative of a differentiable, strictly convex function is strictly increasing.
	\end{proof}
	
	Next, we show that for fixed $\beta$, the ML Estimate  of $h$ exists, and for fixed $h$, the ML Estimate  of $\beta$ exists, asymptotically almost surely. However, if $p$ is even, then the joint ML Estimate  of $(\beta,h)$ does not exist.
	
	\begin{lem}\label{mleexist}
		Fix $N \geq 1$. Then $\hat{\beta}_N$ and $\hat{h}_N$ exist in $[-\infty,\infty]$ and are unique.  Further, $\hat{h}_N $ exists in $\mathbb{R}$ if and only if $|\os| < 1$. For odd $p$, $\hat{\beta}_N$ exists in $\mathbb{R}$ if and only if $|\os| < 1$, and for even $p$, $\hat{\beta}_N$ exists in $\mathbb{R}$ if and only if $\os \notin \{-1,0,1\}$. Hence, $$\lim_{N \rightarrow \infty} \p_{\beta,h,p} \left(\hat{\beta}_N~\textrm{exists in}~\mathbb{R}\right) = \lim_{N \rightarrow \infty} \p_{\beta,h,p} \left(\hat{h}_N~\textrm{exists in}~\mathbb{R}\right) =1.$$ However, if $p$ is even, then for all  $N \geq 1$ and all $\bs \in \sa_N$, the joint ML Estimate  of $(\beta,h)$ does not exist.
	\end{lem}
	\begin{proof}
		The log-likelihood function is given by
		$$\ell_p(\beta,h|\bs) = - N \log 2 + \beta N\overline{\sigma}^p_N + hN\os - F_N(\beta,h,p).$$ By Lemma \ref{increasing}, the functions $\beta \mapsto F_N(\beta,h,p)$ and $h \mapsto F_N(\beta,h,p)$ are strictly convex, and hence, the functions $\beta \mapsto \ell_p(\beta,h|\bs)$ and $h \mapsto \ell_p(\beta,h|\bs)$ are strictly concave. Consequently, $\beta \mapsto \ell_p(\beta,h|\bs)$ attains maximum at $\hat{\beta} \in \mathbb{R}$ if and only if $\frac{\partial}{\partial \beta} \ell_p(\beta,h|\bs)\Big|_{\beta = \hat{\beta}} = 0$, and $h \mapsto \ell_p(\beta,h|\bs)$ attains maximum at $\hat{h} \in \mathbb{R}$ if and only if $\frac{\partial}{\partial h} \ell_p(\beta,h|\bs)\Big|_{h = \hat{h}} = 0$. In those cases, $\hat{\beta}$ and $\hat{h}$ are the unique maximizers of $\ell_p(\beta,h|\bs)$ over $\beta \in \mathbb{R}$ and $h \in \mathbb{R}$, respectively. Now, the equations $\frac{\partial}{\partial \beta} \ell_p(\beta,h|\bs) = 0$ and $\frac{\partial}{\partial h} \ell_p(\beta,h|\bs) = 0$ are (respectively) equivalent to the equations
		\begin{equation}\label{mleeq1}
			\frac{\partial}{\partial \beta} F_N(\beta,h,p) = N\overline{\sigma}^p_N\quad\textrm{and}\quad \frac{\partial}{\partial h} F_N(\beta,h,p) = N\os.
		\end{equation}
		One can easily show that
		\begin{equation}\label{hinm}
			\lim_{h \rightarrow \infty} \frac{\partial}{\partial h} F_N(\beta,h,p) = N\quad\textrm{and}\quad \lim_{h \rightarrow -\infty} \frac{\partial}{\partial h} F_N(\beta,h,p) = -N.
		\end{equation}
		Similarly, if $p$ is odd, we have
		\begin{equation}\label{binm1}
			\lim_{\beta \rightarrow \infty} \frac{\partial}{\partial \beta} F_N(\beta,h,p) = N\quad\textrm{and}\quad \lim_{\beta \rightarrow -\infty} \frac{\partial}{\partial \beta} F_N(\beta,h,p) = -N.
		\end{equation}
		Finally, for even $p$, we have
		\begin{equation}\label{binm2}
			\lim_{\beta \rightarrow \infty} \frac{\partial}{\partial \beta} F_N(\beta,h,p) = N\quad\textrm{and}\quad \lim_{\beta \rightarrow -\infty} \frac{\partial}{\partial \beta} F_N(\beta,h,p) = 0.
		\end{equation}
		The existence and uniqueness of $\hat{h}_N$ and $\hat{\beta}_N$ in $[-\infty,\infty]$, and the necessary and sufficient conditions about the existence of $\hat{h}_N$ and $\hat{\beta}_N$ in $\mathbb{R}$ now follow from \eqref{mleeq1}, \eqref{hinm}, \eqref{binm1} and \eqref{binm2}, since the functions $h \mapsto \frac{\partial}{\partial h} F_N(\beta,h,p)$ and $\beta \mapsto \frac{\partial}{\partial \beta} F_N(\beta,h,p)$ are strictly increasing and continuous. 
		
		Next, we show that the ML estimates are real valued with probability (under $\beta,h$) going to $1$. Towards this, first note that under $\p_{\beta,h,p} $, $\os$ converges weakly to a discrete measure supported on the set of all global maximizers of $H_{\beta,h,p}$ (see Theorem \ref{cltun} (2)). Since $-1$ and $1$ are not global maximizers of $H_{\beta,h,p}$, it follows that $\p_{\beta,h,p} (|\os| = 1) \rightarrow 0$ as $N \rightarrow \infty$. If $h\neq 0$, then $0$ is not a global maximizer of $H_{\beta,h,p}$, so $\p_{\beta,h,p} (\os \in \{-1,0,1\}) \rightarrow 0$ as $N \rightarrow \infty$. Therefore, assume that $h=0$. By Stirling-type bounds, $$\p_{\beta,0,p}(\os = 0) = \frac{1}{2^N}\binom{N}{\frac{N}{2}}Z_{N}(\beta,0,p)^{-1}\bm{1}\{N~\textrm{is even}\} \leq \frac{e}{\pi \sqrt{N}},$$ 
		where the last inequality uses the fact that $F_N(\beta,0,p) \geq 0$ for all $\beta \geq 0$. Hence, $\p_{\beta,0,p}(\os = 0) \rightarrow 0$ as $N \rightarrow \infty$, completing the proof of the finiteness of $\hat{h}_N$ and $\hat{\beta}_N$ for all $\beta,h,p$.

		Finally, let $p$ be even and $N \geq 1$. If the joint ML Estimate  of $(\beta,h)$ exists, then by \eqref{eqmle} and \eqref{eqmleh}, we must have
		$\e_{\beta,h,p}  \overline{\sigma}^p_N = \left(\e_{\beta,h,p}  \os\right)^p$.
		Since each of the measures $\p_{\beta,h,p} $ has support $\mathcal{M}_N$ and the function $x \mapsto x^p$ is non-affine, convex on $\mathcal{M}_N$, we arrive at a contradiction to Jensen's inequality.
	\end{proof}

	\section{Properties of the Function $H$}\label{sec:Hlemmas} 
	
This section we prove various technical results about the function $H$. We start by analyzing the global and local maxima of the sequence of functions $H_{\beta_N,h_N,p}$, where $(\beta_N,h_N)\rightarrow (\beta,h)$.

	\begin{lem}\label{derh44}
		Suppose that $(\beta_N,h_N) \in [0,\infty)\times \mathbb{R}$ is a sequence converging to a point $(\beta,h) \in [0,\infty)\times \mathbb{R}$. Then, we have the following:
		\begin{enumerate}
			\item[$(1)$] Suppose that $(\beta,h)$ is a $p$-regular point, and let $m_*$ be the global maximizer of $H_{\beta,h,p}$. Then, for any sequence $(\beta_N,h_N) \in [0,\infty)\times \mathbb{R}$ converging to $(\beta,h)$, the function $H_{\beta_N,h_N,p}$ will have unique global maximizer $m_*(N)$ for all large $N$, and $m_*(N) \rightarrow m_*$ as $N \rightarrow \infty$. 
			\item[$(2)$]Let $m$ be a local maximizer of the function $H_{\beta,h,p}$, where the point $(\beta,h)$ is not $p$-special. Suppose that $(\beta_N,h_N) \in [0,\infty)\times \mathbb{R}$ is a sequence converging to $(\beta,h)$. Then for all large $N$, the function $H_{\beta_N,h_N,p}$ will have a local maximizer $m(N)$, such that $m(N) \rightarrow m$ as $N \rightarrow \infty$. Further, if $A\subseteq [-1,1]$ is a closed interval such that $m \in \textrm{int}(A)$ and $H_{\beta,h,p}(m) > H_{\beta,h,p}(x)$ for all $x \in A\setminus \{m\}$, then there exists $N_0\geq 1$, such that for all $N \geq N_0$, we have $H_N(m(N)) > H_N(x)$ for all $x \in A \setminus \{m(N)\}$.
		\end{enumerate}
	\end{lem}
	\begin{proof}[Proof of $(1)$] The set $\cR_p$ of all $p$-regular points is an open subset of $[0,\infty)\times \mathbb{R}$. To see this, note that $\cR_p^c$ is given by $\cp \bigcup \{(\check{\beta}_p,\check{h}_p)\}$ if $p$ is odd, and by $\cp \bigcup \{(\check{\beta}_p,\check{h}_p), (\check{\beta}_p,-\check{h}_p)\}$ if $p$ is even. By \cite[Lemma B.3]{cmp}, $\cR_p^c$ is a closed set in either case. Hence, the function $H_{\beta_N,h_N,p}$ will have unique global maximizer $m_*(N)$ for all large $N$. 
		
		To show that $m_*(N) \rightarrow m_*$, let $\{N_k\}_{k \geq 1}$ be a subsequence of the natural numbers. Then, $\{N_k\}_{k \geq 1}$ will have a further subsequence $\{N_{k_\ell}\}_{\ell \geq 1}$, such that $m_*(N_{k_\ell})$ converges to some $m' \in [-1,1]$. Since $H_{\beta_{N_{k_\ell}},h_{N_{k_\ell}},p}\left(m_*(N_{k_\ell})\right) \geq H_{\beta_{N_{k_\ell}},h_{N_{k_\ell}},p}(x)$ for all $x \in [-1,1]$, by taking limit as $\ell \rightarrow \infty$ on both sides, we have $H_{\beta,h,p}(m') \geq H_{\beta,h,p}(x)$ for all $x \in [-1,1]$, showing that $m'$ is a global maximizer of $H_{\beta,h,p}$. Since $m_*$ is the unique global maximizer of $H_{\beta,h,p}$, it follows that $m' = m_*$, completing the proof of (1).  
		
		\medskip
		
		\noindent\noindent\textit{Proof of} (2):~Let us denote $H_{\beta,h,p}$ by $H$ and $H_{\beta_N,h_N,p}$ by $H_N$. It is easy to show that there exists $M \geq 1$ odd, and points $-1=a_0<a_1<\ldots<a_M=1$, such that $H'$ is strictly decreasing on $[a_{2i},a_{2i+1}]$ and strictly increasing on $[a_{2i+1},a_{2i+2}]$ for all $0\leq i \leq \frac{M-1}{2}$. Hence, the local maximizer $m$ of $H$ lies in $(a_{2i},a_{2i+1})$ for some $0\leq i \leq \frac{M-1}{2}$. Since $H'(a_{2i}) > 0$ and $H'(a_{2i+1})<0$, we also have $H_N'(a_{2i}) > 0$ and $H_N'(a_{2i+1})<0$ for all large $N$, and hence $H_N'$ has a root $m(N) \in (a_{2i},a_{2i+1})$ for all large $N$. 
		
		Let us now show that $m(N) \rightarrow m$. Towards this, let $\{N_k\}_{k\geq 1}$ be a subsequence of the natural numbers, whence there is a further subsequence $\{N_{k_\ell}\}_{\ell \geq 1}$ of $\{N_k\}_{k\geq 1}$, such that $m(N_{k_\ell}) \rightarrow m'$ for some $m' \in [a_{2i},a_{2i+1}]$. Since $H_{N_{k_\ell}}'(m(N_{k_\ell})) = 0$ for all $\ell \geq 1$, we have $H'(m') = 0$. But the strict decreasing nature of $H'$ on $[a_{2i},a_{2i+1}]$ implies that $m$ is the only root of $H'$ on this interval, and hence, $m'=m$. This shows that $m(N) \rightarrow m$. 
		
		Next, we show that $m(N)$ is a local maximizer of $H_N$ for all $N$ sufficiently large. For this, we prove something stronger than needed, because this will be useful in proving the last statement of (2). Since $H''(m) < 0$, there exists $\varepsilon > 0$ such that $[m-\varepsilon,m+\varepsilon] \subset (a_{2i},a_{2i+1})$ and $H''<0$ on $[m-\varepsilon,m+\varepsilon]$. If $m_0 \in [m-\varepsilon,m+\varepsilon]$ is such that $H''(m_0) = \sup_{x \in [m-\varepsilon,m+\varepsilon]} H''(x) < 0$, then since $H_N''$ converges to $H''$ uniformly on $(-1,1)$, $$\sup_{x\in [m-\varepsilon,m+\varepsilon]}H_N''(x) < H''(m_0)/2\quad\textrm{for all large}~N.$$ In particular, since $m(N) \in [m-\varepsilon,m+\varepsilon]$ for all large $N$, we have $H_N''(m(N)) < 0$ for all large $N$, showing that $m(N)$ is a local maximizer of $H_N$ for all large $N$. Also, since $H_N'(m(N)) = 0$ and $\sup_{x\in [m-\varepsilon,m+\varepsilon]}H_N''(x) < 0$ for all large $N$, we must have $$H_N(m(N)) > H_N(x)\quad\textrm{for all}~ x \in [m-\varepsilon,m+\varepsilon]\setminus\{m(N)\},\quad \textrm{for all large} N.$$   Finally, suppose that $A \subseteq [-1,1]$ is a closed interval such that $m \in \textrm{int}(A)$ and $H(m) > H(x)$ for all $x \in A\setminus \{m\}$. By  \cite[Lemma B.11]{cmp}, there exists $\varepsilon' >0 $ such that for all $0 < \delta \leq \varepsilon'$, $\sup_{x \in A\setminus (m-\delta,m+\delta)} H(x) = H(m\pm \delta)$. Let $\alpha = \min\{\varepsilon,\varepsilon'\}$. Then, 
		\begin{equation}\label{fs}
			H_N(m(N)) > H_N(x) \quad\textrm{for all} ~x \in [m-\alpha,m+\alpha]\setminus \{m(N)\},\quad\textrm{for all large}~ N,
		\end{equation}
		and $\sup_{x \in A\setminus (m-\alpha,m+\alpha)} H(x) = H(m\pm \alpha) < H(m)$ (since $H'(m) = 0$ and $H''< 0$ on $[m-\alpha,m+\alpha]$). Hence, 
		\begin{equation}\label{ss}
			\sup_{x \in A\setminus (m-\alpha,m+\alpha)} H_N(x) < H_N(m(N))\quad\textrm{for all large}~ N. 
		\end{equation}
		The proof of (2) now follows from \eqref{fs} and \eqref{ss}, and the proof of Lemma \ref{derh44} is now complete. 
	\end{proof}


	In the next lemma, we prove an asymptotic expansion of a local maximum value of the perturbed function $H_N$, around the corresponding local maximum value of the original function $H$. This is required in the proof of Lemma \ref{redunsup}.  
	
	\begin{lem}\label{bdhbest}
		Let $m$ be a local maximizer of $H := H_{\beta,p,h}$. Let $\beta_N := \beta + \bar{\beta}x_N$ and $h_N := h + \bar{h} y_N$ for some fixed constants $\beta,\bar{\beta},h,\bar{h}$ and sequences $x_N, y_N \rightarrow 0$. Suppose that the point $(\beta,h)$ is not $p$-special. Let $m(N)$ denote the local maximizer of $H_{\beta_N,h_N,p}$ converging to $m$. Then we have as $N \rightarrow \infty$,
		$$H_{\beta_N,h_N,p}(m(N)) = H(m) + \bar{\beta} x_N m^p + \bar{h} y_N m + O\left((x_N + y_N)^2\right).$$
	\end{lem}
	\begin{proof}
		For any sequence $({\beta}_N', {h}_N')  \rightarrow (\beta,h)$, let us denote by $m({\beta}_N', {h}_N' ,p)$ the local maximum of $H_{{\beta}_N',{h}_N',p}$ converging to $m$. In particular, $m(\beta_N,h_N,p) = m(N)$ and $m(\beta,h,p) = m$. By a simple application of Taylor's theorem and \cite[Lemma B.5]{cmp}, we have
		\begin{align}\label{ty1}
			m(N) - m & = m(\beta_N,h_N,p) - m(\beta_N,h,p) + m(\beta_N,h,p) - m(\beta,h,p)\nonumber\\
			&= -\frac{\bar{h} y_N}{H_{\beta_N,h,p}''(m(\beta_N,h,p))} -  \frac{\bar{\beta}pm^{p-1} x_N}{H''(m)} + O\left(x_N^2 + y_N^2\right)\nonumber\\
			&= O(x_N + y_N).
		\end{align}
		By another application of Taylor's theorem, we have
		\begin{align}\label{ty2}
			H_{\beta_N,h_N,p}(m(N)) - H(m) &  = H_{\beta_N,h_N,p}(m(N)) - H_{\beta_N,h_N,p}(m) + H_{\beta_N,h_N,p}(m) - H(m)\nonumber\\
			&= O\left((m(N)-m)^2\right) + \bar{\beta} x_N m^p + \bar{h}y_N m.
		\end{align}
		Lemma \ref{bdhbest} now follows from \eqref{ty1} and \eqref{ty2}.
	\end{proof}

The following lemma provides estimates of the first four derivatives of the function $H$ at the maximizer $m_*(N)$ for a perturbation of a $p$-special point. This key result is used in the proof of Lemma \ref{ex2}. 
	
	\begin{lem}\label{mnminusm}
		Let $(\beta,h)$ be a $p$-special point and $(\beta_N,h_N) := (\beta+\bar{\beta}N^{-\frac{3}{4}},h+\bar{h}N^{-\frac{3}{4}} )$ for some $\bar{\beta},\bar{h} \in \mathbb{R}$. If $m_*$ and $m_*(N)$ denote the unique global maximizers of $H := H_{\beta,h,p}$ and $H_N := H_{\beta_N,h_N,p}$ respectively, then we have the following:
		\begin{align} 
			\label{id1} N^\frac{1}{4}(m_*(N) - m_*) & = -\left(\frac{6(\bar{\beta}pm_*^{p-1} + \bar{h})}{H^{(4)}(m_*)}\right)^\frac{1}{3} + O\left(N^{-\frac{1}{4}}\right), \\
			\label{id2}	N^\frac{1}{2}H''(m_*(N)) & = \frac{1}{2}\left(6(\bar{\beta}pm_*^{p-1} + \bar{h})\right)^\frac{2}{3}\left(H^{(4)}(m_*)\right)^{\frac{1}{3}} + O\left(N^{-\frac{1}{4}}\right), \\
			\label{id3} N^\frac{1}{4}H^{(3)}(m_*(N)) & = -\left(6(\bar{\beta}pm_*^{p-1} + \bar{h})\right)^\frac{1}{3}\left(H^{(4)}(m_*)\right)^\frac{2}{3} + O\left(N^{-\frac{1}{4}}\right), \\
			\label{id4}
			H^{(4)}(m_*(N)) & = H^{(4)}(m_*) + O\left(N^{-\frac{1}{4}}\right).
		\end{align} 
	\end{lem} 
	
	\begin{proof}
		Let us start by noting that $$H'(m_*(N)) = H_N'(m_*(N)) - (\bar{\beta}pm_*(N)^{p-1} + \bar{h})N^{-\frac{3}{4}} = -(\bar{\beta} p m_*(N)^{p-1}+ \bar{h})N^{-\frac{3}{4}}.$$ On the other hand, by a Taylor expansion of $H'$ around $m_*$ and using the fact $H'(m_*) = H''(m_*) = H^{(3)}(m_*) = 0$ (see \cite[Lemma B.2]{cmp}), we have
		$$H'(m_*(N)) = \tfrac{1}{6}(m_*(N)-m_*)^3 H^{(4)}(\zeta_N),$$ where $\zeta_N$ lies between $m_*(N)$ and $m_*$. Hence, $$N^\frac{3}{4} (m_*(N)-m_*)^3 = -\frac{6(\bar{\beta}pm_*(N)^{p-1} + \bar{h})}{H^{(4)}(\zeta_N)}.$$	
		Now, it follows from the proof of Lemma \ref{derh44}, part (1), that $m_*(N) \rightarrow m_*$, and hence, $\zeta_N \rightarrow m_*$. This implies that 
		\begin{equation}\label{help1}
			\lim_{N \rightarrow \infty} N^{\tfrac{1}{4}} (m_*(N) - m_*) = -\left(\frac{6(\bar{\beta}pm_*^{p-1} + \bar{h})}{H^{(4)}(m_*)}\right)^\frac{1}{3} .
		\end{equation}
		By a $5$-term Taylor expansion of $H'(m_*(N))$ around $m_*$, one obtains
		\begin{equation}\label{addtayl}
			\tfrac{1}{6}(m_*(N)-m_*)^3 H^{(4)}(m_*) + \tfrac{1}{24}(m_*(N)-m_*)^4 H^{(5)}(\zeta_N') = -(\bar{\beta} p m_*(N)^{p-1}+ \bar{h})N^{-\frac{3}{4}}.
		\end{equation}
		for some sequence $\zeta_N'$ lying between $m_*(N)$ and $m_*$. From \eqref{addtayl} and \eqref{help1}, we have
		\begin{align}\label{addtaylor}
			N^\frac{3}{4} (m_*(N) - m_*)^3 & =  -\frac{6(\bar{\beta} p m_*(N)^{p-1}+ \bar{h})}{H^{(4)}(m_*)} - \frac{N^\frac{3}{4}(m_*(N) - m_*)^4 H^{(5)}(\zeta_N')}{4H^{(4)}(m_*)}\nonumber\\& =  -\frac{6(\bar{\beta} p m_*^{p-1}+ \bar{h})}{H^{(4)}(m_*)} + O\left(N^{-\frac{1}{4}}\right).
		\end{align}
		\eqref{id1} now follows from \eqref{addtaylor}, and \eqref{id2}, \eqref{id3}, \eqref{id4} follow by substituting \eqref{id1} into the following expansions
		\begin{equation*}
			H''(m_*(N)) = \tfrac{1}{2}\left(m_*(N) - m_*\right)^2 H^{(4)}(m_*) + O\left((m_*(N)-m_*)^3\right),
		\end{equation*}
		\begin{equation*}
			H^{(3)}(m_*(N)) = \left(m_*(N) - m_*\right)H^{(4)}(m_*) + O\left((m_*(N)-m_*)^2\right),
		\end{equation*}
		and 	$H^{(4)}(m_*(N))  = H^{(4)}(m_*) + O(m_*(N) - m_*)$. 
	\end{proof}
	
	The final lemma shows that if a function has non-vanishing curvature at a unique point of maxima, then for every sufficiently small open interval $I$ around that point of maxima, it attains its maximum on $I^c$ at either of the endpoints of $I$. This fact is used in the proofs of Lemmas \ref{conc} and \ref{lem:multiple_max}.

\end{document}